\documentclass{imsart}

\RequirePackage[OT1]{fontenc}
\usepackage{a4wide}
\usepackage[table]{xcolor}
\usepackage{multirow}
\usepackage{bm}

\usepackage[ruled,linesnumbered]{algorithm2e}

\usepackage{subcaption}{}

\usepackage[normalem]{ulem}
\usepackage{amsmath,amsthm,amssymb,mathrsfs}
\usepackage{stmaryrd}
\usepackage{nicefrac}
\usepackage{mathtools} %pour \mathclap qui ¬élimine les surplus d'espace autour de substacks (\substack ou \underset) : \underset{\mathclap{\epsilon \to 0}}{=}     \mathclap{\substack{   \\    }}
\usepackage[utf8]{inputenc}
\usepackage{amsfonts,dsfont}

\usepackage{makecell}
\usepackage{soul,graphicx}

\usepackage{natbib}
  \usepackage{pgf, tikz}
  \usetikzlibrary{fit,patterns,decorations.pathreplacing}
\usepackage[ruled]{algorithm2e}  
  
\usepackage{hyperref}

\startlocaldefs
\newtheorem{theorem}{Theorem}[section]
\newtheorem{lemma}[theorem]{Lemma}
\newtheorem{proposition}[theorem]{Proposition}
\newtheorem{corollary}[theorem]{Corollary}
\newtheorem{remark}{Remark}[section]
\newtheorem{example}{Example}[section]
\newtheorem{assumption}{Assumption}

\RequirePackage{xstring}

\let\originalleft\left
\let\originalright\right
\renewcommand{\left}{\mathopen{}\mathclose\bgroup\originalleft}
\renewcommand{\right}{\aftergroup\egroup\originalright}

\newcommand{\paren}[2][a]{%
\IfEqCase{#1}{%
{a}{\left(#2\right)}%
{0}{(#2)}%
{1}{\big(#2\big)}%
{2}{\Big(#2\Big)}%
{3}{\bigg(#2\bigg)}%
{4}{\Bigg(#2\Bigg)}%
}[\PackageError{paren}{Undefined option to paren: #1}{}]%
}
\newcommand{\norm}[2][a]{%
\IfEqCase{#1}{%
{a}{\left\lVert#2\right\rVert}%
{0}{\lVert#2\rVert}%
{1}{\big\lVert#2\big\rVert}%
{2}{\Big\lVert#2\Big\rVert}%
{3}{\bigg\lVert#2\bigg\rVert}%
{4}{\Bigg\lVert#2\Bigg\rVert}%
}[\PackageError{norm}{Undefined option to norm: #1}{}]%
}
\newcommand{\brac}[2][a]{%
\IfEqCase{#1}{%
{a}{\left[#2\right]}%
{0}{[#2]}%
{1}{\big[#2\big]}%
{2}{\Big[#2\Big]}%
{3}{\bigg[#2\bigg]}%
{4}{\Bigg[#2\Bigg]}%
}[\PackageError{brac}{Undefined option to brac: #1}{}]%
}
\newcommand{\inner}[2][a]{%
\IfEqCase{#1}{%
{a}{\left\langle#2\right\rangle}%
{0}{\langle#2\rangle}%
{1}{\big\langle#2\big\rangle}%
{2}{\Big\langle#2\Big\rangle}%
{3}{\bigg\langle#2\bigg\rangle}%
{4}{\Bigg\langle#2\Bigg\rangle}%
}[\PackageError{inner}{Undefined option to inner: #1}{}]%
}
\newcommand{\abs}[2][a]{
\IfEqCase{#1}{%
{a}{\left\vert#2\right\rvert}%
{0}{\vert#2\rvert}%
{1}{\big\vert#2\big\rvert}%
{2}{\Big\vert#2\Big\rvert}%
{3}{\bigg\vert#2\bigg\rvert}%
{4}{\Bigg\vert#2\Bigg\rvert}%
}[\PackageError{abs}{Undefined option to abs: #1}{}]%
}

\newcommand{\set}[2][a]{
\IfEqCase{#1}{%
{a}{\left\{#2\right\}}%
{0}{\{#2\}}%
{1}{\big\{#2\big\}}%
{2}{\Big\{#2\Big\}}%
{3}{\bigg\{#2\bigg\}}%
{4}{\Bigg\{#2\Bigg\}}%
}[\PackageError{set}{Undefined option to set: #1}{}]%
}

\newcommand{\pipe}[1][a]{
\IfEqCase{#1}{%
{a}{\middle|}%
{0}{|}%
{1}{\big|}%
{2}{\Big|}%
{3}{\bigg|}%
{4}{\Bigg|}%
}[\PackageError{set}{Undefined option to set: #1}{}]%
}

\newcommand{\R}{\mathbb{R}}

\newcommand{\1}{1}

\newcommand{\Cov}{\mbox{Cov}}
\renewcommand{\P}{\mathbb{P}}
\newcommand{\E}{\mathbb{E}}

\newcommand{\telque}{\,:\,}
\newcommand{\TDP}{\mathrm{TDP}}

\newcommand{\FDR}{\mathrm{FDR}}

\newcommand{\BH}{\mathrm{BH}}
\newcommand{\FDP}{\mathrm{FDP}}
\newcommand{\FCP}{\mathrm{FCP}}

\newcommand{\bp}{\mathbf{p}}
\newcommand{\mtc}{\mathcal}

\newcommand{\wt}[1]{{\widetilde{#1}}}
\newcommand{\wh}[1]{{\widehat{#1}}}

\newcommand{\cC}{{\mtc{C}}}
\newcommand{\cH}{{\mtc{H}}}

\newcommand{\cR}{{\mtc{R}}}

\newcommand{ \cadlag }{c\`adl\`ag }

\definecolor{mulberry}{rgb}{0.77, 0.29, 0.55}

\definecolor{blendedblue}{rgb}{0.2,0.2,0.7}

\usepackage{cancel} 
{} %

\newcommand{\stkout}[1]{\ifmmode\text{\sout{\ensuremath{#1}}}\else\sout{#1}\fi}

\definecolor{mygreen}{rgb}{0.82, 1.0, 0.82}
\definecolor{myred}{rgb}{ 1.0, 0.84, 0.84}

\newcommand{\Pc}[3][a]{\P\brac[#1]{#2\pipe[#1] #3}}
\newcommand{\Ec}[3][a]{\E\brac[#1]{#2\pipe[#1] #3}}

\newcommand{\dd}{\operatorname{d}}

\newcommand{\range}[1]{\llbracket #1\rrbracket}
\newcommand{\U}{\mathbb{U}}

\newcommand{\Z}{\mathbb{Z}}             %\
\newcommand{\N}{\mathbb{N}}              %\
\newcommand{\K}{\mathbb{K}}              %|
\renewcommand{\P}{\mathbb{P}} 
\newcommand{\T}{\mathcal{T}}            %/

\newcommand{\Loi}{\mathcal{L}}
\newcommand{\Nor}{\mathcal{N}}
\newcommand{\V}{\mathbb{V}}

\renewcommand{\1}[1]{{\mathds{1}_{#1}}}
\renewcommand{\geq}{\geqslant}
\renewcommand{\leq}{\leqslant}
\newcommand{\cvloi}{\overset{\Loi}{\underset{\tau_{n,m}\rightarrow +\infty}{\longrightarrow}}}

\newcommand{\W}{\mathbb{W}}
\newcommand{\dcal}{\mathcal{D}_{{\tiny \mbox{cal}}}}
\newcommand{\dtest}{\mathcal{D}_{{\tiny \mbox{test}}}}
\newcommand{\dtrain}{\mathcal{D}_{{\tiny \mbox{train}}}}

\newcommand{\Pcal}{P_{\tiny \mbox{cal}}}
\newcommand{\Fcal}{F_{\tiny \mbox{cal}}}
\newcommand{\Gncal}[1]{\wh{F}_{\tiny \mbox{cal}}^{#1}}

\newcommand{\Ptest}{P_{\tiny \mbox{test}}}
\newcommand{\Ftest}{F_{\tiny \mbox{test}}}
\newcommand{\Fmtest}{\wh{F}_{m,\tiny \mbox{test}}}

\newcommand{\Fcalcag}{F_{\tiny \mbox{cal}-}}

\newcommand{\Ftestcalcag}{{\widetilde{F}_{\tiny \mbox{test,cal}}}}
\newcommand{\Flimcag}{\widetilde{G}_{{\tiny\mbox{mixt}}}}

\newcommand{\Vwcag}{\widetilde{V^w}_{\tiny \mbox{cal}}}

\newcommand{\p}[2]{{p_{#1}^{(#2)}}}
\newcommand{\pw}[2]{{p_{#1}^{w,(#2)}}}
\newcommand{\Fm}[2]{\wh{G}_{#1}^{(#2)}}
\newcommand{\Fmw}[2]{\wh{G}_{#1}^{w,(#2)}}
\newcommand{\FCPm}{\mathrm{FCP}_m^{(n)}}
\newcommand{\FCPmw}{\mathrm{FCP}_m^{w,(n)}}
\newcommand{\FDPm}{\mathrm{FDP}_m^{(n)}}
\newcommand{\TDPm}{\mathrm{TDP}_m^{(n)}}

\newcommand{\pwo}[2]{{p_{#1}^{w^*,(#2)}}}
\newcommand{\Fmwo}[2]{\wh{G}_{#1}^{w^*,(#2)}}
\newcommand{\FCPmwo}{\mathrm{FCP}_m^{w^*,(n)}}

\newcommand{\FDPmw}{\mathrm{FDP}_m^{(n)}}
\newcommand{\TDPmw}{\mathrm{TDP}_m^{(n)}}

\newcommand{\F}[1]{\wh{{G}}_{m,{#1}}^{(n)}}
\newcommand{\Fw}[1]{\wh{{G}}_{m,{#1}}^{w,(n)}}
\newcommand{\Fwo}[1]{\wh{{G}}_{m,{#1}}^{w^*,(n)}}
\newcommand{\Flim}{G_{{\tiny\mbox{mixt}}}}
\newcommand{\Flimw}{{G_{{\tiny\mbox{mixt}}}^{w^*}}}
\newcommand{\Flimwnull}{{G_{{\tiny\mbox{mixt}}}^{w}}}
\newcommand{\Zw}{\Z^w}
\newcommand{\Zwo}{\Z^{w^*}}

\newcommand{\Wcal}{{F^w_{{\tiny\mbox{cal}}}}}
\newcommand{\Wcalcag}{{\widetilde{F^w_{{\tiny\mbox{cal}}}}}}
\newcommand{\Wcalo}{{F^{w^*}_{{\tiny\mbox{cal}}}}}
\newcommand{\rw}{{\rho^{w}}}
\newcommand{\rwo}{{\rho^{w^*}}}
\newcommand{\Ftestcal}{{F_{\tiny \mbox{test,cal}}}}
\newcommand{\Fzerocal}{{F_{\tiny \mbox{$0$,cal}}}}
\endlocaldefs

\begin{document}

\begin{frontmatter}
\title{Asymptotics for conformal inference}

\runtitle{Asymptotics for conformal inference}

\begin{aug}

%%%%%%%%%%%%%%%%%%%%%%%%%%%%%%%%%%%%%%%%%%%%%%%
\author[A]{ \fnms{Ulysse}~\snm{Gazin}\ead[label=e1]{ugazin@lpsm.paris}}
\runauthor{Gazin }
%%%%%%%%%%%%%%%%%%%%%%%%%%%%%%%%%%%%%%%%%%%%%%
%% Addresses                                %%
%%%%%%%%%%%%%%%%%%%%%%%%%%%%%%%%%%%%%%%%%%%%%%
\address[A]{Université Paris Cité and Sorbonne Université, CNRS, Laboratoire de Probabilités, Statistique et Modélisation, F-75013 Paris, France\printead[presep={,\ }]{e1}
}

\end{aug}

%\tableofcontents

\begin{abstract}
Conformal inference is a versatile tool for building prediction sets in regression or classification.  
We 
study the false coverage proportion (FCP) in a simultaneous inference setting with a calibration sample of $n$ points and a test sample of $m$ points. We identify the exact, distribution-free, asymptotic distribution of the FCP when both $n$ and $m$ tend to infinity. This shows in particular that FCP control can be achieved by using the well-known Kolmogorov distribution, and puts forward that the asymptotic variance is decreasing in the ratio $n/m$. We then provide a number of extensions by considering the problems of novelty detection, weighted conformal inference or distribution shift between the calibration sample and the test sample. In particular, our asymptotic results allow to accurately quantify the asymptotic behavior of the errors (a miscovering interval or declaring a false novelty) when weighted conformal inference is used.
\end{abstract}
\begin{keyword}
\kwd{Conformal prediction}
\kwd{empirical cumulative distribution function}
\kwd{functional central limit theorem}
\kwd{functional delta method}
\kwd{multiple testing}
\kwd{weighted conformal prediction}
\end{keyword}

\end{frontmatter}

%%%%%%%%%%%%%%%%%%%%%%%%%%%%%%%%%%%%%%%%%%%%%%
%%%% Main text entry area:

\section{Introduction}

\subsection{Background}

In classical statistics, producing prediction sets for outcomes often relies on strong model assumptions. Recent advances involve complex data sets and sophisticated machine learning methods, for which such an approach is not appropriate. One recent solution is given by conformal prediction \citep{saunders1999transduction,vovk2005algorithmic,angelopoulos2021gentle} which consists in calibrating the prediction set according to an appropriate quantile of a calibration/training sample.
Strikingly, this method provides a finite-sample valid coverage (i.e. valid for any size $n\geq 1$ of the calibration sample) for any underlying distribution of the data and for any underlying point-prediction machine learning algorithm. Similar techniques can be employed for the novelty detection task \citep{balasubramanian2014conformal,bates2023testing,marandon2022machine}. 

\subsection{Aim and contributions}

We consider here the
multiple
setting \citep{vovk2013transductive}, where it is given 
a training dataset (which is considered as fixed here), a calibration sample of $n$ points $(X_{1},Y_{1})$, $\dots$, $(X_{n},Y_{n})$ and a test sample of $m$ points $(X_{n+1},Y_{n+1}),$ $\dots$, $(X_{n+m},Y_{n+m})$. While the calibration sample is fully observed, the $Y_i$'s of the test points are not observed and a prediction set  should be provided for each of them. The false coverage proportion (FCP) for the $m$ conformal prediction sets $\mathcal{C}^{\alpha}(X_{n+1}), \dots, \mathcal{C}^{\alpha}(X_{n+m})$ (see below for a formal definition) is defined as the proportion of coverage errors among the test sample:
$$
 \FCP_m^{(n)}(\alpha)\coloneqq\frac{1}{m}\sum_{i\in\range{m}}\1{Y_{n+i}\notin\mathcal{C}^\alpha(X_{n+i})}.
$$
Under standard assumptions 
and when the conformal prediction sets come from the inductive/split conformal procedure the distribution of the process $\FCP_m^{(n)}$ has been shown to be distribution-free, in the sense that it does only depend on $n$ and $m$ \citep{f2023universal,huang2024uncertainty,gazin2023transductive}. 
Due to the dependence between the individual coverage errors, combinatorial formulas derived in those references are particulary complex in $m$ and $n$.
Nevertheless, focusing on the maximum absolute deviation 
\begin{equation}\label{equ-diffinfty}
\| \FCP_m^{(n)}- I_n\|_\infty := \sup_{\alpha\in [0,1]} |\FCP_m^{(n)}(\alpha)- I_n(\alpha)|,
\end{equation}
with $I_n(\alpha)=\lfloor (n+1)\alpha\rfloor /(n+1)$, a DKW-type concentration inequality has been derived in \cite{gazin2023transductive}, which is both simple and finite-sample valid. It explicitly involves a rate $\tau_{n,m}^{1/2}$ with
\begin{equation}\label{eq:taunm}
\tau_{n,m}\coloneqq\frac{nm}{n+m}\in\brac{\frac{n\wedge m}{2},n\wedge m}.
\end{equation} 
 However, this DKW inequality is conservative in general (see Figure~\ref{fig:intro} below), which makes the corresponding FCP control stringent.

The aim of this paper is to complement the above studies by analyzing  $\FCP_m^{(n)}$ from an asymptotic point of view, where {\it both}  $m$ and $n$ tend to infinity 
which is equivalent to the convergence of $\tau_{n,m}$ to infinity. This asymptotic study will also cover the case of a distribution shift i.e. when the exchangeability assumption does not hold thus  extending the analysis of \cite{gazin2023transductive} and more broadlythe usual framework of conformal inference. We also study, in the same asymptotic regime, the asymptotic of the Benjamini-Hochberg procedure when conformal $p$-values are used \citep{bates2023testing}, which  complement the non-asymptotic work of \cite{bates2023testing,gazin2023transductive} among others. We refer to Section~\ref{sec:relationtopreviouswork} for more details on the relation of this work with previous studies. 
More precisely, our contributions are the following ones:
\begin{enumerate}
\item We show that under standard assumptions $\FCP_m^{(n)}$ converges uniformly to the nominal value at rate $\tau_{n,m}^{1/2}$ and that the asymptotic covariance process is a standard Brownian bridge  (Theorem~\ref{thr:BB}). Compared to the ``oracle''  case where $n=\infty$, it means that the variance is inflated by a factor asymptotically  equivalent to $(n+m)/n$, for instance $2$ in the case where $n\sim m$. 
\item A direct corollary of this result is that $\tau_{n,m}^{1/2}\| \FCP_m^{(n)}- I_n\|_\infty$ converges to the well known Kolmogorov distribution, that is, the distribution with  c.d.f. $x\mapsto (1-2\sum_{k\geq 1} (-1)^k e^{-2 k^2x^2})\1{x\geq 0}$. A comparison between the quantiles of this distribution, those given by empirical simulations of $\FCPm$ or those from the Conformal-DKW is provided in Figure~\ref{fig:intro}. While its validity is only asymptotic with $n,m\to +\infty$, it appears that the Kolmogorov quantile is simple and more accurate than the quantile obtained from DKW. Hence, our limit theorem allows to get simple, accurate and asymptotically-valid confidence bounds for the FCP process as well as asymptotic approximations of all quantities related to the distribution of the $\FCP$ process.
\item We then extend this result to the case where the distribution of the calibration sample is not equal to the test sample, that is, under a {\it distribution shift} (Theorem~\ref{thr:cvAlter}). As expected, the convergence of the FCP is not towards the nominal level in this case but rather towards a new term $G$ that takes this shift into account. 
The asymptotic covariance process is also modified according to $G$ and an explicit formula is given. The role of the ratio $n/(n+m)$ (number of calibration points divided by the total number of points)  clearly appears in the limit and indicates  the existence of three asymptotic regimes: under, proportional and over calibrated, see Table~\ref{tab:3Regime}.
\item To recover the appropriate nominal level in the limit, we adopt the weighted conformal approach \citep{tibshirani2019conformal,barber2023conformal} with specific weights that rely on the data distribution, that we refer to as {\it oracle weights}. The obtaned central limit theorem shares similarities with the exchangeable case described above, with the essential difference that the asymptotic covariance process is not distribution free, and depends on the sample distributions (Theorem~\ref{thr:BBWeight}).
\item We also obtain a convergence result in case of non-oracle weights (Theorem~\ref{thr:cvAlterWeight}), which is crucial to quantify the FCP asymptotic behavior in the difficult but realistic case where the user has not access to the true distribution shift. An illustration is displayed in Figure~\ref{fig:IlluNonOracle}, where the asymptotic confidence interval for the FCP is given as a function of an error $\Delta$, measuring how the used weights deviate from the oracle ones. This puts forward that the FCP gets significantly away from $\alpha$ when $\Delta$ is above $\approx 0.133$ or below $\approx -0.11$ in this framework. 
\item Finally, we obtain similar results for the novelty detection task, by studying the asymptotic behavior of the false discovery proportion (FDP) of classical procedures \citep{bates2023testing,jin2023model} (Theorems~\ref{thr:BH95} and \ref{thr:BH95Weight}).  
\end{enumerate}
The proofs are based on specific decompositions of the processes that can be found in Section~\ref{sec:empiricalProcess}, while further details are postponed to appendices, notably Section~\ref{sec:WeakCPlimit} with a presentation of the main results of the paper under relaxed assumptions.

\begin{figure}[h!]
\center
\includegraphics[scale=0.28]{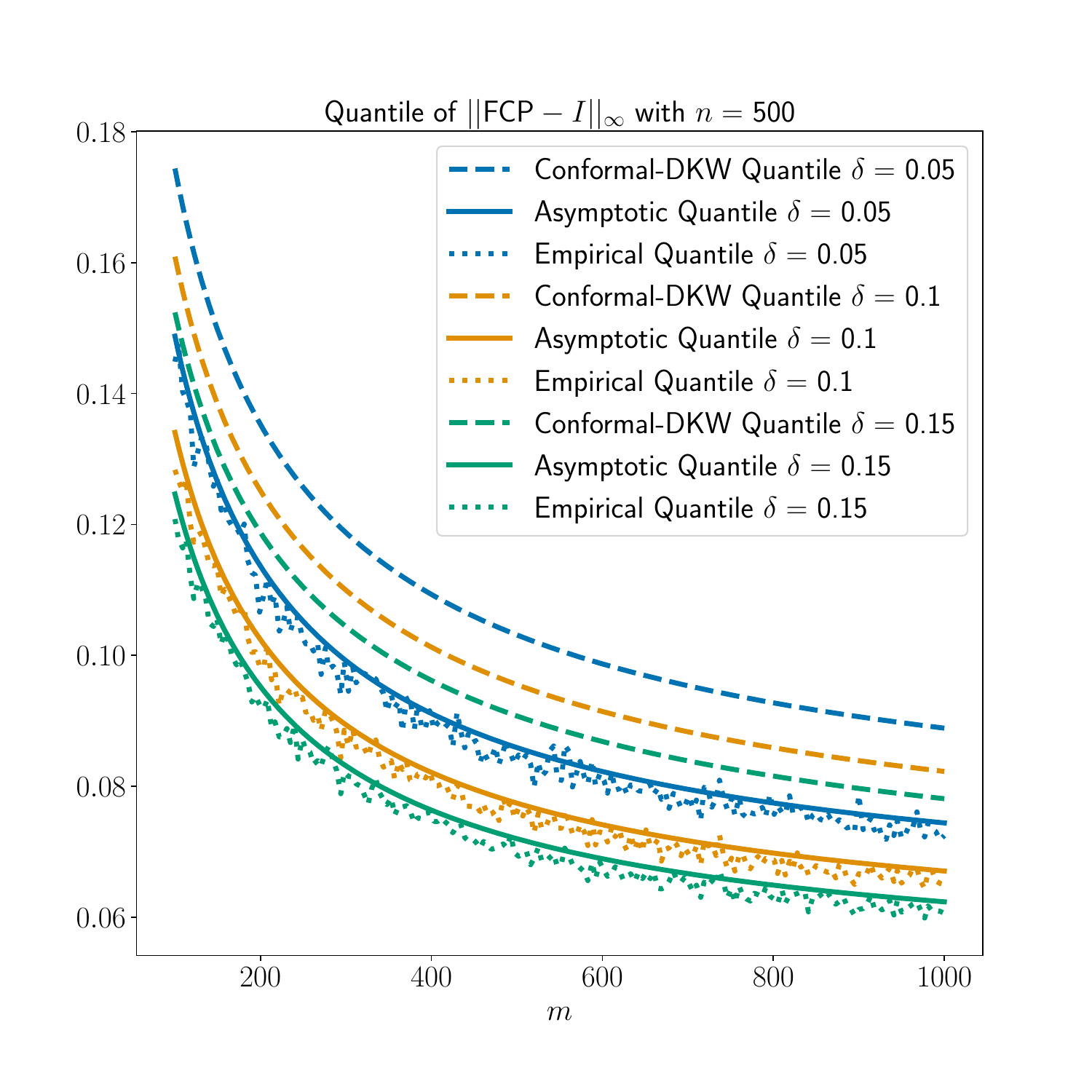}\includegraphics[scale=0.28]{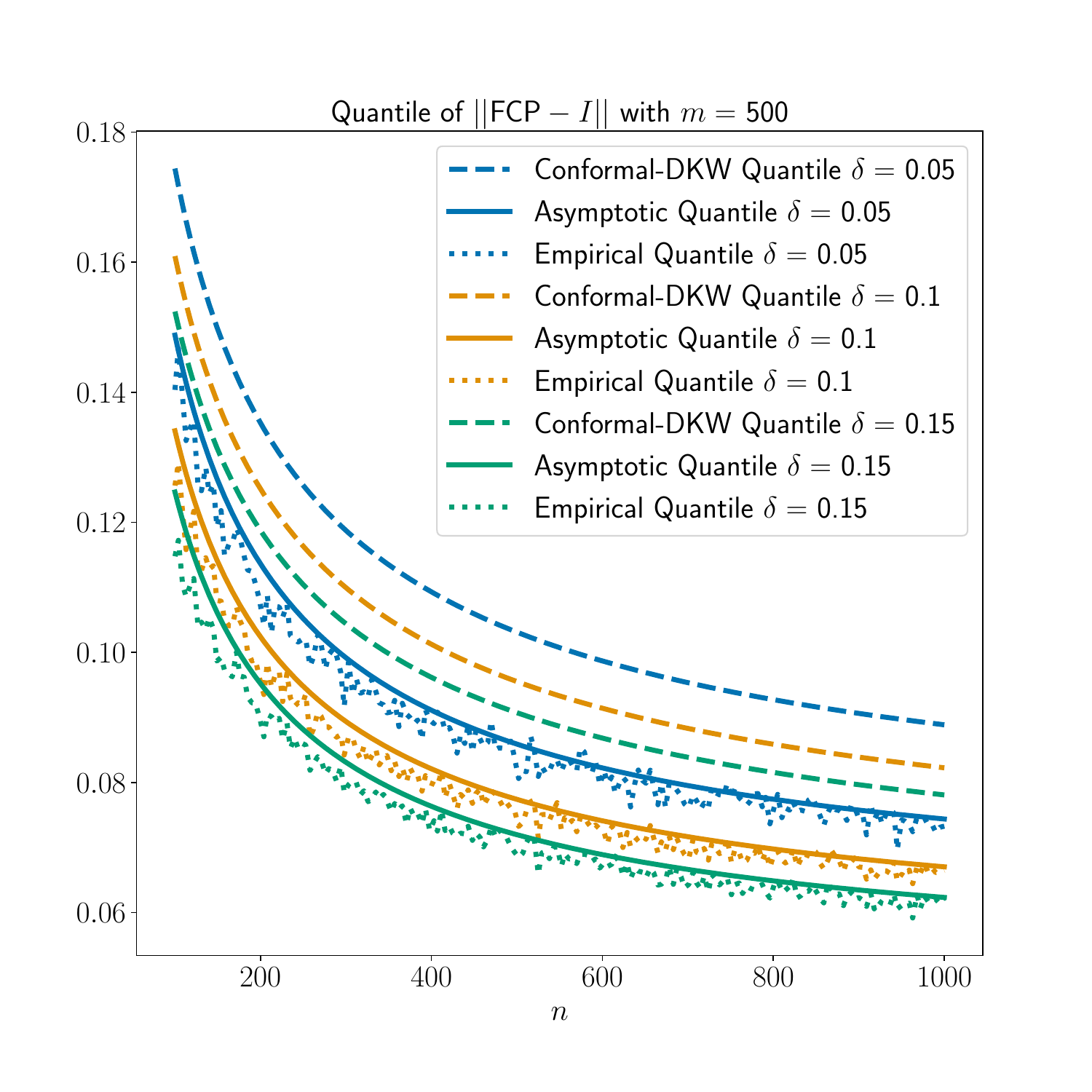}
\caption{
\label{fig:intro}
Comparison of the $(1-\delta)$-quantile of different approximations of the distribution of 
$\| \FCP_m^{(n)}- I_n\|_\infty$ in \eqref{equ-diffinfty} for different values of $n,m$ and $\delta$. The approximations include Monte-Carlo ($1000$ replications), Conformal-DKW \citep{gazin2023transductive} and the asymptotic exact Kolmogorov one.}
\end{figure}
\begin{figure}[h!]
\center
\includegraphics[scale=0.4]{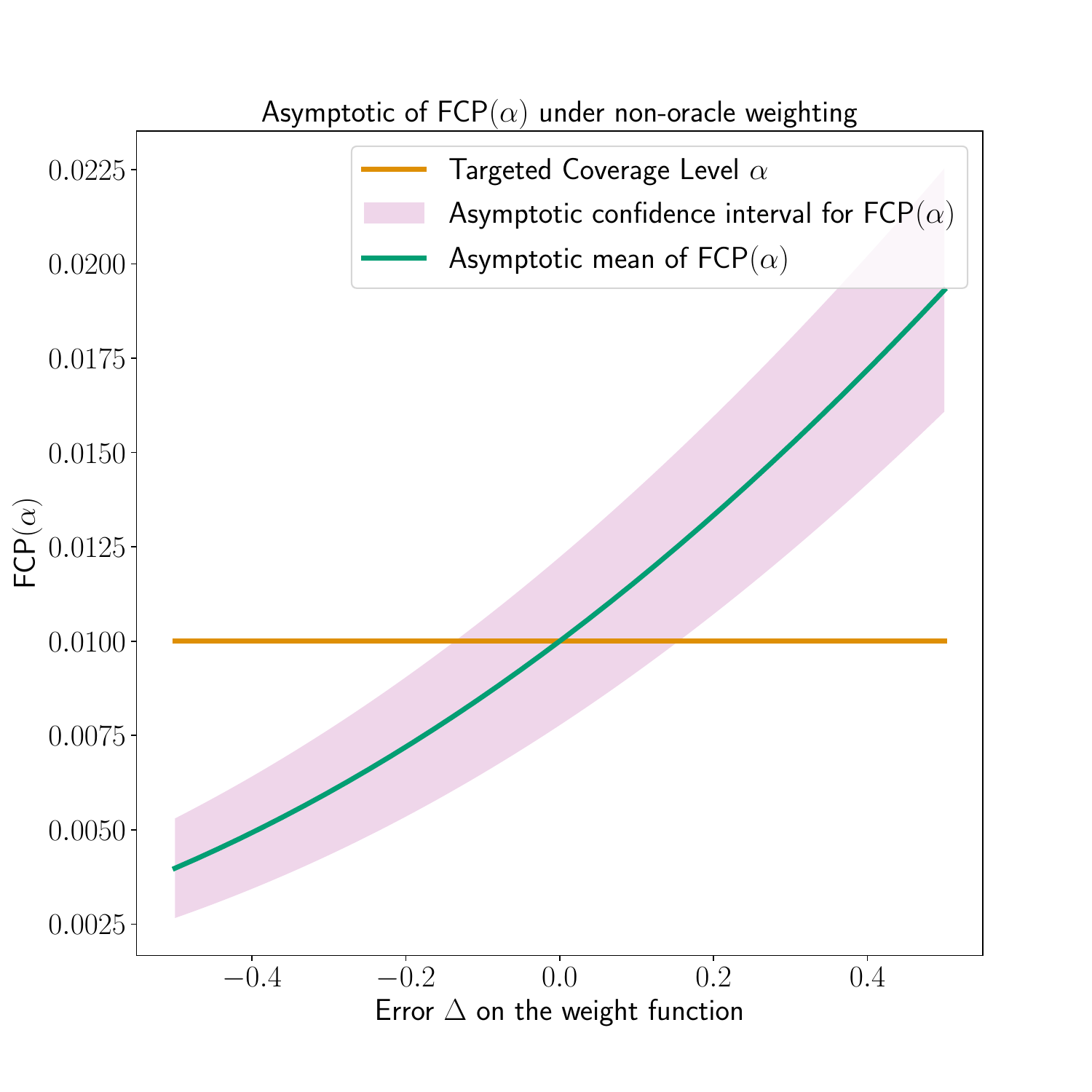}
\caption{Plot of the asymptotic confidence interval for the $\FCP$ of the weighted conformal method (at level $80\%$) obtained in Theorem~\ref{thr:cvAlterWeight} versus a parameter estimation error $\Delta$. The calibration sample and the test sample are distributed according to the exponential distribution with mean $1$ and $1/3$, respectively. The weight function used in the conformal method is $w_\Delta(x)=\exp(-(2+\Delta)x)\1{x>0}$, which corresponds to the oracle choice if $\Delta=0$ (for which the asymptotic average of $\FCP(\alpha)$ is equal to $\alpha$) and deviates from it if $\Delta\neq 0$. 
\label{fig:IlluNonOracle}
  }
\end{figure}
\subsection{Relation to previous work}\label{sec:relationtopreviouswork}

Conformal prediction is a general pipeline and we refer the reader to \cite{vovk2005algorithmic} or \cite{angelopoulos2021gentle} for reviews. We focus here on the inductive/split conformal inference \citep{papadopoulos2002inductive}, where an independent training sample is used to build the predictors, while a calibration sample is used to adjust the prediction sets. All our results can be thought of holding conditionally on the training sample, that is, they hold once the point-predictions have been computed. 

The process $\FCP_m^{(n)}$ coincides with the empirical cumulative distribution function (e.c.d.f.) of conformal $p$-values as introduced by \cite{saunders1999transduction}. As recalled above, for a given nominal level $\alpha$, the non-asymptotic distribution of $\FCP_m^{(n)}(\alpha)$ has been obtained in \cite{f2023universal,huang2024uncertainty} and the full distribution of the process $\FCP_m^{(n)}$ has been given in \cite{gazin2023transductive}. The main difference with the current study is that they provide non-asymptotic properties   under the exchangeable setting while we provide asymptotic properties under the exchangeable and the non-exchangeable frameworks.
In addition, \cite{f2023universal} gives an asymptotic result for $\FCP_m^{(n)}(\alpha)$ with $\alpha\in(0,1)$ fixed in the regime where $n$ tends to infinity only after making $m$ tend to infinity, whereas here we make $n$ and $m$ grow to infinity independently
leading to a double asymptotic. Furthermore all our results are uniform in $\alpha$.  

Next, the asymptotic distribution of $\FCPm$ with $n$ fixed and $m\rightarrow+\infty$ is also studied with a PAC point of view since this limit is equal to the miscoverage probability given the calibration sample $\Pc[0]{Y_{n+1}\notin\mathcal{C}^\alpha(X_n+1)}{(X_i,Y_i)_{i\in\range{n}}}$: \cite{vovk2012conditional} derives  Hoeffding-like tail bounds for a fixed $\alpha$, and other bounds was derived by \cite{bian2022training,sarkar2023post} the latter being uniform in $\alpha$. \cite{angelopoulos2021gentle} states that $\Pc[0]{Y_{n+1}\notin\mathcal{C}^\alpha(X_n+1)}{(X_i,Y_i)_{i\in\range{n}}}$ follows a Beta distribution, which is exactly the limit identified by \cite{f2023universal}. Similarly, \cite{nguyen2024data} obtained asymptotic results for the error as $m$ tends to infinity while $n$ is kept fixed but with more general risk function and not necessary the $0-1$ loss function. In the novelty detection setting,  \cite{bates2023testing} also provided  bounds for the type I error given the calibration sample uniformly in  $\alpha$.

Finally, we study here novelty detection procedures obtained by applying the Benjamini-Hochberg procedure ($\BH$, \citealp{BH1995}) to the conformal $p$-values \citep{mary2021semisupervised,bates2023testing,marandon2022machine} or to the weighted conformal $p$-values \citep{tibshirani2019conformal,barber2023conformal,jin2023model}.  Asymptotics for the FDP of such procedures have been extensively studied in the literature, see \cite{GW2002,GW2004,Neu2008} for i.i.d. uniform $p$-values and \cite{DR2011,DR2016,kluger2024central} for several types of dependence structures. 
The present work is in this line of research by seeing (weighted) conformal $p$-values as a particular case of dependent $p$-values having a specific dependency structure induced by the common calibration sample. In particular, our work shows that the ratio $n/(n+m)$ measures the dependency strength in the double asymptotic regime where both $n$ (calibration sample size) and $m$ (the number of test) tend to infinity.
In the novelty detection setting, we derive our results by using the techniques introduced in \cite{Neu2008}, that combine functional central limit theorems with the functional delta method. In particular, we provide the full asymptotic FDP distribution for the procedure of \cite{jin2023model} (or a variant thereof), for which only in-expectation results were known.

\section{Preliminaries}

\subsection{Prediction setting}\label{sec:predictionSet}

We consider the classical inductive/split conformal prediction \citep{papadopoulos2002inductive}. 
A calibration set $\{Z_k,k\in\range{n}\}$ is observed with $Z_k=(X_k,Y_k)$ and, given a new point $Z_{n+1}=(X_{n+1},Y_{n+1})$ for which only the covariate $X_{n+1}\in \mathcal{X}$ is observed, a prediction set should be inferred for the outcome $Y_{n+1}\in  \mathcal{Y}$.\footnote{The regression and classification settings correspond to $\mathcal{Y}=\R$ and $\mathcal{Y}$ finite, respectively.}
A training dataset $\dtrain$ \emph{independent of the test point and the calibration sample} is observed and used to compute a non-conformity score function $S:\mathcal{X}\times\mathcal{Y}\rightarrow\R$. The scalar $S(X,Y)\in\R$ measures the non-conformity of the response $Y\in\mathcal{Y}$ with the covariate $X\in\mathcal{X}$.  All our results hold conditionally on the training dataset, hence to alleviate notation we will consider the score function $S$ as fixed and determinist. 
This is in accordance with the usual split conformal model: the training/calibration split is considered as already done, with a fixed  independent training sample and a growing calibration sample.
A classical example in regression is the absolute value of the residual $S(X,Y)=|Y-\wh{\mu}(X)|$ where $\wh{\mu}:\mathcal{X}\rightarrow\mathcal{Y}$ is a regression function trained from an independent training sample (considered as fixed here). 
 The (split) conformal prediction set at level $(1-\alpha)$ for $X_{n+1}$, denoted by $\mathcal{C}^{\alpha}(X_{n+1})$, is defined as 
\begin{equation}\label{eq:PredictionSet}
\mathcal{C}^{\alpha}(X_{n+1})\coloneqq\set{y\in\mathcal{Y}:\ S(X_{n+1},y)\leq S_{\paren{\lceil(n+1)(1-\alpha)\rceil}}},
\end{equation}
where $S_{(1)}\leq S_{(2)}\leq\cdots\leq S_{(n)}< S_{(n+1)}:=+\infty$ correspond to the ordered calibration scores $\dcal\coloneqq\set{S_k:{k\in\range{n}}}=\set{S(X_k,Y_k)\telque k\in\range{n}}$. 
The interest of the conformal prediction set is the validity of the prediction region under weak assumption. More precisely, if the random variables $(Z_1,\cdots,Z_n,Z_{n+1})$ (i.e. the $n$ calibration points and the new test point) are exchangeable (or i.i.d.) then  {we have the marginal coverage property:} $\P(Y_{n+1}\notin\mathcal{C}^\alpha(X_{n+1}))\leq \alpha$. {Furthermore, } if the scores $(S(X_i,Y_i))_{i\in\range{n+1}}$ are exchangeable and without ties almost surely, then $\P(Y_{n+1}\notin\mathcal{C}^\alpha(X_{n+1}))$ is lower bounded by $\alpha-(n+1)^{-1}$ (see \citealp{vovk2005algorithmic} for more details). 
The set $\mathcal{C}^{\alpha}(X_{n+1})$ can be equivalently described by using  $p$-values (this classical fact can be retrieved from Lemma~\ref{lemma:QuantileWeight}). In particular, by defining the (unobserved) \emph{conformal $p$-value}  $\p{1}{n}$ as in   \eqref{eq:IntroConfpvalues} with the test score $T_1=S(X_{n+1},Y_{n+1})$  (see below), we obtain that $Y_{n+1}\notin\mathcal{C}^{\alpha}(X_{n+1})$ if and only if $\p{1}{n}\leq \alpha$. By the above marginal coverage property, $\p{1}{n}$ is a valid $p$-value for testing exchangeability of $(Z_1,\cdots,Z_n,Z_{n+1})$ since it is super-uniformly distributed under exchangeability. 
In this paper, we consider a 
multiple/simultaneous inference setting (also called transductive by \citealp{vovk2013transductive,gazin2023transductive,behboodi2025fundamental}), where a decision should be made for a whole test sample of size $m$ $(Z_{n+1},\dots,Z_{n+m})$ with $Z_{n+i}=(X_{n+i},Y_{n+i})$, for which only the covariates $X_{n+1},\dots,X_{n+m}$ are  observed. We denote $\dtest\coloneqq\set{T_i\telque i\in\range{m}}$ with $T_i=S(X_{n+i},Y_{n+i})$ the set of the unobserved (since only the $X_{n+i}$'s are observed) scores of the test sample. 
Considering the conformal prediction sets $\mathcal{C}^\alpha(X_{n+i})$, $i\in\range{m}$ --- corresponding to \eqref{eq:PredictionSet} for any test point --- gives rise to a family of $m$ conformal $p$-values $(\p{i}{n})_{i\in\range{m}}$ defined by
\begin{equation}\label{eq:IntroConfpvalues}
\p{i}{n}\coloneqq(n+1)^{-1}\paren{1+\sum_{k\in\range{n}} \1{S_k\geq T_i}}, \:\:\:  i\in\range{m}.
\end{equation} 
We also define the family of {\it theoretical $p$-values} $(\p{i}{+\infty})_{i\in\range{m}}$ when the calibration scores are i.i.d. as
\begin{equation}\label{eq:theopvalue}
\p{i}{+\infty}\coloneqq \Pc{S_1\geq T_i}{T_i}, \:\:\:  i\in\range{m}.
\end{equation}
By applying the law of large number when the calibration scores are i.i.d., the $p$-value $\p{i}{+\infty}$ can be seen as the limit of $\p{i}{n}$ when $n$ tends to infinity.
The false coverage proportion $\FCP_m^{(n)}(\alpha)$ of the prediction set family $\paren{\mathcal{C}^\alpha(X_{n+i})}_{i\in\range{m}}$ is defined by
\begin{equation}\label{eq:FCP}
\FCP_m^{(n)}(\alpha)\coloneqq\frac{1}{m}\sum_{i\in\range{m}}\1{Y_{n+i}\notin\mathcal{C}^\alpha(X_{n+i})}=\frac{1}{m}\sum_{i\in\range{m}}\1{\p{i}{n}\leq \alpha}.
\end{equation}
Note that this corresponds to the empirical cumulative distribution function (e.c.d.f.) of the conformal $p$-values family $(\p{i}{n})_{i\in\range{m}}$. 
We emphasise that the set $\dtest$ of test scores is not observed, and neither are the conformal $p$-values \eqref{eq:IntroConfpvalues} in the prediction setting. The $\FCP$ is not observed, and the aim is to study its behavior when the sizes of the calibration and test samples both tend to infinity.
The following assumption will be considered throughout the paper for the (weighted or not) conformal prediction task.
\begin{assumption}
\label{as:CP}
The set of calibration scores $\paren{S_k}_{k\geq 1}$ and the set of test scores  $\paren{T_i}_{i\geq 1}$ are two independent families of real random variables. The $S_k$'s ($k\geq 1$), (resp. $T_i$'s ($i\geq 1$)) are i.i.d. with distribution $\Pcal$ (resp. $\Ptest$) and cumulative distribution function  $\Fcal$ (resp. $\Ftest$).
Moreover, $\Fcal$ and $\Ftest$ are continuous functions.
\end{assumption}

When the variables $Z_i, i\in \range{n+m}$, are  i.i.d., Assumption~\ref{as:CP} is true with $\Pcal=\Ptest$ conditionally on $\dtrain$, see Remark~\ref{rem:Independence} below for more details.
Since the vector of scores $\dcal\cup\dtest$  is exchangeable in this case, 
  the conformal $p$-values are marginally super-uniform \citep{vovk2005algorithmic,RW2005} which leads to non-asymptotically valid prediction sets. However, in case of a distribution shift between the distributions of the calibration and the test sample, that is $\Pcal\neq \Ptest$, this property is lost. To solve this issue, \cite{tibshirani2019conformal,barber2023conformal} proposed in this case to use weighted conformal prediction. We follow this approach (with light formal variations for mathematical convenience) by introducing a nonnegative  {\it weight function} $w:\R\rightarrow\R^+$, $w(+\infty)\in \R^+$ and the prediction set
\begin{equation*}
\mathcal{C}^{w,\alpha}(X_{n+1})\coloneqq\set{y\in\mathcal{Y}:\ S(X_{n+1},y)\leq Q_{1-\alpha}\paren{{\frac{w(+\infty)}{W}\delta_{+\infty}+\sum_{k\in\range{n}}\frac{w\paren{S_k}}{W}\delta_{S_k}}}},
\end{equation*}
where $Q_{1-\alpha}(\mu)$ denotes the $1-\alpha$-quantile of the probability measure $\mu$, $\delta_a$ denotes the Dirac measure in $a$ and $W:=w(+\infty)+\sum_{k\in\range{n}} w\paren{S_k}$ is a normalization constant. 
There are different ways to choose the weight function $w$. 
In the case of a known distribution shift in some specific model, the weight function is usually chosen as the Radon-Nikodym derivative of the shift \citep{tibshirani2019conformal,jin2023model}. \cite{barber2023conformal} assumes that the weight function $w$ is deterministic and fixed by the user before looking at the data. To accommodate both points of view, we assume that $w$ can be estimated from a (possibly infinite) independent training dataset (potentially different from $\dtrain$ above) containing both data with the same distribution as the calibration points, and data with the same distribution as the test points. Our results hold conditionally on this training sample and $w$ can be considered as deterministic.
Thanks to Lemma~\ref{lemma:QuantileWeight}, the set $\mathcal{C}^{w,\alpha}(X_{n+1})$ can 
 be described with {\it weighted} conformal $p$-values: $Y_{n+1}\notin\mathcal{C}^{w,\alpha}(X_{n+1})$ if and only if $\pw{1}{n}\leq\alpha$, with $\pw{1}{n}$ the (unobserved) \emph{weighted conformal $p$-value} given by \eqref{eq:IntroConfpvaluesWeight} below.
Similarly, with a test sample of size $m$, we obtained $m$ {\it{weighted conformal $p$-values}},
\begin{equation}\label{eq:IntroConfpvaluesWeight}
\pw{i}{n}\coloneqq \frac{w(+\infty)+\sum_{k=1}^{n}w(S_k)\1{S_k\geq T_i}}{w(+\infty)+\sum_{k=1}^{n}w(S_k)}, \text{ for all $i\in\range{m}$}.
\end{equation}
The $\FCP$ of these weighted conformal prediction sets is then defined by
\begin{equation}\label{eq:FCPWeight}
\FCPmw(\alpha)\coloneqq\frac{1}{m}\sum_{i\in\range{m}}\1{Y_{n+i}\notin\mathcal{C}^{w,\alpha}(X_{n+i})}=\frac{1}{m}\sum_{i\in\range{m}}\1{\pw{i}{n}\leq \alpha},
\end{equation}
which, similarly to \eqref{eq:FCP}, is the e.c.d.f. of the weighted conformal $p$-values family. 
As for the unweighted case, the weighted conformal $p$-values are not observed, and neither is the corresponding false coverage proportion.

Finally, under Assumption~\ref{as:CP} and if $\Ptest$ is absolutely continuous with respect to $\Pcal$,  a particularly interesting weight function, called the {\it oracle weight function}, is given by 
\begin{equation}\label{eq:oracleWeight}
w^*\coloneqq \frac{\dd \Ptest}{\dd \Pcal}.
\end{equation}
 Clearly, if the calibration and test sample are exchangeable ($\Pcal=\Ptest$), the function $w^*$ is  constantly  equal to $1$. However, under a distribution shift ($\Pcal\neq\Ptest$), the oracle weight function is different, and oracle-weighted conformal $p$-values $\pwo{i}{n}$ \eqref{eq:IntroConfpvaluesWeight} are different from the regular ones $\p{i}{n}$ \eqref{eq:IntroConfpvalues}. They satisfy the marginally super-uniform property provided that $w^*(+\infty)\geq \sup_{u\in\R} w^*(u)$, see \cite{tibshirani2019conformal}.
 
\begin{remark}\label{rem:Independence}
In split conformal inference the score function $S(\cdot)$ and the weight function $w(\cdot)$ are constructed here with an independent training dataset which creates a dependence among the scores $(S(Z_i))_{i\in\range{n+m}}$. 
 However, our results are stated conditionally on the training phase hence the weight function  $w(\cdot)$ and score function $S(\cdot)$ can be considered as deterministic and the scores $(S(Z_i))_{i\in\range{n+m}}$ are independent as stated in Assumption~\ref{as:CP} provided that the data $(Z_i)_{i\in\range{n+m}}$ are independent (which is a common assumption in the conformal prediction literature).
\end{remark}

\subsection{Novelty detection setting}\label{sec:noveltyDetect}
In the novelty detection setting (see for instance \citealp{vovk2005algorithmic,bates2023testing}), we observe a calibration set $\{Z_k,k\in\range{n}\}$ of size $n$, with values in $\mathcal{Z}$ distributed according to an unknown ``null'' distribution $\P_0$ and a test sample $\{Z_{n+i},i\in\range{m}\}$ with $Z_{n+i}\in\mathcal{Z}$ either distributed as $\P_0$ or not. Formally, we introduce a subset $\cH_0\subset \{1,2,\dots\}$  so that $Z_{n+i}\sim \P_0$ when $i\in \cH_0$. We also denote $\cH_1=\{1,2,\dots\}\backslash \cH_0$.
We assume that we have access to an independent training dataset $\dtrain$ used to learn a non-conformity score function $S:\mathcal{Z}\rightarrow \R$ such that $S(Z)\in\R$ measures the non-conformity of the variables $Z$ with respect to $\P_0$. Our results are stated conditionally on $\dtrain$ and the training phase. Therefore, as in the prediction setting,  we assume that the score function $S(\cdot)$ is fixed.
We denote $\dcal=\set{S_k\telque k\in\range{n}}$ for $S_k\coloneqq S(Z_k)$, and $\dtest=\set{T_i\telque i\in\range{m}}$ for $T_i:=S(Z_{n+i})$ the set of the scores from the calibration and test samples respectively.
A novelty detection procedure decides, for each $i\in\range{m}$, whether $Z_{n+i}$ is a novelty (that is, does not follow $\P_0$), or not.
The procedure described by \cite{bates2023testing} 
consists in computing the conformal $p$-values defined in \eqref{eq:IntroConfpvalues} (by using the specific $S_k$'s and $T_i$'s of novelty detection which are observed in this case),
and then applying the Benjamini-Hochberg procedure \citep{BH1995} 
on this $p$-value family $(\p{i}{n})_{i\in\range{m}}$  (see Section~\ref{sec:FDPasProcess} for more formal details). This gives a rejection set $\mathcal{R}_\alpha \subset \range{m}$ corresponding to the indices of the declared novelties. We introduce the two main quantities to measure the performance of the rejection set,
\begin{align}
\FDP_m^{(n)}(\mathcal{R}_\alpha)&\coloneqq\frac{\abs{\mathcal{R}_\alpha\cap \cH_0}}{\abs{\mathcal{R}_\alpha}\vee 1};\label{eq:FDP}\\
\TDP_m^{(n)}(\mathcal{R}_\alpha)&\coloneqq\frac{\abs{\mathcal{R}_\alpha\cap \cH_1}}{\abs{\cH_1}\vee 1}.\label{eq:TDP}
\end{align} 
The $\FDP$ corresponds to the proportion of errors among the declared novelties (related to a type I error notion), while the $\TDP$ corresponds to the proportion of correct decisions among the true novelties (related to power and type II error notions).
In the novelty detection setting all the scores are available hence the conformal $p$-values can be computed. 
Nevertheless, the $\FDP$ and the $\TDP$ remain unknown since the indices of the nulls/alternatives are not observed. The following is assumed throughout the paper for (weighted) novelty detection.
\begin{assumption}
\label{as:ND}
The set of calibration scores $\paren{S_k}_{k\geq 1}$ and the set of test scores $\paren{T_i}_{i\geq 1}$ 
are two independent families of real random variables. The variables $S_k$, $k\geq 1$,  are i.i.d. with distribution $\Pcal$ and c.d.f. $\Fcal$.  The variables $T_i$, $i\geq 1$, are independent,  the variables $T_i$, $i\in\cH_0$, are identically distributed as a null score distribution $P_0$ and c.d.f. $F_0$, and the variables $T_i$, $i\in\cH_1$, are identically distributed as an alternative score distribution $\Ptest$ (potentially different from $P_0$) with c.d.f. $\Ftest$. Moreover, $\Fcal$, $F_0$ and $\Ftest$ are continuous. 
\end{assumption}

We also consider the case of a distribution shift between the distribution of the calibration set $\Pcal$ and the null distribution $P_0$, in which case we propose to use a weighted $p$-value approach (as for the prediction task). The weights are assumed to be built with an independent training sample so they are considered as fixed here (our results hold conditionally on this sample) considering oracle weights as a particular case with perfect training. 
For some {\it weight function}  $w:\R\rightarrow\R^+$ and $w(+\infty)\in \R^+$, the weighted conformal $p$-values are the ones from \eqref{eq:IntroConfpvaluesWeight}  (by using the specific $S_k$'s and $T_i$'s of novelty detection).
We define the {weighted version of the procedure of} \cite{bates2023testing} as the BH procedure applied to this
weighted $p$-value family and we denote its rejection set by $\mathcal{R}^w_\alpha\subset \range{m}$. 
{The $\FDP$ and $\TDP$ of $\mathcal{R}^w_\alpha$ are defined similarly to above.}

Under Assumption~\ref{as:ND} and if $P_0$ is absolutely continuous with respect to $\Pcal$,  a particular weight is the {\it oracle weight function} defined by
\begin{equation}\label{eq:oracleWeightNovelty}
w^*\coloneqq\frac{\dd P_0}{\dd \Pcal}.
\end{equation}
The oracle weighted conformal $p$-values $\pwo{i}{n}$, $i\in \cH_0$, are marginally super-uniform provided that $w^*(+\infty)\geq \sup_{u\in\R} w^*(u)$, see \cite{tibshirani2019conformal}.

\begin{remark}\label{rem:IndependenceND}
The score function $S(\cdot)$ and the weight function $w(\cdot)$ are trained here with an independent training dataset which induces a non-trivial dependence between the scores $(S(Z_i))_{i\in\range{n+m}}$. {However, our results are stated conditionally on the training phase and the score function $S(\cdot)$ and weight function $w(\cdot)$ can be considered as deterministic. Hence, the scores $(S(Z_i))_{i\in\range{n+m}}$ satisfy Assumption~\ref{as:ND} conditionally on the training phase as long as the underlying data $(Z_i)_{i\in\range{n+m}}$ are independent and have the same distribution within a group (calibration, null, alternative), which is a classical assumption \citep{bates2023testing,jin2023model}.}
\end{remark}

\subsection{Spaces for process convergence}\label{sec:Topology}

We study the asymptotic convergence of random processes and we consider usual spaces defined in \cite{Bill1999} and \cite{VW1996}, denoted as usual $D[0,1]$, $D(0,1)$ and $\ell^{\infty}(0,1)$, and which are briefly described below.

First, $D[0,1]$ is the set of \cadlag functions $f:[0,1]\rightarrow \R$ with the usual Skorohod topology. Second, $D(0,1)$ is the set of \cadlag functions $f:(0,1)\rightarrow \R$ with the extended Skorohod topology\footnote{This definition is similar to the convergence in $D[0,+\infty)$ considered in  \cite{Bill1999}} defined as follows: $(x_n)_n\in D(0,1)^\N$ converges to $x\in D(0,1)$ if and only if, for all $(a,b)\in(0,1)^2$ with $a<b$ such that $x$ is continuous at points $a$ and $b$, the sequence $(x_n)_n$ restricted to $[a,b]$ converges to $x$ restricted to $[a,b]$ in $D[a,b]$ with the usual Skorohod topology. Finally, $\ell^{\infty}(0,1)$ is the space of locally bounded functions $f:(0,1)\rightarrow \R$ with the topology of the uniform convergence on all compact sets of $(0,1)$: a sequence $(f_n)_n\in[\ell^{\infty}(0,1)]^{\N}$ converges to $f\in\ell^{\infty}(0,1)$ if and only if for all $K$ compact subsets of $(0,1)$, $\sup_{x\in K}\abs{f_n(x)-f(x)}\rightarrow 0$.
We also consider $D(\R)$ the space of \cadlag functions from $\R$ to $\R$ with the extended Skorohod topology and $\ell^{\infty}(\R)$ the space of locally bounded functions from $\R$ to $\R$ with the topology of the uniform convergence on all compact sets of $\R$.

\section{Asymptotics for conformal prediction}\label{sec:predictionSetAsymptotic}

The aim here is to study the asymptotic properties of the processes $\FCPm$ \eqref{eq:FCP} and $\FCPmw$ \eqref{eq:FCPWeight}. We study each process in both exchangeable or distribution shift framework.

\subsection{Main results}
Let us 
 introduce the two following quantities:
\begin{align}
I(t)&\coloneqq t,\  t\in[0,1];\label{eq:identity}\\
G(t)&\coloneqq 1-\Ftest\circ\Fcal^{-1}(1-t),\  t\in(0,1),\label{eq:altcdf}
\end{align}
where $\Fcal^{-1}(u)=\inf\set{x\in\R:\Fcal(x)\geq u}$ denotes the general inverse of $\Fcal$.
Formally, $I$ and $G$ correspond to the c.d.f. of the theoretical $p$-values $p^{(+\infty)}=1-\Fcal(T)$ when $T\sim\Pcal$ and $T\sim \Ptest$, respectively. We denote $G'$ the derivative of $G$ provided that it is well defined.

\begin{theorem}\label{thr:BB}
Under Assumption~\ref{as:CP} with $\Fcal=\Ftest$, we have
\begin{equation*}
\sqrt{\tau_{n,m}}\paren{\FCPm-I}\cvloi \U \text{ on $D[0,1]$,}
\end{equation*}
where $\tau_{n,m}$ is defined by \eqref{eq:taunm} and $\U$ is a standard Brownian bridge.
\end{theorem}

\begin{remark}
Theorem~\ref{thr:BB} is true under slightly more general assumptions: namely, it is sufficient that all the scores are exchangeable and without ties almost surely
(see Theorem~\ref{thr:BBechangeable} for a formal statement).
Using the latter is useful since non i.i.d. scores arise naturally in various setting, like in link prediction for graph \citep{blanchard2024fdr} or when using adaptive scores \citep{gazin2023transductive}.
\end{remark}

Theorem~\ref{thr:BB} is proved in Section~\ref{proof:BB}. 
As a corollary, if $n/(n+m)\rightarrow \sigma^2\in(0,1]$, this gives 
\begin{equation*}
\sqrt{m}\paren{\FCPm-I}\cvloi \sigma^{-1}\U.
\end{equation*}
The latter shows that if  the ratio $\tau_{n,m}m^{-1}$ or $nm^{-1}$ is small (see Lemma~\ref{lemma:taunmAsymptotic} for the equivalence), the FCP becomes asymptotically over-dispersed. This is markedly different from the case where we have $m$ i.i.d. $p$-values uniformly distributed on $(0,1)$ (compare with the Donsker theorem, see Theorem~\ref{lemma:Donsker}).

Theorem~\ref{thr:BB} (or more precisely its extension Theorem~\ref{thr:BBechangeable} when the scores are exchangeable and without ties almost surely)  can be formulated as a result on the asymptotic properties of the empirical c.d.f. of the random vectors $(q_1,\cdots,q_m)$ with the distribution $P_{n,m}$ defined in \cite{gazin2023transductive}. More formally, let  us assume that we are doing $m$ sucessive draws in a P\'{o}lya urn model with initially $1$ ball of each of the $n+1$ colors labelled  $\set{i/(n+1),\ i\in\range{n+1}}$. We denote  $(q_i^{(n)})_{i\in\range{m}}$ the colors of the $m$ successive draws. Then Theorem~\ref{thr:BB} yields,
\begin{equation*}
 \paren[4]{\sqrt{\tau_{n,m}}\brac[3]{m^{-1}\sum_{i\in\range{m}}\1{q_i^{(n)}\leq t}-t}}_{t\in[0,1]}\cvloi \U\ \text{on }D[0,1],
\end{equation*}
with $\U$ being a standard Brownian Bridge. Since obtaining this P\'{o}lya urn consists on a first sampling of $n+m$ i.i.d. real random variables with a continuous c.d.f., and then sample without replacement $n$ points to compute the ranks of each of the $m$ other points, Theorem~\ref{thr:BB} can also be used to obtain the asymptotic of this kind of sampling without replacement procedure.
Another interesting application is the convergence of the uniform distribution over the set of discrete random measure on the set $\set{i/(n+1),\ i\in\range{n+1}}$ with $m$ atoms to the uniform distribution on $(0,1)$ whenever $n\wedge m\rightarrow +\infty$. Formally, let $Q_m^{(n)}$ be a random probability measure drawn uniformly from the set   of discrete random measure on the set $\set{i/(n+1),\ i\in\range{n+1}}$ with $m$ atoms. Denoting for all $t\in[0,1]$, $F_m^{(n)}(t)=Q_m^{(n)}([0,t])$ its c.d.f. Theorem~\ref{thr:BB} yields the following convergence in distribution
\begin{equation*}
\sqrt{\tau_{n,m}}\paren{F_m^{(n)}-I}\cvloi \U\ \text{on $D[0,1]$},
\end{equation*} 
with $\U$ being a standard Brownian Bridge. This gives that the random measure $(Q_m^{(n)})_{m,n}$ converges in probability to the uniform distribution on $(0,1)$ for the Kolmogorov Smirnov divergence, with an asymptotic quantification of its fluctuations. All those convergences are direct consequences a direct corollary of Theorem~\ref{thr:BB} (and Theorem A in the supplementary material of \cite{gazin2023transductive} giving us various properties of $P_{n,m}$). 
However, establishing such links between conformal inference and P\'{o}lya urn theory in the event of a distribution shift is neither straightforward nor direct.

Let us now consider the case where we have potentially a distribution shift, that is, $\Fcal\neq \Ftest$. For this, we consider the following additional assumptions
on $\Fcal$ and $\Ftest$.

\begin{assumption}\label{as:strictCroissance}
$\Fcal$ is increasing on its support $(a,b)\subset\R$ with $-\infty\leq a<b\leq +\infty$ and is continuously differentiable.
In addition, $\Ftest$ is continuously differentiable.
\end{assumption}

\begin{remark}
Assumptions on $\Fcal$ and $\Ftest$ (Assumption~\ref{as:strictCroissance} 
and the continuity assumption of Assumption~\ref{as:CP}) can be slightly relaxed by rather imposing a condition on the marginal distribution $\Ftest\circ\Fcalcag^{-1}$ of the theoretical $p$-value \eqref{eq:theopvalue}: Section~\ref{sec:WeakCPlimit} is devoted to the formal statement of those relaxed theorems.
\end{remark}

\begin{theorem}\label{thr:cvAlter}
Under Assumptions~\ref{as:CP} and \ref{as:strictCroissance} 
 and assuming that $n/(n+m)$ tends to $\sigma^2\in[0,1]$, we have
\begin{align*}
\sqrt{\tau_{n,m}}\paren{\FCPm-G}\cvloi\sigma\mathbb{U}\circ G+\sqrt{1-\sigma^2}G'\mathbb{V}
\text{ on $D(0,1)$},
\end{align*}
where $\U$ and $\V$ are two independent Brownian bridges.
\end{theorem}

Theorem~\ref{thr:cvAlter} is proved in Section~\ref{proof:cvAlter}. When $\Ftest=\Fcal$ it recovers Theorem~\ref{thr:BB} (under the additional convergence assumption of $n/(n+m)$) since if $\U$ and $\V$ are two independent standard Brownian bridges and $\sigma\in[0,1]$, then $\sigma\U+\sqrt{1-\sigma^2}\V$ is also a standard Brownian bridge. By contrast, when $\Ftest\neq\Fcal$,  Theorem~\ref{thr:cvAlter}  quantifies the effect of the distribution shift. In particular, the limit of $\FCPm$ is $G$ instead of simply $I$, thus  $\FCPm(\alpha)$ is not provided to converge to $\alpha$ anymore.

The conclusions of Theorems~\ref{thr:BB} and \ref{thr:cvAlter} involve a double asymptotic in $n$ and $m$ for which $n/(n+m)\rightarrow \sigma^2\in[0,1]$. The different regimes are summarized in Table~\ref{tab:3Regime}. Whenever $\sigma^2$ is neither $0$ nor $1$ this is equivalent to have $m\sim_{\tau_{n,m}\rightarrow \infty}(\sigma^{-2}-1)n$ and thus $m$ and $n$ are asymptotically proportional. For $\sigma^2=0$ (`under-calibrated' regime),  the convergence rate is slower, of order $\sqrt{n}\ll \sqrt{m}$, for instance $m^{1/4}$ for $n=\sqrt{m}$, or $\sqrt{m(\log(m))^{-1}}$ for $n=m(\log(m))^{-1}$. This corresponds to the case where the conformal $p$-values are `strongly'  dependent.
On the other hand, if $\sigma^2=1$ (`over-calibrated' regime), the rate $\sqrt{\tau_{n,m}}$ is equivalent to the classical rate in case of independence $\sqrt{m}$, with the same asymptotic variance. In that case, the convergence result is the same as for the independent theoretical $p$-values \eqref{eq:theopvalue}.
This observation still apply in the weighted conformal prediction case.

{\scriptsize
{%\small
\renewcommand{\arraystretch}{1.7}
%\setcellgapes{1pt}
%\makegapedcells
 \begin{table}[h!]
 \centering
\begin{tabular}{|c||c|c|c|}
   \cline{2-4}
    \multicolumn{1}{c|}{} & {Regime of convergence}& {Rate of convergence: $\sqrt{\tau_{n,m}}$}& Variance term\\
   \hline
\multirow{3}*{Conformal}&Under-calibrated: $m/n\rightarrow +\infty$ & $\sqrt{n}$ & $\sigma^2=0$ \\ \cline{2-4}
& Proportional: $n\sim\gamma m,\ \gamma >0$ & $\sqrt{\frac{\gamma}{\gamma+1}m}$& $\sigma^2={\frac{\gamma}{\gamma+1}}\in(0,1)$ \\ \cline{2-4}
&Over-calibrated: $m/n\rightarrow 0$ & $\sqrt{{m}}$  & $\sigma^2=1$ \\ 
    \hline
Independent & Theoretical $p$-values (as $n=+\infty$) & $\sqrt{m}$ & as $\sigma^2=1$ \\
\hline
 \end{tabular}
 \caption{Three different regimes of convergence: $\sigma^2=0$, $\sigma^2\in(0,1)$ and $\sigma^2=1$ linked with the ratio between the number of calibration points and test points.
   The case of theoretical independent $p$-values is presented for comparison.
 }\label{tab:3Regime}
 % \vspace{-5mm}
 \end{table}
}}

\subsection{Weighted case}\label{sec:CPNotationWeight}

Let $w:\R\cup \{\infty\}\mapsto \R^+$ be a  bounded and measurable weight function. 
To describe the asymptotic behavior of $\FCPmw$, we need to introduce few additional quantities. First, we let
\begin{equation}\label{eq:cdfcalWeight}
\Wcal(t)\coloneqq\frac{\int_{-\infty}^t w(u)\dd \Pcal(u)}{\int_\R w(u)\dd \Pcal(u)},\  t\in\R,
\end{equation}
which corresponds to the c.d.f. of the distribution 
induced by weight function $w$ on the distribution $\Pcal$ (i.e. proportional to $w(u)\dd\Pcal(u)$). Note that when choosing the oracle weight function \eqref{eq:oracleWeight}, the latter is simply the c.d.f. of $\Ptest$. 
Second, we let
\begin{equation}\label{eq:altcdfWeight}
G^w(t)\coloneqq 1-\Ftest\circ(\Wcal)^{-1}(1-t),\  t\in(0,1),
\end{equation}
the c.d.f. of the theoretical $p$-values $p^{w,(+\infty)}=1-\Wcal(T)$ with $T\sim\Ptest$, and we denote $(G^w)'$ its derivative  when it exists.
Finally, we introduce quantities 
involved in the asymptotic variance:
\begin{align}
\rw&\coloneqq\frac{\paren{\int_\R w(u)^2\dd \Pcal(u)}^{1/2}}{\int_\R w(u)\dd \Pcal(u)}\geq 1;\label{eq:varianceratio}\\
I^w(t)&\coloneqq 1-{V}^w_{{\tiny\mbox{cal}}}\circ{(\Wcal)^{-1}}(1-t),\  t\in(0,1),\label{eq:varianceIdentity}
\end{align}
with
\begin{equation}
{V}^w_{{\tiny \mbox{cal}}}(t)\coloneqq \frac{\int_{-\infty}^t w(u)^2\dd \Pcal(u)}{\int_\R w(u)^2\dd \Pcal(u)},\  t\in\R.\label{eq:variancecdf}
\end{equation}

\begin{assumption}\label{as:weight}
The weight function $w:\R\cup \{\infty\}\rightarrow \R^{+}$ is uniformly bounded by a constant and is measurable.   
Moreover, $\Wcal$ is increasing on its support $(a^w,b^w)\subset\R$ with $-\infty\leq a^w<b^w\leq +\infty$ and is continuously differentiable. 
\end{assumption}
Typically, Assumption~\ref{as:weight} holds if $\set{t\in\R,\ w(t)>0}$ is a possibly infinite interval  of $\R$ on which $w$ is bounded and continuous, and $\Fcal$ satisfies Assumption~\ref{as:strictCroissance}.

\begin{theorem}\label{thr:cvAlterWeight}
Under Assumptions~\ref{as:CP}, \ref{as:strictCroissance}, 
 \ref{as:weight} 
 and assuming that $n/(n+m)$ tends to $\sigma^2\in[0,1]$, we have
\begin{equation*}
\sqrt{\tau_{n,m}}\paren{\FCPmw-G^w}
\cvloi\sigma\U\circ G^w+\sqrt{1-\sigma^2}\rw (G^w)'\paren{\mathbb{V}\circ{I^w}+\brac{I-I^w}N} \text{ on $D(0,1)$},
\end{equation*}
where $\U$, $\V$ are two independent standard Brownian bridges, $N$ is an independent standard Gaussian random variable, $\rw$ is defined by \eqref{eq:varianceratio} and $I^w$  is defined by \eqref{eq:varianceIdentity}.
\end{theorem}

Theorem~\ref{thr:cvAlterWeight} is proved in Section~\ref{proof:cvAlterWeight}. Note that it recovers Theorem~\ref{thr:cvAlter} when $w\equiv1$. 
The asymptotic expectation of $\FCPmw(\alpha)$ is $G^w(\alpha)$ which is different from $\alpha$ for a general $w$. Interestingly, when choosing the oracle weight function $w^*$ given by \eqref{eq:oracleWeight}, we have $G^{w^*}=I$ and the convergence in Theorem~\ref{thr:cvAlterWeight} reads as follows:

\begin{theorem}\label{thr:BBWeight}
Under Assumptions~\ref{as:CP} and~\ref{as:strictCroissance}, assume that $\Ptest$ is absolutely continuous with respect to $\Pcal$ and  that the oracle weight function $w^*$ \eqref{eq:oracleWeight} satisfies Assumption~\ref{as:weight}. 
If the ratio $n/(n+m)$ tends to $\sigma^2\in[0,1]$, we have, 
\begin{equation*}
\sqrt{\tau_{n,m}}\paren{\FCPmwo-I}\cvloi\sigma\U+\sqrt{1-\sigma^2}\rwo\paren{\mathbb{V}\circ{I^{w^*}}+\brac{I-I^{w^*}}N} \text{ on $D(0,1)$},
\end{equation*}
where $\U$, $\V$ are two independent standard Brownian bridge and $N$ is an independent standard Gaussian random variable.
\end{theorem}

Theorem~\ref{thr:BBWeight} shows that, in the case of a perfect training of the weight function, one can recover the same asymptotic behavior as in case of no distribution shift. However, note that the asymptotic variance is no longer universal and depends on the distributions proportionnal to $w\dd\Pcal$ and $w^2\dd\Pcal$. By contrast, Theorem~\ref{thr:cvAlterWeight} quantifies the impact  of the distribution shift when the weight function is mis-specified, with a limit $G^w$ and not $I$. Comparing this to Theorem~\ref{thr:cvAlter}, Theorem~\ref{thr:cvAlterWeight} indicates that, for a level $\alpha$, the FDP behavior is (asymptotically) more favorable if we consider a weight function such that  $G^w(\alpha)$ is closer to $\alpha$ than $G(\alpha)$. This phenomenon is illustrated in Figure~\ref{fig:IlluNonOracle} in a particular setting.

Finally, the covariance terms of Theorems~\ref{thr:cvAlterWeight}~and~\ref{thr:cvAlter} can be used to build asymptotic confidence intervals for the FCP, as illustrated in Figure~\ref{fig:IlluNonOracle}.

\section{Asymptotics for novelty detection}\label{sec:ND}

Recall the novelty detection setting of Section~\ref{sec:noveltyDetect}. 
The aim here is to study the asymptotic properties of the process $\FDPm(\cR_\alpha)$ \eqref{eq:FDP} and of the process $\TDPm(\cR_\alpha)$ \eqref{eq:TDP}.

\subsection{Additional notation and assumptions}

 We denote, for all $m\geq 1$, $m_0(m)=|\cH_0\cap \range{m}|$ and $\pi_0(m)=m_0(m)/m$  the number and the proportion of nulls (i.e. of non-novelties) among $m$ tested points, respectively.
 We introduce the following quantities and recall the definition of $G~\eqref{eq:altcdf}$:
 \begin{align}
 G(t)&=1-\Ftest\circ\Fcal^{-1}(1-t),\ t\in(0,1) \nonumber,\\
G_0(t)&=1-F_0\circ\Fcal^{-1}(1-t),\ t\in(0,1) \label{eq:wrongnullecdf},\\
\Flim&=\pi_0 G_0+(1-\pi_0)G,\label{eq:mixtureLimit}
\end{align}
that correspond to the (limiting) c.d.f. of the theoretical $p$-values~\eqref{eq:theopvalue} under the alternative, the null and under the test sample mixture, respectively.
For a given weight function $w$, we also define their weighted counterparts and {recall the definition of $G^w$~\eqref{eq:altcdfWeight}}:
\begin{align}
G^w(t)&= 1-\Ftest\circ(\Wcal)^{-1}(1-t),\  t\in(0,1), \nonumber \\
G_0^w(t)&=1-F_0\circ{\paren{\Wcal}^{-1}}(1-t),\ t\in(0,1),\label{eq:wrongnullecdfWeight}\\
\Flim^w&=\pi_0 G_0^w+(1-\pi_0)G^{w}\label{eq:mixtureLimitWeight}.
\end{align}
We denote $\Flim'$ and $(\Flim^w)'$ the derivatives of $\Flim$ and $\Flim^w$, respectively. 
 Note that Assumption~\ref{as:ND} with $F_0=\Fcal$ entails $G_0=I$. If $F_0\neq \Fcal$, we still have $G_0^{w^*}=I$ when  $w^*$ is the oracle weight function.

Let us introduce the two following assumptions:
\begin{assumption}\label{as:concavity}
$\Flim$ is strictly concave on $(0,1)$.\end{assumption}
\begin{assumption}\label{as:concavityWeight}
$\Flim^w$ is strictly concave on $(0,1)$.
\end{assumption}
Since $\Flim$ can be seen as the cdf of the theoretical $p$-values \eqref{eq:theopvalue} in the test sample, Assumption~\ref{as:concavity} is classical in the multiple testing literature \citep{GW2002} (see Section~\ref{sec:WeakCPlimit} for more precision), and Assumption~\ref{as:concavityWeight}  corresponds to its weighted version.

As shown in \cite{Chi2007,Neu2008}  in the independent case, there is a critical value for the BH procedure given by $\alpha^*=[\Flim'(0^+)]^{-1}$;  the central limit theorem for the FDP/TDP of BH procedure at level $\alpha$ can only be obtained if $\alpha>\alpha^*$ (for $\alpha<\alpha^*$, the BH procedure at level $\alpha$ has asymptotically no power). This quantity plays a similar role in the (dependent) conformal setting and we prove results only if $\alpha>\alpha^*$.

Finally, as in \cite{Neu2008}, we consider also the following thresholds
\begin{align}
\T_\alpha&=\sup\set{t\in(0,1),\ \Flim(t)\geq \frac{t}{\alpha}};\label{eq:asymptoticThreshold}\\
\T^{w}_\alpha&=\sup\set{t\in(0,1),\ \Flim^w(t)\geq \frac{t}{\alpha}},\label{eq:asymptoticThresholdWeight}
\end{align}
which are well defined (and belong to $(0,1)$) under Assumption~\ref{as:concavity} if $\alpha >[(\Flim)'(0^+)]^{-1}$ and Assumption~\ref{as:concavityWeight} if $\alpha>[(\Flim^w)'(0^+)]^{-1}$, respectively.

\subsection{Main results}

When $P_0=\Pcal$ (no distribution shift), the following result provides the convergence behavior of the novelty detection procedure of  \cite{bates2023testing}, that is, of $\mathcal{R}_\alpha$ given in Section~\ref{sec:noveltyDetect}. 

\begin{theorem}\label{thr:BH95}
Under Assumption~\ref{as:ND} with $P_0=\Pcal$, Assumption~\ref{as:strictCroissance} 
and Assumption~\ref{as:concavity}, let us consider $\mathcal{R}_\alpha$ the BH procedure at level $\alpha$ applied to the conformal $p$-values \eqref{eq:IntroConfpvalues}, with a level $\alpha>[\Flim'(0^+)]^{-1}$. If $n/(n+m)\rightarrow \sigma^2\in[0,1]$ and $\pi_0(m)\rightarrow\pi_0\in(0,1)$, we have
\begin{align}
\sqrt{\tau_{n,m}}\paren{\FDPm(\mathcal{R}_\alpha)-\pi_0\alpha}&\cvloi \mathcal{N}\paren{0,\alpha^2\pi_0\brac{\sigma^2+(1-\sigma^2)\pi_0}\frac{1-\T_\alpha}{\T_\alpha}};\label{eq:asymptoticFDPBH}\\
\sqrt{\tau_{n,m}}\paren{\TDPm\paren{\mathcal{R}_\alpha}-G(\T_\alpha)}&\cvloi\mathcal{N}\paren{0,\frac{\Sigma_\alpha}{(\alpha^{-1}-\Flim'(\T_\alpha))^2}},\label{eq:asymptoticTDPBH} 
\end{align}
with
$
\Sigma_\alpha={{G'(\T_\alpha)^2\T_\alpha(1-\T_\alpha)\brac{\pi_0\sigma^2+(1-\sigma^2){\alpha^{-2}}}}}{}
+{\brac{\alpha^{-1}-\pi_0}^2\paren{1-\pi_0}^{-1}G(\T_\alpha)(1-G(\T_\alpha))\sigma^2}{}.
$
\end{theorem}

Theorem~\ref{thr:BH95}, proved in Section~\ref{proof:BH95}, relies on the pipeline introduced by \cite{Neu2008}, which consists in first deriving a functional central limit theorem for the e.c.d.f. of the $p$-values and then using the functional delta method \citep{Vaart1998}. 

Interestingly, under the assumptions of Theorem~\ref{thr:BH95} and if $\sigma^2>0$, \eqref{eq:asymptoticFDPBH} yields 
\begin{equation*}
\sqrt{m}\paren{\FDPm(\mathcal{R}_\alpha)-\pi_0\alpha}\cvloi \mathcal{N}\paren{0,\alpha^2\pi_0\brac{1+(\sigma^{-2}-1)\pi_0}\frac{1-\T_\alpha}{\T_\alpha}}.
\end{equation*}
When $\sigma=1$ (that is $n/m \to \infty$), we note that the FDP convergence is the same as when the $p$-values are independent \citep{Neu2008}. However, when $n/m$ is bounded, the asymptotic variance is affected by the dependence and gets larger when $\sigma$ decreases.  
In addition, and maybe more importantly, the convergence rate is heavily affected in the under-calibrated setting when $\sigma^2=0$, with a convergence rate $\sqrt{\tau_{n,m}}\ll \sqrt{m}$. This phenomenon is illustrated in Figure~\ref{fig:AsymptoticFDPBH} for $m=10^3$ and various $n$ in a particular setting, with corresponding convergence rates reported in Table~\ref{tab:rateConvergence}. The three different regimes from the prediction setting displayed in Table~\ref{tab:3Regime} also appear in this novelty detection setting.

\begin{figure}[h!]
\center
\includegraphics[scale=0.4]{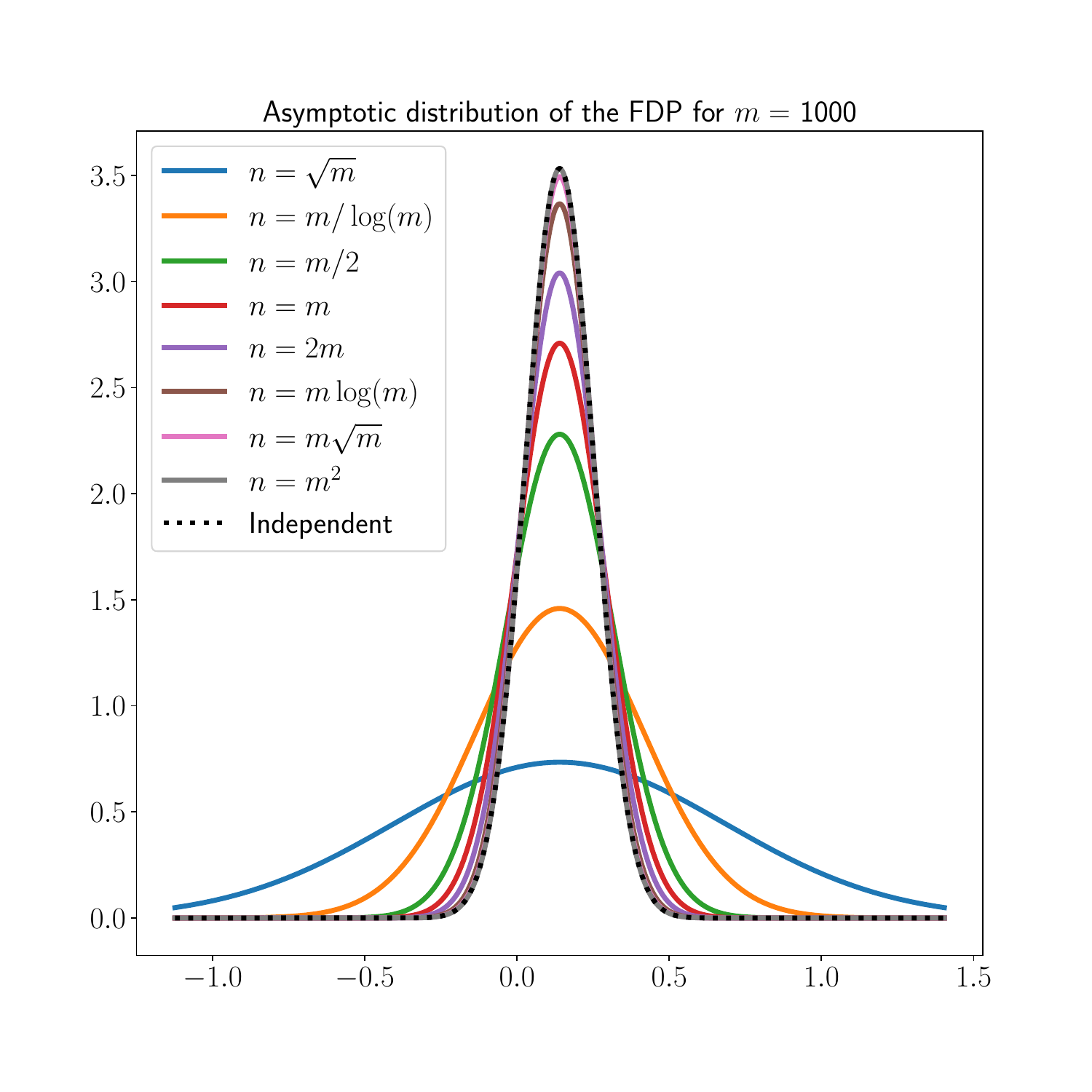}
\caption{Asymptotic approximation of the distribution of the $\FDP$ of the Benjamini-Hochberg procedure with conformal $p$-values given by Theorem~\ref{thr:BH95} with $m=10^3$ and several choice of $n$ (plain lines) at level $1-\alpha=0.8$. The case of independent theoretical  $p$-values (dashed line) is displayed for comparison. The null distribution and the calibration distribution is $\mathcal{N}(0,1)$ while the distribution of the alternative is $\mathcal{N}(3,1)$, with a null proportion $\pi_0$ equal to $70\%$. 
  }\label{fig:AsymptoticFDPBH}
\end{figure}

{\scriptsize
{%\small
\renewcommand{\arraystretch}{1.7}
%\setcellgapes{1pt}
%\makegapedcells
 \begin{table}[h!]
 \centering
\begin{tabular}{|c||c|c|c|}
   \cline{2-4}
    \multicolumn{1}{c|}{} &{Calibration sample size: $n$} & {Rate of convergence: $\sqrt{\tau_{n,m}}$}& Independence coefficient: $\sigma^2$\\
   \hline
\multirow{9}*{Conformal}&$\sqrt{m}$ & $m^{1/4}$ & $0$ \\ \cline{2-4}
&$\frac{m}{\log(m)}$ & $\sqrt{\frac{m}{\log(m)}}$& $0$ \\ \cline{2-4}
&$\frac{m}{2}$ & $\sqrt{\frac{m}{3}}$  & $\frac{1}{3}$ \\ \cline{2-4} 
&$m$ & $\sqrt{\frac{m}{2}}$ & $\frac{1}{2}$ \\ \cline{2-4} 
&$2m$ & $\sqrt{\frac{2m}{3}}$ & $\frac{2}{3}$ \\ \cline{2-4} 
&$\gamma m$, $\gamma>0$ & $\sqrt{\frac{\gamma}{\gamma+1}m}$ & $\frac{\gamma}{\gamma +1}$ \\ \cline{2-4}
&$m\log(m)$ & $\sqrt{m}$ & $1$ \\ \cline{2-4}
&$m^{3/2}$ & $\sqrt{m}$ & $1$ \\ \cline{2-4}
&$m^2$ & $\sqrt{m}$& $1$ \\ 
    \hline
Independent &$+\infty$ & $\sqrt{m}$ & $1$\\
\hline
 \end{tabular}
 \caption{
{{Rates $\sqrt{\tau_{n,m} }$ and $\sigma^2$ in Theorems~\ref{thr:cvAlter},~\ref{thr:cvAlterWeight},~\ref{thr:BH95} and \ref{thr:BH95Weight}, with different choices of $n$ and $m$.}
  Except the case $n=\gamma m$, they all correspond to the different regimes displayed in Figure~\ref{fig:AsymptoticFDPBH}. The case with independent $p$-values is presented for comparison. }
 }\label{tab:rateConvergence}
 % \vspace{-5mm}
 \end{table}
}}
Under distribution shift, the following result holds. 

\begin{theorem}\label{thr:BH95Weight}
Under Assumptions~\ref{as:ND} and~\ref{as:concavityWeight}, assume that $P_0$ is absolutely continuous with respect to $\Pcal$ and that the oracle weight function $w^*$  \eqref{eq:oracleWeightNovelty} satisfies Assumptions~\ref{as:strictCroissance} and~\ref{as:weight}. Consider the BH procedure at level $\alpha$ applied to the oracle weighted $p$-values $(\pwo{i}{n})_{i\in \range{m}}$ \eqref{eq:IntroConfpvaluesWeight}, with a level $\alpha>[(\Flimw)'(0^+)]^{-1}$. Then if  $n/(n+m)\rightarrow \sigma^2\in[0,1]$ and $\pi_0(m)\rightarrow\pi_0\in(0,1)$, we have
\begin{align}
\sqrt{\tau_{n,m}}\paren{\FDPmw(\mathcal{R}^{w^*}_\alpha)-\pi_0\alpha}&\cvloi \Nor\paren{0,\alpha^2\pi_0\Xi^{{w^*}}_\alpha\frac{1-\T^{{w^*}}_\alpha}{\T^{{w^*}}_\alpha}};\label{eq:asymptoticFDPBHWeight}\\
\sqrt{\tau_{n,m}}\paren{\TDPmw(\mathcal{R}^{w^*}_\alpha)-{G^{w^*}}(\T^{{w^*}}_\alpha)}&\cvloi\mathcal{N}\paren{0,\frac{\Sigma^{{w^*}}_\alpha}{(\alpha^{-1}-(\Flimw)'(\T^{{w^*}}_\alpha))^2}}\label{eq:asymptoticTDPBHWeight},
\end{align}
where we denote
\begin{align*}
\Xi^{{w^*}}_\alpha&=\sigma^2+(1-\sigma^2)\rwo^2\pi_0\frac{I^{w^*}(\T^{{w^*}}_\alpha){\T^{{w^*}}_\alpha}^{-1}+{\T^{{w^*}}_\alpha}-2I^{w^*}(\T^{{w^*}}_\alpha)}{1-\T^{w^*,*}_\alpha};\\
\Sigma^{{w^*}}_\alpha&={{\paren{{(G^{w^*})}'(\T^{{w^*}}_\alpha)}^2\T^{{w^*}}_\alpha(1-\T^{{w^*}}_\alpha)\pi_0\sigma^2}}{}\nonumber\\
&\ +{\paren{\rwo{(G^{w^*})}'(\T^{{w^*}}_\alpha){\alpha^{-1}}}^2\brac{\T^{{w^*}}_\alpha-I^{w^*}(\T^{{w^*}}_\alpha)^2}(1-\sigma^2)}{}\nonumber\\
&\ +{\brac{\alpha^{-1}-\pi_0}^2\paren{1-\pi_0}^{-1}{G^{w^*}}(\T^{{w^*}}_\alpha)(1-{G^{w^*}}(\T^{{w^*}}_\alpha))\sigma^2}{}. 
\end{align*}

\end{theorem}

Theorem~\ref{thr:BH95Weight} is proved in Section~\ref{proof:thr:BH95Weight}. 
The weighted conformal BH procedure with (similar) weighted conformal $p$-values has been studied in 
\cite{jin2023model} and they obtained a convergence of FDR and TDR 
(i.e. $\E[\FDPmw(\mathcal{R}^{w^*}_\alpha)]$ and $\E[\TDPmw(\mathcal{R}^{w^*}_\alpha)]$) to quantities analogue to $\pi_0\alpha$ and ${G^{w^*}}(\T^{{w^*}}_\alpha)$, respectively. By contrast, Theorem~\ref{thr:BH95Weight} provides the full asymptotic distribution.

While Theorem~\ref{thr:BH95Weight} focuses on the oracle weight function, the non-oracle case is deferred to Section~\ref{sec:BHnonUniform} (see Theorem~\ref{thr:BH95ProcessnonNull}) .

The limiting distribution in Theorems~\ref{thr:BH95} and \ref{thr:BH95Weight} are similar to the previous ones in the literature \citep{Neu2008,DR2011,DR2016,kluger2024central} which is well expected since all these works study the asymptotic normality of the $\FDP$ and the $\TDP$ of the $\BH$ procedure. However, the novelty is in the setting: we do not have access to the theoretical and uniform $p$-values under the null, but only to the conformal $p$-values. Conformal $p$-values are by nature dependent (and $\mbox{PRDS}$ for the unweighted ones \citep{bates2023testing}) even with independent scores, while the dependency structures on the theoretical $p$-values in \cite{DR2011,DR2016,kluger2024central} come from the dependency of the gaussian test statistics. Furthermore conformal $p$-values are approximations of the theoretical ones, hence the study of the double asymptotic in both the nimber of test $m$ (which is the only one studied in the  literature) and the size $n$ of the calibration set is crucial and provided here.
Equivalent of the rate of convergence $\sqrt{\tau_{n,m} }$ and the limit $\sigma^2$ in Theorems~\ref{thr:cvAlter} and \ref{thr:BH95} and their weighted version Theorems~\ref{thr:cvAlterWeight} and \ref{thr:BH95Weight} for different choice of $n$ depending on $m$.
\subsection{Applications}

Theorem~\ref{thr:BH95} has several possible applications. 
A first application is to obtain an asymptotic confidence region of the $\FDP$ of the BH procedure in the conformal case. 

\begin{corollary}\label{cor:asymptoticFDPControl}
Let $\delta>0$. Under Assumption~\ref{as:ND} with $P_0=\Pcal$ and Assumptions~\ref{as:strictCroissance} and \ref{as:concavity}, let us consider $\mathcal{R}_\alpha$ the BH procedure at level $\alpha$ applied to $m$ conformal $p$-values \eqref{eq:IntroConfpvalues} with a calibration sample of size $n$, with a level $\alpha>[\Flim'(0^+)]^{-1}$. Denote $\widehat{\T}_\alpha=\alpha\abs{\mathcal{R}_\alpha}/m$ the threshold of $\mathcal{R}_\alpha$.
Define the upper bound:
\begin{equation*}
\widehat{\lambda}_{n,m}(\delta,\alpha,\pi_0)=\alpha\sqrt{\pi_0\brac{\frac{1}{m}+\frac{1}{n}\pi_0}\frac{1-\widehat{\T}_\alpha}{\widehat{\T}_\alpha}}\Phi^{-1}(1-\delta),
\end{equation*}
with $\Phi^{-1}$ the quantile function of the standard gaussian distribution.
Then if $n/(n+m)\rightarrow \sigma^2\in[0,1]$ and $\pi_0(m)\rightarrow\pi_0\in(0,1)$, we have
\begin{equation*}
\P\paren{\FDPm(\mathcal{R}_\alpha)\leq \pi_0\alpha+\widehat{\lambda}_{n,m}(\delta,\alpha,\pi_0)}\underset{\tau_{n,m}\rightarrow +\infty}{\longrightarrow} 1-\delta.
\end{equation*}
In addition, if  $\widehat{\pi}_0\in[0,1]$ is an estimator of $\pi_0$ such that $\varliminf\widehat{\pi}_0\geq \pi_0$ then,
\begin{equation}\label{eq:asymptoticApproximationFDPBH}
\varliminf_{\tau_{n,m}\rightarrow +\infty}\P\paren{\FDPm(\mathcal{R}_\alpha)\leq \widehat{\pi}_0\alpha+\widehat{\lambda}_{n,m}(\delta,\alpha,\widehat{\pi}_0)} \geq 1-\delta.
\end{equation}
\end{corollary}

Corollary~\ref{cor:asymptoticFDPControl} is a direct application of the asymptotic result \eqref{eq:asymptoticFDPBH}. It provides an asymptotic FDP control which takes into account the fluctuation of the FDP around the $\FDR$. In \eqref{eq:asymptoticApproximationFDPBH} the estimator $\wh{\pi}_0$ should satisfy $\varliminf\widehat{\pi}_0\geq \pi_0$. This includes the conservative case where $\wh{\pi}_0$ is constantly equal to $1$, but also the (in general) sharper bound for which $\wh{\pi}_0$ is the Storey estimator \citep{STS2004}, defined as
 \begin{equation}\label{eq:StoreyEstimator}
\wh{\pi}_0^{{\tiny{\mbox{Storey}}}}(m,n)=\frac{1+\sum_{i=1}^{m}\1{\p{i}{m}\geq\lambda}}{m(1-\lambda)},
\end{equation} 
where $\lambda\in(0,1)$ is a fixed parameter. Indeed, 
Applying Theorem~\ref{thr:conformalBridge} (or more precisely arguments in its proof, because the convergence of all empirical processes holds almost surely by the Glivenko-Cantelli theorem and by continuity), we have 
\begin{equation*}
\wh{\pi}_0^{{\tiny{\mbox{Storey}}}}(n,m){{\underset{\tau_{n,m}\rightarrow +\infty}{\longrightarrow}}}\pi_0+(1-\pi_0)\frac{1-G(\lambda)}{1-\lambda}\geq \pi_0\ a.s.
\end{equation*}
Note that using the Storey estimator in the context of conformal $p$-values has been already considered in the literature (e.g., \cite{bates2023testing,marandon2022machine}), but not for deriving an FDP bound.

Another application of Theorem~\ref{thr:BH95} is to use the asymptotic distribution in \eqref{eq:asymptoticTDPBH} to design the score function $S:\mathcal{Z}\rightarrow \R$, see Section~\ref{sec:noveltyDetect}. Recall that the latter is  computed from an independent  training sample $\dtrain$. For simplicity, assume that $\dtrain=\set{(Z_i,h_i)\in\mathcal{Z}\times\set{0,1},i\in\range{n_{{\tiny \mbox{train}}}}}$ is a sample containing both labeled data under the null ($h_i=0$) and under the alternative ($h_i=1$), so that the score function can typically be trained as the best classifier of null versus alternative \citep{marandon2022machine}, or as the likelihood ratio of null and alternative ``conformal $p$-values'' leading to some pruning step \citep{liang2024integrative}.
The asymptotic distribution in \eqref{eq:asymptoticTDPBH} delivers additional insight, and allows 
 to use the asymptotic distribution to design a score function in a more $\TDP$ (or broadly ``multiple testing'') oriented way.
Namely, let us assume that the score function $S$ lies in a parametric class $\set{S_\theta, \theta\in\Theta}$ of score functions  (for example a classifier based on a neural network). The classical machine learning approach is to seek for the best two-class classifier 
by minimising the empirical risk 
$\theta\in\Theta\mapsto n_{{\tiny \mbox{train}}}^{-1}\sum_{i\in\range{n_{{\tiny \mbox{train}}}}}\1{S_\theta(Z_i)\neq h_i}$
{which is an empirical substitute of $\P(S_\theta(Z_1)\neq h_1)$ 
computed on $\dtrain$.}
{Alternatively, let us consider the situation where the user seeks to have the ``majority'' of the $\TDP$ above a specific power level $\beta\in(0,1)$, while conserving the finite-sample $\FDR$ control given by the conformal BH procedure. In this case, the task is to choose $\theta$ that maximises $\P(\TDPm>\beta)\approx 1-\Phi(\sqrt{\tau_{n,m}}L_\alpha(\beta))$ with $L_\alpha(\beta)=(\beta-G(\T_\alpha))(\alpha^{-1}-\Flim'(\T_\alpha))\Sigma_\alpha^{-1/2}$ by using  Theorem~\ref{thr:BH95}. }
{A natural empirical substitution leads to minimise}
\begin{align}\label{equ-empiricalpower}
\theta\in\Theta\mapsto(\beta-\widehat{G}_\theta)(\wh{\T}_{\theta,\alpha})(\alpha^{-1}-\wh{G}_{\theta, {\tiny \mbox{mixt}}}'(\wh{\T}_{\theta,\alpha}))\widehat{\Sigma}_{\theta,\alpha}^{-1/2},
\end{align}
where $\wh{G}_{\theta}$ 
 is a differentiable estimate based on $\dtrain$ (e.g. a kernel estimator or a local polynomial estimator) of $G$~\eqref{eq:altcdf} when the score function is $S_\theta$ (i.e. of $1-F_{\theta,\tiny \mbox{test}}\circ F_{\theta,\tiny \mbox{cal}}^{-1}(1-\cdot)$ where for all $t\in\R$, $F_{\theta,\tiny \mbox{test}}(t)$ (resp. $F_{\theta,\tiny \mbox{cal}}(t)$) is equal to $\P(S_\theta(Z_i)\leq t)$ with $i\in\cH_0$ (resp. $i\in\cH_1$)), its derivative is denoted $\wh{G}'_{\theta}$ and all the other quantities are plug-in estimates of $\Flim$~\eqref{eq:mixtureLimit}, $\T_\alpha$~\eqref{eq:functionalBH} and $\Sigma_\alpha$ defined in \eqref{eq:asymptoticTDPBH} with $G$ and $G'$ replaced by $\wh{G}_{\theta}$ and $\wh{G}'_{\theta}$, the function $G_0$ replaced by $I$ and $\pi_0$ replaced by $n_{{\tiny \mbox{train}}}^{-1}\sum_{i\in\range{n_{{\tiny \mbox{train}}}}}(1-h_i)$. 
{More generally, other objective functions can be optimised by considering other $\TDP$ related quantities approximated by the asymptotic distribution~\eqref{eq:asymptoticTDPBH} or by using
versions of the empirical risk that are penalised by \eqref{equ-empiricalpower}. 
Finally, a similar approach can be followed in the case of a distribution shift by applying Theorem~\ref{thr:BH95Weight} (or its generalisation Theorem~\ref{thr:BH95ProcessnonNull}). In this case, the oracle weight should also be seen as a function of $\theta$ (if the oracle weight can be known for every score function $S_\theta$, $\theta\in\Theta$) or must be chosen among a parameterised family of weight function $\set{w_{\lambda,\theta},\ (\lambda,\theta)\in\Lambda\times\Theta}$ (when only some information on the distribution shift is known) and the training sample should contain both clean and shifted data.}
This heuristic way to develop score (and weight) function with a $\TDP$ focus and not only a best ``marginal'' classification is left for future work.

\section{Proofs} 
\label{sec:empiricalProcess}

In this section, we prove the main results of the paper. 
They rely on a particular decomposition of the FCP/FDP into two processes that further jointly converge. Applying the functional delta method \citep{Vaart1998} then allows to conclude.

\subsection{Proofs for Section~\ref{sec:predictionSetAsymptotic}}\label{proof:pred}

\subsubsection{$\FCP$ decomposition}

To emphasise that the FCP is the e.c.d.f. of $({\p{i}{n}})_{i\in\range{m}}$, we let for all $n\geq 1$ and $m\geq 1$,
\begin{equation}\label{eq:ecdf}
\Fm{m}{n}\coloneqq\frac{1}{m}\sum_{i=1}^m\1{\p{i}{n}\leq t} = \FCPm(t),\ t\in[0,1].
\end{equation}
We also introduce, for all $n\geq 1$ and $m\geq 1$, 
\begin{align}
\Gncal{(n)}(t)&=\frac{1}{n+1}\sum_{k=1}^n \1{S_k\leq t},\  t\in\R,\label{eq:ecdftest}\\
\Fmtest(t)&=\frac{1}{m}\sum_{i=1}^{m}\1{T_i\leq t},\  t\in\R,\label{eq:ecdfcal}
\end{align}
corresponding to the e.c.d.f. of the calibration score sample (with $+\infty$) and test score sample, respectively.
Note that we have $\p{i}{n}=1-\Gncal{(n)}(T_i)$ almost surely (under Assumption~\ref{as:CP} to ensure that there are no ties almost surely).

By Lemma~\ref{lemma:QuantileWeight} we have that $\p{i}{n}\leq t$ if and only if $T_i>{({\Gncal{(n)}})^{-1}}(1-t)$ for all $t\in(0,1)$. Recalling \eqref{eq:altcdf}, this leads to the following decomposition for $\Fm{m}{n}$: for all $t\in(0,1)$,
\begin{align*}
\Fm{m}{n}(t)-G(t)&=\frac{1}{m}\sum_{i\in\range{m}}\1{\p{i}{n}\leq t}-G(t)\\
&=\frac{1}{m}\sum_{i\in\range{m}}\1{{({\Gncal{(n)}})^{-1}}(1-t)<T_i}-\paren{1-\Ftest\circ\Fcal^{-1}(1-t)}\\
&=1-\frac{1}{m}\sum_{i\in\range{m}}\1{{({\Gncal{(n)}})^{-1}}(1-t)\geq T_i}-\paren{1-\Ftest\circ\Fcal^{-1}(1-t)}\\
&=\Ftest\circ\Fcal^{-1}(1-t)-\Fmtest\circ{({\Gncal{(n)}})^{-1}}(1-t).
\end{align*}
Hence, it follows
\begin{align}
\Fm{m}{n}(t)-G(t)
&=\Ftest\circ{({\Gncal{(n)}})^{-1}}(1-t)-\Fmtest\circ{({\Gncal{(n)}})^{-1}}(1-t)\nonumber\\
&\quad+\Ftest\circ\Fcal^{-1}(1-t)-\Ftest\circ{({\Gncal{(n)}})^{-1}}(1-t).\nonumber
\end{align}
Finally, we obtain the decomposition
\begin{align}
\sqrt{\tau_{n,m}}\paren{\Fm{m}{n}(t)-G(t)}=&-\sqrt{\frac{\tau_{n,m}}{m}}\sqrt{m}\paren{\Fmtest\circ{({\Gncal{(n)}})^{-1}}(1-t)-\Ftest\circ{({\Gncal{(n)}})^{-1}}(1-t)}\nonumber\\ 
&-\sqrt{\frac{\tau_{n,m}}{n}}\sqrt{n}\paren{\Ftest\circ{({\Gncal{(n)}})^{-1}}(1-t)-\Ftest\circ\Fcal^{-1}(1-t)}.\label{eq:ProcessDecomposition}
\end{align}

\subsubsection{Joint convergence}

Thanks to \eqref{eq:ProcessDecomposition}, in order to obtain a convergence result for $\sqrt{\tau_{n,m}}\paren{\Fm{m}{n}-G}$, we only need to derive the joint convergence of the two processes delineated in the decomposition \eqref{eq:ProcessDecomposition}.

\begin{proposition}\label{prop:BBCouple}
Under Assumptions~\ref{as:CP} and \ref{as:strictCroissance}  we have, 
\begin{equation}
\begin{pmatrix}
\sqrt{m}\brac{\Fmtest\circ{({\Gncal{(n)}})^{-1}}(1-I)-\Ftest\circ{({\Gncal{(n)}})^{-1}}(1-I)}\\
\sqrt{n}\brac{\Ftest\circ{({\Gncal{(n)}})^{-1}}(1-I)-\Ftest\circ\Fcal^{-1}(1-I)}
\end{pmatrix}
\overset{\Loi}{\rightarrow}
\begin{pmatrix}
\mathbb{U}\circ G\\
G'\mathbb{V}
\end{pmatrix}
\text{ on $\brac{D(0,1)}^2$},\label{eq:jointcvloi}
\end{equation}
where $\U$ and $\V$ are two independent Brownian bridges.
\end{proposition}
Proposition~\ref{prop:BBCouple} is proved in Section~\ref{proof:BBCouple}. The main idea of the proof is to study each coordinate separately and then to use independence to obtain a joint convergence. 
The process of the first coordinate (test term) is studied by  using the Donsker Theorem for $\Fmtest$ and the continuous mapping theorem with a random change of time (Lemma~\ref{cor:JointComposition}). The second coordinate (calibration term) is investigated by using the Donsker theorem for ${({\Gncal{(n)}})^{-1}}$ and then using the functional delta method with  the map $\phi\mapsto \Ftest\circ\phi$ (Lemma~\ref{lemma:hadamardComposition}).

\subsubsection{Proof of Theorem~\ref{thr:BB}}\label{proof:BB}

{To prove Theorem~\ref{thr:BB}, we show the following more general result.}
\begin{theorem}\label{thr:BBechangeable}
Assume that  for all $n,m$ the calibration and test scores $(S_1,\dots,S_n,T_1,\dots,T_m)$ are exchangeable without ties almost surely. Then
\begin{equation*}
\sqrt{\tau_{n,m}}\paren{\FCPm-I}\cvloi \U \text{ on $D[0,1]$,}
\end{equation*}
where $\tau_{n,m}$ is defined by \eqref{eq:taunm} and $\U$ is a standard Brownian bridge.
\end{theorem}

 The two assumptions in Theorem~\ref{thr:BBechangeable} are implied by Assumption~\ref{as:CP} with $\Fcal=\Ftest$, hence Theorem~\ref{thr:BB} is implied by Theorem~\ref{thr:BBechangeable}. Let us prove the latter. 
 Since the calibration and test scores are exchangeable without ties a.s. by  Proposition 2.2 from \cite{gazin2023transductive} we get for all $n,m\in\N^2$ that the conformal $p$-values $(\p{i}{n})_{i\in\range{m}}$ has a distribution only depending on $n$ and $m$, denoted by $P_{n,m}$. Hence, for all $T:D[0,1]\rightarrow \R$ bounded continuous funtion we have :
\begin{equation}\label{eq:SameDistribution}
\E\brac{T\paren{\sqrt{\tau_{n,m}}\paren{\FCPm-I}}}=\E\brac{T\paren{\sqrt{\tau_{n,m}}\paren{\widetilde{\FCP}_m^{(n)}-I}}},
\end{equation}
where $\widetilde{\FCP}_m^{(n)}$ is the $\FCP$ obtained with i.i.d scores uniformly distributed on $(0,1)$. Those specific scores satisfy Assumptions~\ref{as:CP} and \ref{as:strictCroissance} with $\Fcal=\Ftest=I$ on $[0,1]$, and by \eqref{eq:SameDistribution} we only have to prove the convergence in this specific setting.

Up to consider a subsequence, one can assume that $n/(n+m)\rightarrow \sigma^2\in[0,1]$ by compacity of $[0,1]$. By applying Proposition~\ref{prop:BBCouple} with the i.i.d. uniformly distributed on $(0,1)$ scores from the right hand side of \eqref{eq:SameDistribution}, we get by Slutsky's Lemma:
\begin{equation*}
\begin{pmatrix}
\sqrt{m}\brac{\Fmtest\circ{({\Gncal{(n)}})^{-1}}(1-I)-{({\Gncal{(n)}})^{-1}}(1-I)}\\
\sqrt{n}\brac{{({\Gncal{(n)}})^{-1}}(1-I)-(1-I)}\\
\frac{\tau_{n,m}}{{m}} \\
\frac{\tau_{n,m}}{{n}}
\end{pmatrix}
\overset{\Loi}{\rightarrow}
\begin{pmatrix}
\mathbb{U}\\
\mathbb{V}\\
\sigma^2\\
1-\sigma^2
\end{pmatrix}
\text{ on $\brac{D(0,1)}^2\times[0,1]^2$},
\end{equation*}
with $(\U,\V)$ two independent standard Brownian bridges. Applying the continuous mapping theorem with the decomposition \eqref{eq:ProcessDecomposition} gives the following convergence:
\begin{equation*}
\sqrt{\tau_{n,m}}\paren{\widetilde{\FCP}_m^{(n)}-I}\cvloi \sigma\U+\sqrt{1-\sigma^2}\V\overset{\Loi}{=}\W,
\end{equation*}
with $\W$ a standard Brownian bridge. By universality of the limiting distribution and compacity of $[0,1]$, the convergence in distribution holds without considering any subsequence.

\subsubsection{Proof of Theorem~\ref{thr:cvAlter}}\label{proof:cvAlter}

Thanks to Proposition~\ref{prop:BBCouple} we have $\U$ and $\V$ two independent standard Brownian bridges such that \eqref{eq:jointcvloi} holds. 
Since $n/(n+m)\rightarrow \sigma^2\in[0,1]$ and $\tau_{n,m}/n\rightarrow 1- \sigma^2\in[0,1]$ by assumption, we obtain by Slutsky's Lemma
\begin{equation*}
\begin{pmatrix}
\sqrt{m}\brac{\Fmtest\circ{({\Gncal{(n)}})^{-1}}(1-I)-\Ftest\circ{({\Gncal{(n)}})^{-1}}(1-I)}\\
\sqrt{n}\brac{\Ftest\circ{({\Gncal{(n)}})^{-1}}(1-I)-\Ftest\circ\Fcal^{-1}(1-I)}\\
\frac{\tau_{n,m}}{{m}} \\
\frac{\tau_{n,m}}{{n}}
\end{pmatrix}
\overset{\Loi}{\rightarrow}
\begin{pmatrix}
\mathbb{U}\circ G\\
G'\mathbb{V}\\
\sigma^2\\
1-\sigma^2
\end{pmatrix}
\text{ on $\brac{D(0,1)}^2\times[0,1]^2$}.
\end{equation*}
The result follows by applying the continuous mapping theorem with the decomposition \eqref{eq:ProcessDecomposition}.

\subsubsection{Proof of Theorem~\ref{thr:cvAlterWeight}}\label{proof:cvAlterWeight}

For weighted conformal $p$-values, we introduce 
\begin{align}
\Fmw{m}{n}(t)&\coloneqq\frac{1}{m}\sum_{i\in\range{m}}\1{\pw{i}{n}\leq t}=\FCPmw(t),\  t\in\R;\label{eq:FCPProcessWeight}
\\
\Gncal{w,(n)}(t)&\coloneqq\frac{\sum_{k=1}^{n}w(S_k)\1{S_k\leq t}}{w(+\infty)+\sum_{k=1}^{n}w(S_k)},\  t\in\R,\label{eq:ecdfcalWeight}
\end{align}
 the counterparts of $\Fm{m}{n}$ \eqref{eq:ecdf} and $\Gncal{(n)}$ \eqref{eq:ecdftest} in the weighted case, respectively.
Note that $\pw{i}{n}=1-\Gncal{w,(n)}(T_i)$ almost surely under Assumption~\ref{as:CP}.
Hence, the following decomposition, analogue to \eqref{eq:ProcessDecomposition}, holds: 
\begin{align}
\sqrt{\tau_{n,m}}\paren{\Fmw{m}{n}(t)-G^w(t)}=&-\sqrt{\frac{\tau_{n,m}}{m}}\sqrt{m}\paren{\Fmtest\circ{({\Gncal{w,(n)}})^{-1}}(1-t)-\Ftest\circ{({\Gncal{w,(n)}})^{-1}}(1-t)}\nonumber\\ 
&-\sqrt{\frac{\tau_{n,m}}{n}}\sqrt{n}\paren{\Ftest\circ{({\Gncal{w,(n)}})^{-1}}(1-t)-\Ftest\circ{\paren{\Wcal}^{-1}}(1-t)}.\label{eq:ProcessDecompositionWeight}
\end{align}
Thus, the novelty of \eqref{eq:ProcessDecompositionWeight} with respect to \eqref{eq:ProcessDecomposition} is only the presence of ${({\Gncal{w,(n)}})^{-1}}$ instead of ${({\Gncal{(n)}})^{-1}}$. By Assumption~\ref{as:weight}, we can show that the family of function $\mathcal{F}=\set{w_t:x\in\R\mapsto w(x)\1{x\leq t}; t\in\R}$ is $\Pcal$-Donsker and $\Pcal$-Glivenko-Cantelli. Since $\mathcal{F}$ is Glivenko-Cantelli,  the convergence of the test term happens with the same argument in the unweighted case. Because this class is Donsker, using twice the functional delta method gives us that $(\sqrt{n}[{({\Gncal{w,(n)}})^{-1}}-{(\Wcal)^{-1}}])_n$ converges in distribution to some known random process on the set $D(0,1)$. This result is stated and then proved in Lemma~\ref{lemma:DonskerWeight}. 
This leads to the joint convergence Proposition~\ref{prop:BBCoupleweight}, which is the analogue of Proposition~\ref{prop:BBCouple} in the weighted case. 
Theorem~\ref{thr:cvAlterWeight} is thus proved by applying the continuous mapping theorem to the decomposition~\eqref{eq:ProcessDecompositionWeight}.

\subsection{Proofs for Section~\ref{sec:ND}}\label{proof:nov}

\subsubsection{$\FDP$ expression}\label{sec:FDPasProcess}
 We introduced, as in the works from \cite{GW2004} and \cite{Neu2008}, $\F0$ and $\F1$ the two e.c.d.f.'s of conformal $p$-values for non-novelties and novelties, respectively:
\begin{align}
\F0(t)&\coloneqq\frac{1}{m_0(m)}\sum_{i\in\range{m}\cap\cH_{0}} \1{\p{i}{n}\leq t},\  t\in(0,1),\label{eq:ecdfNovzero}\\
\F1(t)&\coloneqq\frac{1}{m_1(m)}\sum_{i\in\range{m}\cap\cH_{1}} \1{\p{i}{n}\leq t},\  t\in(0,1).\label{eq:ecdfNovtest}
\end{align} 
We  also introduce the mixture e.c.d.f. of the test sample
\begin{align}\label{eq:ecdfmixture}
\Fm{m}{n}(t)&=\frac{1}{m}\sum_{i\in\range{m}} \1{\p{i}{n}\leq t}=\pi_0(m)\F0(t)+(1-\pi_0(m))\F1(t),\  t\in(0,1).
\end{align} 
 
For all $t\in(0,1)$, we denote simply $\FDPm(t)\coloneqq\FDPm({\{i\in\range{m},\ \p{i}{n}\leq t\}})$ (resp. $\TDPm(t)\coloneqq\TDPm({\{i\in\range{m},\ \p{i}{n}\leq t\}})$) the $\FDP$  (resp. $\TDP$) of the procedure rejecting all the conformal $p$-values smaller than $t$, see \eqref{eq:FDP} and \eqref{eq:TDP}. The following equalities hold:
\begin{equation}
\begin{split}
\FDPm(t)&=\frac{\pi_0(m)\F0(t)}{\Fm{m}{n}(t)\vee m^{-1}},\  t\in(0,1),\\
\TDPm(t)&=\F1(t),\ t\in(0,1).
\end{split}\label{FDPm}
\end{equation}
Following \cite{Neu2008}, and since the novelty detection procedure from \cite{bates2023testing} is the $\BH$ procedure applied to conformal $p$-values, we can be described it as a thresholding procedure with threshold $\T^{\BH_\alpha}(\Fm{m}{n})$, where the functional $\T^{\BH_\alpha}$ is defined by 
\begin{equation}\label{eq:functionalBH}
\T^{\BH_\alpha}(F)\coloneqq\sup\set{t\in[0,1], F(t)\geq \frac{t}{\alpha}}.
\end{equation}
In other words, we have $\mathcal{R}_\alpha=\{i\in\range{m}:\ \p{i}{n}\leq\T^{\BH_\alpha}(\Fm{m}{n})\}$ and $\FDPm(\mathcal{R}_\alpha)=\FDPm(\T^{\BH_\alpha}(\Fm{m}{n}))$ with the notation above.

\subsubsection{Joint convergence and application to the $\FDP$ and $\TDP$}

Following \cite{Neu2008}, $G\mapsto\T^{\BH_\alpha}(G)$ and $G\mapsto\FDPm(\T^{\BH_\alpha}(G))$ are Hadamard differentiable at $\Flim$ provided that $\Flim$ is concave and differentiable and that $\alpha>[\Flim'(0^+)]^{-1}$. 
The following result complete the picture by studying the convergence of $\Fm{m}{n}$.

\begin{proposition}\label{thr:conformalBridge}
Under Assumptions~\ref{as:ND} and \ref{as:strictCroissance}, assuming that $n/(n+m)\rightarrow \sigma^2\in[0,1]$ and $\pi_0(m)\rightarrow\pi_0\in(0,1)$, we have
\begin{align*}
\sqrt{\tau_{n,m}}
\begin{pmatrix}
\F0-I\\
\F1-G
\end{pmatrix}
\cvloi
\begin{pmatrix}
\frac{\sigma}{\sqrt{\pi_0}}\mathbb{U}+\sqrt{1-\sigma^2}\mathbb{W}\\
\frac{\sigma}{\sqrt{1-\pi_0}}\mathbb{V}_{G}+G'\sqrt{1-\sigma^2}\mathbb{W}
\end{pmatrix}
\eqqcolon
\begin{pmatrix}
\Z_0\\
\Z_1
\end{pmatrix}
\text{ on $\brac{D(0,1)}^2$},
\end{align*}
with $\mathbb{U},\ \mathbb{V}$ and $\mathbb{W}$ three independent Brownian bridges. As a result,
\begin{align*}
\sqrt{\tau_{n,m}}\paren{\Fm{m}{n}-\Flim}\cvloi &\sqrt{\pi_0 \sigma^2}\mathbb{U}+\sqrt{\paren{1-\pi_0}\sigma^2}\mathbb{V}_G+\Flim'\sqrt{1-\sigma^2}\mathbb{W}\\
&=\pi_0\Z_0+\paren{1-\pi_0}\Z_1\eqqcolon\Z\text{ on $D(0,1)$}.
\end{align*}
\end{proposition}

Proposition~\ref{thr:conformalBridge} is proved in Section~\ref{proof:conformalBridge}. The proof is similar to Proposition~\ref{prop:BBCouple}, with the additional technicality that the decomposition~\eqref{eq:ProcessDecomposition} should be considered for the two processes $\F0$ and $\F1$. The fact that these decompositions are based on the same process ${\Gncal{(n)}}$ induces a dependence between the components that results in the term $\mathbb{W}$ in the asymptotic variance.

Using Proposition~\ref{thr:conformalBridge} with the continuous functionals~\eqref{FDPm} gives us convergence in $D(0,1)$ of the asymptotic processes $(\FDP(t))_{t\in(0,1)}$ and $(\TDP(t))_{t\in(0,1)}$ and is left to the readers. This can be compared with the results in \cite{GW2004} where they use independent and uniform $p$-values under the null. Here we have to deals with the dependence giving us the $\W$ term and the double asymptotic giving us the $\sigma^2$ term in Proposition~\ref{thr:conformalBridge}. The same applies for the weighted case. However we are intesting to understand the asymptotic when the threshold is data-driven, and specifically the BH threshold.

By using the functional delta method theorem with the BH functionals \eqref{FDPm} and \eqref{eq:functionalBH} described above (see \citealp{Neu2008}, supplementary files from \citealp{DR2016} and Lemma S.2.2 from \citealp{kluger2024central}), we obtain the following result.

\begin{proposition}\label{thr:BH95Process}
Under Assumptions~\ref{as:ND}, \ref{as:strictCroissance} and \ref{as:concavity}, assume that BH is applied with a level $\alpha>[\Flim'(0^+)]^{-1}$. If $n/(n+m)\rightarrow \sigma^2\in[0,1]$ and $\pi_0(m)\rightarrow\pi_0\in(0,1)$, we have the following convergences:
\begin{align*}
\sqrt{\tau_{n,m}}\paren{\T^{\BH_\alpha}(\Fm{m}{n})-\T_\alpha}&\cvloi
\frac{1}{\frac{1}{\alpha}-\Flim'(\T_\alpha)}\Z({\T_\alpha});\\
\sqrt{\tau_{n,m}}\paren{\FDPm\paren{\T^{\BH_\alpha}(\Fm{m}{n})}-\pi_0\alpha\frac{}{}}&\cvloi \frac{\pi_0}{\Flim(\T_\alpha)}\Z_0({\T_\alpha});\\
\sqrt{\tau_{n,m}}\paren{\TDPm\paren{\T^{\BH_\alpha}(\Fm{m}{n})}-G(\T_\alpha)}&\cvloi \frac{G'(\T_\alpha)}{\frac{1}{\alpha}-\Flim'(\T_\alpha)}\Z({\T_\alpha})+\Z_1(\T_\alpha),
\end{align*}
where $\T_\alpha=\T^{\BH_\alpha}(\Flim)$ and $\Z_0$, $\Z_1$ and $\Z$ are the three processes defined in Proposition~\ref{thr:conformalBridge}.
\end{proposition}
Proposition~\ref{thr:BH95Process} is proved in Section~\ref{proof:BH95Process}. 
As a sanity check, we see that Proposition~\ref{thr:BH95Process} reduces to the result of \cite{Neu2008} when $n/m$ tends to infinity ($\sigma^2=1$). 

\subsubsection{Proof of Theorem~\ref{thr:BH95}}\label{proof:BH95}

Theorem~\ref{thr:BH95} is obtained from Proposition~\ref{thr:BH95Process} by computing the different asymptotic covariance functions. For this, we use that $(\Z_0,\Z_1,\Z)$ can be written as \begin{align*}
\begin{pmatrix}
\Z_0\\
\Z_1\\
\Z
\end{pmatrix}
=
\begin{pmatrix}
\frac{\sigma}{\sqrt{\pi_0}}\mathbb{U}+\sqrt{1-\sigma^2}\mathbb{W}\\
\frac{\sigma}{\sqrt{1-\pi_0}}\mathbb{V}_{G}+G'\sqrt{1-\sigma^2}\mathbb{W}\\
\sqrt{\pi_0 \sigma^2}\mathbb{U}+\sqrt{\paren{1-\pi_0}\sigma^2}\mathbb{V}_G+\Flim'\sqrt{1-\sigma^2}\mathbb{W}
\end{pmatrix},
\end{align*}
with $\U$, $\V$ and $\W$ being three independent standard Brownian bridges. The result follows from direct computations.

\subsubsection{Proof of Theorem~\ref{thr:BH95Weight}} \label{proof:thr:BH95Weight}

Proving Theorem~\ref{thr:BH95Weight} is analogue to the proof of Theorem~\ref{thr:BH95} above, but starting from weighted processes. For short, the full description of the proof is postponed to Section~\ref{sec:proof:ndweighted}.  Therein, Propositions~\ref{thr:conformalBridgeWeight} and~\ref{thr:BH95ProcessWeight} are the analogues of Propositions~\ref{thr:conformalBridge} and~\ref{thr:BH95Process}, respectively.

\section{Conclusion}

In this paper we obtain the exact asymptotic distribution of $\FCP$ and $\FDP$ for conformal inference methods when both the sizes of the calibration sample and test sample grow simultaneously. Our theory covers both the prediction and novelty detection settings, including a potential distribution shift. 
Our results quantify exactly how the covariance process is affected by the dependence inherent to the conformal  
settings, that use the same calibration sample for all the test examples. First, we prove that the convergence rate $\sqrt{\tau_{n,m}}$ can be largely deteriorated when $n/m$ vanishes to zero (that is, $\sigma=0$). Otherwise, when $n$ and $m$ are of the same order, the convergence rate is the usual one ($m^{1/2}$), but the asymptotic covariance is affected. Nevertheless, when $n/m$ tends to infinity, the convergence is strictly the same as in the usual independent case.
Interestingly, our results can be used to calibrate easily and accurately quantiles for controlling an error amount when performing conformal inference with large $n$ and $m$. 
We also quantify the effects of doing conformal inference while there is a distribution shift between the calibration sample and the test sample. We exhibit how this distribution shift acts on the asymptotic behaviour of the $\FCP$, by changing the mean and the variance, and how the correction with weighted conformal $p$-values impacts the asymptotic variance. 

While our work paves the way for studying asymptotic convergences in conformal inferences, it leaves some open directions. For instance, in case of a distribution shift, the oracle weight function is mostly unknown, and is often estimated \citep{jin2023model}. Finding the exact asymptotic distribution for the processes using estimated weights is a very interesting and challenging problem for future investigations. In a finite sample setting, using the asymptotic expectation and variance as a proxy to design weight and score functions can provide very compelling new ideas and can unlock new training procedures.

%%%%%%%%%%%%%%%%%%%%%%%%%%%%%%%%%%%%%%%%%%%%%%
%% Support information, if any,             %%
%% should be provided in the                %%
%% Acknowledgements section.                %%
%%%%%%%%%%%%%%%%%%%%%%%%%%%%%%%%%%%%%%%%%%%%%%
\begin{acks}[Acknowledgments]
I am grateful to \'Etienne Roquain for his precious advices and reviews throughout this work. I wish to thank Gilles Blanchard and Sylvain Delattre for their comments, and Magalie Fromont for fruitful discussions. I acknowledge the Emergence project MARS of Sorbonne Universit\'e,  and grant  ANR-23-CE40-0018-01 (BACKUP) of the French National Research Agency ANR.
\end{acks}
\bibliographystyle{imsart-nameyear} % Style BST file (imsart-number.bst or imsart-nameyear.bst)
\bibliography{biblio}       % Bibliography file (usually '*.bib')

\begin{thebibliography}{39}
% BibTex style file: imsart-nameyear.bst, 2017-11-03
% Default style options (sort=1,type=nameyear).
% Used options (sort=1,type=nameyear).

\bibitem[\protect\citeauthoryear{Angelopoulos and
  Bates}{2021}]{angelopoulos2021gentle}
\begin{barticle}[author]
\bauthor{\bsnm{Angelopoulos},~\bfnm{Anastasios~N}\binits{A.~N.}} \AND
  \bauthor{\bsnm{Bates},~\bfnm{Stephen}\binits{S.}}
(\byear{2021}).
\btitle{A gentle introduction to conformal prediction and distribution-free
  uncertainty quantification}.
\bjournal{arXiv preprint arXiv:2107.07511}.
\end{barticle}
\endbibitem

\bibitem[\protect\citeauthoryear{Balasubramanian, Ho and
  Vovk}{2014}]{balasubramanian2014conformal}
\begin{bbook}[author]
\bauthor{\bsnm{Balasubramanian},~\bfnm{Vineeth}\binits{V.}},
  \bauthor{\bsnm{Ho},~\bfnm{Shen-Shyang}\binits{S.-S.}} \AND
  \bauthor{\bsnm{Vovk},~\bfnm{Vladimir}\binits{V.}}
(\byear{2014}).
\btitle{Conformal prediction for reliable machine learning: theory, adaptations
  and applications}.
\bpublisher{Morgan Kaufmann books.}
\end{bbook}
\endbibitem

\bibitem[\protect\citeauthoryear{Barber et~al.}{2023}]{barber2023conformal}
\begin{barticle}[author]
\bauthor{\bsnm{Barber},~\bfnm{Rina~Foygel}\binits{R.~F.}},
  \bauthor{\bsnm{Candes},~\bfnm{Emmanuel~J}\binits{E.~J.}},
  \bauthor{\bsnm{Ramdas},~\bfnm{Aaditya}\binits{A.}} \AND
  \bauthor{\bsnm{Tibshirani},~\bfnm{Ryan~J}\binits{R.~J.}}
(\byear{2023}).
\btitle{Conformal prediction beyond exchangeability}.
\bjournal{The Annals of Statistics}
\bvolume{51}
\bpages{816--845}.
\end{barticle}
\endbibitem

\bibitem[\protect\citeauthoryear{Bates et~al.}{2023}]{bates2023testing}
\begin{barticle}[author]
\bauthor{\bsnm{Bates},~\bfnm{Stephen}\binits{S.}},
  \bauthor{\bsnm{Cand{\`e}s},~\bfnm{Emmanuel}\binits{E.}},
  \bauthor{\bsnm{Lei},~\bfnm{Lihua}\binits{L.}},
  \bauthor{\bsnm{Romano},~\bfnm{Yaniv}\binits{Y.}} \AND
  \bauthor{\bsnm{Sesia},~\bfnm{Matteo}\binits{M.}}
(\byear{2023}).
\btitle{Testing for outliers with conformal p-values}.
\bjournal{Ann. Statist.}
\bvolume{51}
\bpages{149--178}.
\end{barticle}
\endbibitem

\bibitem[\protect\citeauthoryear{Behboodi
  et~al.}{2025}]{behboodi2025fundamental}
\begin{barticle}[author]
\bauthor{\bsnm{Behboodi},~\bfnm{Arash}\binits{A.}},
  \bauthor{\bsnm{Correia},~\bfnm{Alvaro~HC}\binits{A.~H.}},
  \bauthor{\bsnm{Massoli},~\bfnm{Fabio~Valerio}\binits{F.~V.}} \AND
  \bauthor{\bsnm{Louizos},~\bfnm{Christos}\binits{C.}}
(\byear{2025}).
\btitle{Fundamental bounds on efficiency-confidence trade-off for transductive
  conformal prediction}.
\bjournal{arXiv preprint arXiv:2509.04631}.
\end{barticle}
\endbibitem

\bibitem[\protect\citeauthoryear{Benjamini and Hochberg}{1995}]{BH1995}
\begin{barticle}[author]
\bauthor{\bsnm{Benjamini},~\bfnm{Yoav}\binits{Y.}} \AND
  \bauthor{\bsnm{Hochberg},~\bfnm{Yosef}\binits{Y.}}
(\byear{1995}).
\btitle{Controlling the false discovery rate: a practical and powerful approach
  to multiple testing}.
\bjournal{J. Roy. Statist. Soc. Ser. B}
\bvolume{57}
\bpages{289--300}.
\bmrnumber{MR1325392 (96d:62143)}
\end{barticle}
\endbibitem

\bibitem[\protect\citeauthoryear{Bian and Barber}{2022}]{bian2022training}
\begin{barticle}[author]
\bauthor{\bsnm{Bian},~\bfnm{Michael}\binits{M.}} \AND
  \bauthor{\bsnm{Barber},~\bfnm{Rina~Foygel}\binits{R.~F.}}
(\byear{2022}).
\btitle{Training-conditional coverage for distribution-free predictive
  inference}.
\bjournal{arXiv preprint arXiv:2205.03647}.
\end{barticle}
\endbibitem

\bibitem[\protect\citeauthoryear{Billingsley}{1999}]{Bill1999}
\begin{bbook}[author]
\bauthor{\bsnm{Billingsley},~\bfnm{Patrick}\binits{P.}}
(\byear{1999}).
\btitle{Convergence of probability measures},
\bedition{second} ed.
\bseries{Wiley Series in Probability and Statistics: Probability and
  Statistics}.
\bpublisher{John Wiley \& Sons Inc.}, \baddress{New York}.
\bnote{A Wiley-Interscience Publication}.
\bmrnumber{MR1700749 (2000e:60008)}
\end{bbook}
\endbibitem

\bibitem[\protect\citeauthoryear{Blanchard et~al.}{2024}]{blanchard2024fdr}
\begin{barticle}[author]
\bauthor{\bsnm{Blanchard},~\bfnm{Gilles}\binits{G.}},
  \bauthor{\bsnm{Durand},~\bfnm{Guillermo}\binits{G.}},
  \bauthor{\bsnm{Marandon-Carlhian},~\bfnm{Ariane}\binits{A.}} \AND
  \bauthor{\bsnm{P{\'e}rier},~\bfnm{Romain}\binits{R.}}
(\byear{2024}).
\btitle{FDR control and FDP bounds for conformal link prediction}.
\bjournal{arXiv preprint arXiv:2404.02542}.
\end{barticle}
\endbibitem

\bibitem[\protect\citeauthoryear{Chi}{2007}]{Chi2007}
\begin{barticle}[author]
\bauthor{\bsnm{Chi},~\bfnm{Zhiyi}\binits{Z.}}
(\byear{2007}).
\btitle{On the performance of {FDR} control: constraints and a partial
  solution}.
\bjournal{Ann. Statist.}
\bvolume{35}
\bpages{1409--1431}.
\bdoi{10.1214/009053607000000037}
\bmrnumber{MR2351091}
\end{barticle}
\endbibitem

\bibitem[\protect\citeauthoryear{Delattre and Roquain}{2011}]{DR2011}
\begin{barticle}[author]
\bauthor{\bsnm{Delattre},~\bfnm{S.}\binits{S.}} \AND
  \bauthor{\bsnm{Roquain},~\bfnm{E.}\binits{E.}}
(\byear{2011}).
\btitle{On the false discovery proportion convergence under {G}aussian
  equi-correlation}.
\bjournal{Statist. Probab. Lett.}
\bvolume{81}
\bpages{111--115}.
\bdoi{10.1016/j.spl.2010.09.025}
\bmrnumber{2740072}
\end{barticle}
\endbibitem

\bibitem[\protect\citeauthoryear{Delattre and Roquain}{2016}]{DR2016}
\begin{barticle}[author]
\bauthor{\bsnm{Delattre},~\bfnm{Sylvain}\binits{S.}} \AND
  \bauthor{\bsnm{Roquain},~\bfnm{Etienne}\binits{E.}}
(\byear{2016}).
\btitle{On empirical distribution function of high-dimensional {G}aussian
  vector components with an application to multiple testing}.
\bjournal{Bernoulli}
\bvolume{22}
\bpages{302--324}.
\bdoi{10.3150/14-BEJ659}
\bmrnumber{3449784}
\end{barticle}
\endbibitem

\bibitem[\protect\citeauthoryear{Donsker}{1952}]{Don1952}
\begin{barticle}[author]
\bauthor{\bsnm{Donsker},~\bfnm{Monroe~D.}\binits{M.~D.}}
(\byear{1952}).
\btitle{Justification and extension of {D}oob's heuristic approach to the
  {K}omogorov-{S}mirnov theorems}.
\bjournal{Ann. Math. Statistics}
\bvolume{23}
\bpages{277--281}.
\bmrnumber{0047288 (13,853n)}
\end{barticle}
\endbibitem

\bibitem[\protect\citeauthoryear{Fromont, Lerasle and
  Reynaud-Bouret}{2016}]{fromont2016family}
\begin{barticle}[author]
\bauthor{\bsnm{Fromont},~\bfnm{Magalie}\binits{M.}},
  \bauthor{\bsnm{Lerasle},~\bfnm{Matthieu}\binits{M.}} \AND
  \bauthor{\bsnm{Reynaud-Bouret},~\bfnm{Patricia}\binits{P.}}
(\byear{2016}).
\btitle{Family-wise separation rates for multiple testing}.
\bjournal{The Annals of Statistics}
\bvolume{44}
\bpages{2533--2563}.
\end{barticle}
\endbibitem

\bibitem[\protect\citeauthoryear{Gazin, Blanchard and
  Roquain}{2024}]{gazin2023transductive}
\begin{binproceedings}[author]
\bauthor{\bsnm{Gazin},~\bfnm{Ulysse}\binits{U.}},
  \bauthor{\bsnm{Blanchard},~\bfnm{Gilles}\binits{G.}} \AND
  \bauthor{\bsnm{Roquain},~\bfnm{Etienne}\binits{E.}}
(\byear{2024}).
\btitle{Transductive conformal inference with adaptive scores}.
In \bbooktitle{Proceedings of The 27th International Conference on Artificial
  Intelligence and Statistics}
(\beditor{\bfnm{Sanjoy}\binits{S.}~\bsnm{Dasgupta}},
  \beditor{\bfnm{Stephan}\binits{S.}~\bsnm{Mandt}} \AND
  \beditor{\bfnm{Yingzhen}\binits{Y.}~\bsnm{Li}}, eds.).
\bseries{Proceedings of Machine Learning Research}
\bvolume{238}
\bpages{1504--1512}.
\bpublisher{PMLR}.
\end{binproceedings}
\endbibitem

\bibitem[\protect\citeauthoryear{Genovese and Wasserman}{2002}]{GW2002}
\begin{barticle}[author]
\bauthor{\bsnm{Genovese},~\bfnm{Christopher}\binits{C.}} \AND
  \bauthor{\bsnm{Wasserman},~\bfnm{Larry}\binits{L.}}
(\byear{2002}).
\btitle{Operating characteristics and extensions of the false discovery rate
  procedure}.
\bjournal{J. R. Stat. Soc. Ser. B Stat. Methodol.}
\bvolume{64}
\bpages{499--517}.
\bmrnumber{MR1924303 (2003h:62027)}
\end{barticle}
\endbibitem

\bibitem[\protect\citeauthoryear{Genovese and Wasserman}{2004}]{GW2004}
\begin{barticle}[author]
\bauthor{\bsnm{Genovese},~\bfnm{Christopher}\binits{C.}} \AND
  \bauthor{\bsnm{Wasserman},~\bfnm{Larry}\binits{L.}}
(\byear{2004}).
\btitle{A stochastic process approach to false discovery control}.
\bjournal{Ann. Statist.}
\bvolume{32}
\bpages{1035--1061}.
\bmrnumber{MR2065197 (2005k:62149)}
\end{barticle}
\endbibitem

\bibitem[\protect\citeauthoryear{Huang et~al.}{2024}]{huang2024uncertainty}
\begin{barticle}[author]
\bauthor{\bsnm{Huang},~\bfnm{Kexin}\binits{K.}},
  \bauthor{\bsnm{Jin},~\bfnm{Ying}\binits{Y.}},
  \bauthor{\bsnm{Candes},~\bfnm{Emmanuel}\binits{E.}} \AND
  \bauthor{\bsnm{Leskovec},~\bfnm{Jure}\binits{J.}}
(\byear{2024}).
\btitle{Uncertainty quantification over graph with conformalized graph neural
  networks}.
\bjournal{Advances in Neural Information Processing Systems}
\bvolume{36}.
\end{barticle}
\endbibitem

\bibitem[\protect\citeauthoryear{Jin and Cand{\`e}s}{2023}]{jin2023model}
\begin{barticle}[author]
\bauthor{\bsnm{Jin},~\bfnm{Ying}\binits{Y.}} \AND
  \bauthor{\bsnm{Cand{\`e}s},~\bfnm{Emmanuel~J}\binits{E.~J.}}
(\byear{2023}).
\btitle{Model-free selective inference under covariate shift via weighted
  conformal p-values}.
\bjournal{arXiv preprint arXiv:2307.09291}.
\end{barticle}
\endbibitem

\bibitem[\protect\citeauthoryear{Kluger and Owen}{2024}]{kluger2024central}
\begin{barticle}[author]
\bauthor{\bsnm{Kluger},~\bfnm{Dan~M}\binits{D.~M.}} \AND
  \bauthor{\bsnm{Owen},~\bfnm{Art~B}\binits{A.~B.}}
(\byear{2024}).
\btitle{A central limit theorem for the Benjamini-Hochberg false discovery
  proportion under a factor model}.
\bjournal{Bernoulli}
\bvolume{30}
\bpages{743--769}.
\end{barticle}
\endbibitem

\bibitem[\protect\citeauthoryear{Liang, Sesia and
  Sun}{2024}]{liang2024integrative}
\begin{barticle}[author]
\bauthor{\bsnm{Liang},~\bfnm{Ziyi}\binits{Z.}},
  \bauthor{\bsnm{Sesia},~\bfnm{Matteo}\binits{M.}} \AND
  \bauthor{\bsnm{Sun},~\bfnm{Wenguang}\binits{W.}}
(\byear{2024}).
\btitle{Integrative conformal p-values for out-of-distribution testing with
  labelled outliers}.
\bjournal{Journal of the Royal Statistical Society Series B: Statistical
  Methodology}
\bvolume{86}
\bpages{671--693}.
\end{barticle}
\endbibitem

\bibitem[\protect\citeauthoryear{Marandon et~al.}{2024}]{marandon2022machine}
\begin{barticle}[author]
\bauthor{\bsnm{Marandon},~\bfnm{Ariane}\binits{A.}},
  \bauthor{\bsnm{Lei},~\bfnm{Lihua}\binits{L.}},
  \bauthor{\bsnm{Mary},~\bfnm{David}\binits{D.}} \AND
  \bauthor{\bsnm{Roquain},~\bfnm{Etienne}\binits{E.}}
(\byear{2024}).
\btitle{{Adaptive novelty detection with false discovery rate guarantee}}.
\bjournal{The Annals of Statistics}
\bvolume{52}
\bpages{157 -- 183}.
\bdoi{10.1214/23-AOS2338}
\end{barticle}
\endbibitem

\bibitem[\protect\citeauthoryear{{Marques F.}}{2025}]{f2023universal}
\begin{barticle}[author]
\bauthor{\bsnm{{Marques F. }},~\bfnm{Paulo~C.}\binits{P.~C.}}
(\byear{2025}).
\btitle{Universal distribution of the empirical coverage in split conformal
  prediction}.
\bjournal{Statistics \& Probability Letters}
\bvolume{219}
\bpages{110350}.
\bdoi{https://doi.org/10.1016/j.spl.2024.110350}
\end{barticle}
\endbibitem

\bibitem[\protect\citeauthoryear{Mary and
  Roquain}{2022}]{mary2021semisupervised}
\begin{barticle}[author]
\bauthor{\bsnm{Mary},~\bfnm{David}\binits{D.}} \AND
  \bauthor{\bsnm{Roquain},~\bfnm{Etienne}\binits{E.}}
(\byear{2022}).
\btitle{{Semi-supervised multiple testing}}.
\bjournal{Electronic Journal of Statistics}
\bvolume{16}
\bpages{4926 -- 4981}.
\bdoi{10.1214/22-EJS2050}
\end{barticle}
\endbibitem

\bibitem[\protect\citeauthoryear{Neuvial}{2008}]{Neu2008}
\begin{barticle}[author]
\bauthor{\bsnm{Neuvial},~\bfnm{Pierre}\binits{P.}}
(\byear{2008}).
\btitle{Asymptotic properties of false discovery rate controlling procedures
  under independence}.
\bjournal{Electron. J. Stat.}
\bvolume{2}
\bpages{1065--1110}.
\bdoi{10.1214/08-EJS207}
\bmrnumber{MR2460858 (2010c:62154)}
\end{barticle}
\endbibitem

\bibitem[\protect\citeauthoryear{Nguyen et~al.}{2024}]{nguyen2024data}
\begin{barticle}[author]
\bauthor{\bsnm{Nguyen},~\bfnm{Drew~T}\binits{D.~T.}},
  \bauthor{\bsnm{Pathak},~\bfnm{Reese}\binits{R.}},
  \bauthor{\bsnm{Angelopoulos},~\bfnm{Anastasios~N}\binits{A.~N.}},
  \bauthor{\bsnm{Bates},~\bfnm{Stephen}\binits{S.}} \AND
  \bauthor{\bsnm{Jordan},~\bfnm{Michael~I}\binits{M.~I.}}
(\byear{2024}).
\btitle{Data-Adaptive Tradeoffs among Multiple Risks in Distribution-Free
  Prediction}.
\bjournal{arXiv preprint arXiv:2403.19605}.
\end{barticle}
\endbibitem

\bibitem[\protect\citeauthoryear{Papadopoulos
  et~al.}{2002}]{papadopoulos2002inductive}
\begin{binproceedings}[author]
\bauthor{\bsnm{Papadopoulos},~\bfnm{Harris}\binits{H.}},
  \bauthor{\bsnm{Proedrou},~\bfnm{Kostas}\binits{K.}},
  \bauthor{\bsnm{Vovk},~\bfnm{Volodya}\binits{V.}} \AND
  \bauthor{\bsnm{Gammerman},~\bfnm{Alex}\binits{A.}}
(\byear{2002}).
\btitle{Inductive confidence machines for regression}.
In \bbooktitle{13th European Conference on Machine Learning (ECML 2002)}
\bpages{345--356}.
\bpublisher{Springer}.
\end{binproceedings}
\endbibitem

\bibitem[\protect\citeauthoryear{Romano and Wolf}{2005}]{RW2005}
\begin{barticle}[author]
\bauthor{\bsnm{Romano},~\bfnm{Joseph~P.}\binits{J.~P.}} \AND
  \bauthor{\bsnm{Wolf},~\bfnm{Michael}\binits{M.}}
(\byear{2005}).
\btitle{Exact and approximate stepdown methods for multiple hypothesis
  testing}.
\bjournal{J. Amer. Statist. Assoc.}
\bvolume{100}
\bpages{94--108}.
\bmrnumber{MR2156821}
\end{barticle}
\endbibitem

\bibitem[\protect\citeauthoryear{Sarkar and Kuchibhotla}{2023}]{sarkar2023post}
\begin{barticle}[author]
\bauthor{\bsnm{Sarkar},~\bfnm{Siddhaarth}\binits{S.}} \AND
  \bauthor{\bsnm{Kuchibhotla},~\bfnm{Arun~Kumar}\binits{A.~K.}}
(\byear{2023}).
\btitle{Post-selection Inference for Conformal Prediction: Trading off Coverage
  for Precision}.
\bjournal{arXiv preprint arXiv:2304.06158}.
\end{barticle}
\endbibitem

\bibitem[\protect\citeauthoryear{Saunders, Gammerman and
  Vovk}{1999}]{saunders1999transduction}
\begin{binproceedings}[author]
\bauthor{\bsnm{Saunders},~\bfnm{Craig}\binits{C.}},
  \bauthor{\bsnm{Gammerman},~\bfnm{Alexander}\binits{A.}} \AND
  \bauthor{\bsnm{Vovk},~\bfnm{Volodya}\binits{V.}}
(\byear{1999}).
\btitle{Transduction with confidence and credibility}.
In \bbooktitle{16th International Joint Conference on Artificial Intelligence
  (IJCAI 1999)}
\bpages{722--726}.
\end{binproceedings}
\endbibitem

\bibitem[\protect\citeauthoryear{Shorack and Wellner}{1986}]{SW1986}
\begin{bbook}[author]
\bauthor{\bsnm{Shorack},~\bfnm{Galen~R.}\binits{G.~R.}} \AND
  \bauthor{\bsnm{Wellner},~\bfnm{Jon~A.}\binits{J.~A.}}
(\byear{1986}).
\btitle{Empirical processes with applications to statistics}.
\bseries{Wiley Series in Probability and Mathematical Statistics: Probability
  and Mathematical Statistics}.
\bpublisher{John Wiley \& Sons Inc.}, \baddress{New York}.
\bmrnumber{838963 (88e:60002)}
\end{bbook}
\endbibitem

\bibitem[\protect\citeauthoryear{Storey}{2002}]{Storey2002}
\begin{barticle}[author]
\bauthor{\bsnm{Storey},~\bfnm{John~D.}\binits{J.~D.}}
(\byear{2002}).
\btitle{A direct approach to false discovery rates}.
\bjournal{J. R. Stat. Soc. Ser. B Stat. Methodol.}
\bvolume{64}
\bpages{479--498}.
\bmrnumber{MR1924302 (2003h:62029)}
\end{barticle}
\endbibitem

\bibitem[\protect\citeauthoryear{Storey, Taylor and Siegmund}{2004}]{STS2004}
\begin{barticle}[author]
\bauthor{\bsnm{Storey},~\bfnm{John~D.}\binits{J.~D.}},
  \bauthor{\bsnm{Taylor},~\bfnm{Jonathan~E.}\binits{J.~E.}} \AND
  \bauthor{\bsnm{Siegmund},~\bfnm{David}\binits{D.}}
(\byear{2004}).
\btitle{Strong control, conservative point estimation and simultaneous
  conservative consistency of false discovery rates: a unified approach}.
\bjournal{J. R. Stat. Soc. Ser. B Stat. Methodol.}
\bvolume{66}
\bpages{187--205}.
\bmrnumber{MR2035766 (2004k:62056)}
\end{barticle}
\endbibitem

\bibitem[\protect\citeauthoryear{Tibshirani
  et~al.}{2019}]{tibshirani2019conformal}
\begin{barticle}[author]
\bauthor{\bsnm{Tibshirani},~\bfnm{Ryan~J}\binits{R.~J.}},
  \bauthor{\bsnm{Foygel~Barber},~\bfnm{Rina}\binits{R.}},
  \bauthor{\bsnm{Candes},~\bfnm{Emmanuel}\binits{E.}} \AND
  \bauthor{\bsnm{Ramdas},~\bfnm{Aaditya}\binits{A.}}
(\byear{2019}).
\btitle{Conformal prediction under covariate shift}.
\bjournal{Advances in neural information processing systems}
\bvolume{32}.
\end{barticle}
\endbibitem

\bibitem[\protect\citeauthoryear{van~der Vaart}{1998}]{Vaart1998}
\begin{bbook}[author]
\bauthor{\bparticle{van~der} \bsnm{Vaart},~\bfnm{A.~W.}\binits{A.~W.}}
(\byear{1998}).
\btitle{Asymptotic statistics}.
\bseries{Cambridge Series in Statistical and Probabilistic Mathematics}
\bvolume{3}.
\bpublisher{Cambridge University Press}, \baddress{Cambridge}.
\bmrnumber{MR1652247 (2000c:62003)}
\end{bbook}
\endbibitem

\bibitem[\protect\citeauthoryear{van~der Vaart and Wellner}{1996}]{VW1996}
\begin{bbook}[author]
\bauthor{\bparticle{van~der} \bsnm{Vaart},~\bfnm{Aad~W.}\binits{A.~W.}} \AND
  \bauthor{\bsnm{Wellner},~\bfnm{Jon~A.}\binits{J.~A.}}
(\byear{1996}).
\btitle{Weak convergence and empirical processes}.
\bseries{Springer Series in Statistics}.
\bpublisher{Springer-Verlag}, \baddress{New York}.
\bnote{With applications to statistics}.
\bmrnumber{MR1385671 (97g:60035)}
\end{bbook}
\endbibitem

\bibitem[\protect\citeauthoryear{Vovk}{2012}]{vovk2012conditional}
\begin{binproceedings}[author]
\bauthor{\bsnm{Vovk},~\bfnm{Vladimir}\binits{V.}}
(\byear{2012}).
\btitle{Conditional validity of inductive conformal predictors}.
In \bbooktitle{4th Asian conference on machine learning (ACML 2012)}
\bpages{475--490}.
\bpublisher{PMLR}.
\end{binproceedings}
\endbibitem

\bibitem[\protect\citeauthoryear{Vovk}{2013}]{vovk2013transductive}
\begin{binproceedings}[author]
\bauthor{\bsnm{Vovk},~\bfnm{Vladimir}\binits{V.}}
(\byear{2013}).
\btitle{Transductive conformal predictors}.
In \bbooktitle{Artificial Intelligence Applications and Innovations: 9th IFIP
  WG 12.5 International Conference (AIAI 2013)}
\bpages{348--360}.
\bpublisher{Springer}.
\end{binproceedings}
\endbibitem

\bibitem[\protect\citeauthoryear{Vovk, Gammerman and
  Shafer}{2005}]{vovk2005algorithmic}
\begin{bbook}[author]
\bauthor{\bsnm{Vovk},~\bfnm{Vladimir}\binits{V.}},
  \bauthor{\bsnm{Gammerman},~\bfnm{Alexander}\binits{A.}} \AND
  \bauthor{\bsnm{Shafer},~\bfnm{Glenn}\binits{G.}}
(\byear{2005}).
\btitle{Algorithmic learning in a random world}.
\bpublisher{Springer}.
\end{bbook}
\endbibitem

\end{thebibliography}

%% or include bibliography directly:
% \begin{thebibliography}{}
% \bibitem[\protect\citeauthoryear{???}{???}]{b1}
% \end{thebibliography}

\begin{appendix}

\section{Standardisation lemma}

In this section, we introduce the standardisation lemma, which will be extensively used in our proofs. 
Let us introduce the notation
\begin{align}
\Ftestcal(t)&=\Ftest\circ \Fcal^{-1}(t) , \:t \in [0,1]\label{eq:Compositioncdf};\\
\Fzerocal(t)&=F_0\circ \Fcal^{-1}(t) , \:t \in [0,1]\label{eq:CompositioncdfF0},
\end{align}    
where by convention $\Fcal^{-1}(0)$ denotes the infimum of the support of the distribution given by $\Fcal$ and $\Ftestcal(1)=\Fzerocal(1)=1$.
The following lemma holds.

\begin{lemma}\label{lemma:UniformCal}
Consider either the prediction setting with Assumption~\ref{as:CP} (with parameters $\Fcal$, $\Ftest$) or the novelty detection setting with  Assumption~\ref{as:ND} (with parameters $\Fcal$, $\Ftest$, $F_0$). If $\Fcal$, $\Ftest$ satisfy Assumption~\ref{as:strictCroissance}, for any weight function $w:\R\cup \{\infty\}\mapsto \R^+$ with $w(+\infty)>0$, the distribution of the $w$-weighted conformal $p$-value family under the parameters $\Fcal$, $\Ftest$ (and $F_0$) is the same as the $w\circ \Fcal^{-1}$-weighted conformal $p$-value family under the parameters $\Fcal=I$, $\Ftest=\Ftestcal$ (and $F_0=\Fzerocal$).
\end{lemma}

\begin{proof}[Proof of Lemma~\ref{lemma:UniformCal}]
Since $\Fcal$ is continuous increasing on its support, we can write almost surely
\begin{align*}
\sum_{k\in\range{n}} w(S_k)\1{ {S}_k\geq T_i}&=\sum_{k\in\range{n}} (w\circ \Fcal^{-1})(\Fcal(S_k)) \1{ \Fcal({S}_k)\geq \Fcal(T_i)}\\
&=\sum_{k\in\range{n}} (w\circ \Fcal^{-1})(S'_k) \1{ {S}'_k\geq T'_i},
\end{align*}
with $S'_k=\Fcal(S_k)$ which are i.i.d. uniform and $T'_i=\Fcal(T_i)$ which are i.i.d. $\sim \Ftestcal$ (or either $\sim \Fzerocal$ under the null or $\sim \Ftestcal$ under the alternative,  in the novelty detection setting).
\end{proof}

\section{Proofs of auxiliary results}\label{sec:ProofCLT}

In this section, we prove Proposition~\ref{prop:BBCouple}, Propositions~\ref{thr:conformalBridge} and~\ref{thr:BH95Process}. Proofs for the weighted case are given in Section~\ref{sec:proofsweight}.

\subsection{Proof of Proposition \ref{prop:BBCouple}}\label{proof:BBCouple}

First,  the Donsker theorem (Theorem~\ref{lemma:Donsker}) provides
\begin{equation*}
\sqrt{m}\paren{\Fmtest-\Ftest}\cvloi \U\circ\Ftest \text{ on $D(\R)$,}
\end{equation*}
with $\U$ being a standard Brownian bridge.
Moreover, by the Glivenko-Cantelli theorem, we have that $(\Gncal{(n)})_n$ converges in probability on $\ell^{\infty}(\R)$ to $\Fcal$. Since by Assumption~\ref{as:strictCroissance} the inverse map is continuous at $\Fcal$ we obtain that $({({\Gncal{(n)}})^{-1}})_n$ converges in probability on $D(0,1)$ to $\Fcal^{-1}$.

Second, applying again the Donsker theorem (Theorem~\ref{lemma:Donsker}), we have 
\begin{equation*}
\sqrt{n}\paren{\Gncal{(n)}-\Fcal}\cvloi \V\circ\Fcal \text{ on $D(\R)$,}
\end{equation*}
with $\V$ a standard Brownian bridge independent of $\U$. Now, by using the fonctional delta method with the inverse map (see Lemma~\ref{lemma:hadamardInverse})  by Assumption~\ref{as:strictCroissance}, we obtain,
\begin{equation}\label{eq:DonskerQuantileCal}
\sqrt{n}\paren{{({\Gncal{(n)}})^{-1}}-\Fcal^{-1}}\cvloi ({\Fcal^{-1}})'\V \text{ on $D(0,1)$,}
\end{equation}
where $({\Fcal^{-1}})'$ denotes the derivative of ${\Fcal^{-1}}$.
Again, we use the Hadamard differentiability of the map $\varphi\mapsto \Ftest\circ\varphi$ which is true by Assumption~\ref{as:strictCroissance} (see Lemma~\ref{lemma:hadamardComposition}) to obtain
\begin{equation*}
\sqrt{n}\paren{\Ftest\circ{({\Gncal{(n)}})^{-1}}-\Ftest\circ\Fcal^{-1}}\cvloi \Ftest'\circ\Fcal^{-1}({\Fcal^{-1}})'\V =\paren{\Ftest\circ\Fcal^{-1}}'\V\text{ on $D(0,1)$.}
\end{equation*}

Since for all $t\in(0,1)$, $G'(t)=\paren{\Ftest\circ\Fcal^{-1}}'(1-t)$ and since $\V$ is a standard Brownian bridge if and only if $(\V_t)_t=(\V_{1-t})_t$ is a standard Brownian bridge we obtain the second term in \eqref{eq:jointcvloi} .

Finally, since $\dcal$ and $\dtest$ are independent, we obtain the joint convergence:
\begin{equation*}
\begin{pmatrix}
\sqrt{m}\brac{\Fmtest-\Ftest}\\
\sqrt{n}\brac{\Ftest\circ{({\Gncal{(n)}})^{-1}}(1-I)-\Ftest\circ\Fcal^{-1}(1-I)}
\end{pmatrix}
\overset{\Loi}{\rightarrow}
\begin{pmatrix}
\mathbb{U}\circ \Ftest\\
G'\mathbb{V}
\end{pmatrix}
\text{ on $D(\R)\times D(0,1)$},
\end{equation*}
with $\U$ and $\V$ two independent standard Brownian bridges. By Slutsky's Lemma, we obtain the following joint convergence
\begin{align*}
\begin{pmatrix}
\sqrt{m}\brac{\Fmtest-\Ftest}\\
\sqrt{n}\brac{\Ftest\circ{({\Gncal{(n)}})^{-1}}(1-I)-\Ftest\circ\Fcal^{-1}(1-I)}\\
{({\Gncal{(n)}})^{-1}}(1-I)
\end{pmatrix}
\overset{\Loi}{\rightarrow}
\begin{pmatrix}
\mathbb{U}\circ \Ftest\\
G'\mathbb{V}\\
\Fcal^{-1}(1-I)
\end{pmatrix}\\
\text{ on $D(\R)\times D(0,1)\times \ell^{\infty}(0,1)$}.
\end{align*}
Then, using Lemma~\ref{cor:JointComposition}, we finally obtain
\begin{align*}
\begin{pmatrix}
\sqrt{m}\brac{\Fmtest\circ{({\Gncal{(n)}})^{-1}}(1-I)-\Ftest\circ{({\Gncal{(n)}})^{-1}}(1-I)}\\
\sqrt{n}\brac{\Ftest\circ{({\Gncal{(n)}})^{-1}}(1-I)-\Ftest\circ\Fcal^{-1}(1-I)}
\end{pmatrix}
\overset{\Loi}{\rightarrow}
\begin{pmatrix}
\mathbb{U}\paren{ \Ftest\circ\Fcal^{-1}(1-I)}\\
G'\mathbb{V}
\end{pmatrix}\\
\text{ on $\brac{D(0,1)}^2$}.
\end{align*}
Since $\U\paren{ \Ftest\circ\Fcal^{-1}(1-I)}$ has the same distribution as $\U(G)$, we obtain \eqref{eq:jointcvloi}.

\subsection{Proof of Proposition~\ref{thr:conformalBridge}}\label{proof:conformalBridge}

 By using Lemma~\ref{lemma:UniformCal}, one can assume without  loss of generality that $P_0=\Pcal=\mathcal{U}(0,1)$ and $\Ptest$ has for c.d.f. $\Ftestcal$. By applying the Donsker theorem (Theorem~\ref{lemma:Donsker}) with the independent families $(S_k,k\geq 1)$, $(T_i,i\in \cH_0)$ and $(T_i,i\in \cH_1)$ and following the same reasoning as in the proof of Proposition~\ref{prop:BBCouple}, there exist
 $\U$, $\V$ and $\W$ three independent standard Brownian bridges such that
\begin{equation*}
\begin{pmatrix}
\sqrt{m_0(m)}\brac{\wh{F}_{{\tiny \mbox{$m$,$0$,test}}}-I}\\
\sqrt{m_1(m)}\brac{\wh{F}_{{\tiny \mbox{$m$,$1$,test}}}-\Ftestcal}\\
\sqrt{n}\brac{{({\Gncal{(n)}})^{-1}}(1-I)-(1-I)}\\
\sqrt{n}\brac{\Ftestcal\circ{({\Gncal{(n)}})^{-1}}(1-I)-\Ftestcal(1-I)}\\
{({\Gncal{(n)}})^{-1}}(1-I)
\end{pmatrix}
\overset{\Loi}{\rightarrow}
\begin{pmatrix}
\mathbb{U}\\
\mathbb{V}(1-\Ftestcal)\\
\W\\
G'\mathbb{W}\\
1-I
\end{pmatrix}
\text{ on $\brac{D(0,1)}^4 \times\ell^\infty(0,1)$},
\end{equation*}
where we denoted $\wh{F}_{{\tiny \mbox{$m$,$r$,test}}}$ the e.c.d.f. of $\set{T_i,i\in\range{m}\cap\cH_r} $ for $r\in\set{0,1}$. Now, by using Lemma~\ref{cor:JointComposition} (or more precisely, an obvious extension of it for $5$ joint processes), we obtain

\begin{equation*}
\begin{pmatrix}
\sqrt{m_0(m)}\brac{\wh{F}_{{\tiny \mbox{$m$,$0$,test}}}\circ{({\Gncal{(n)}})^{-1}}(1-I)-{({\Gncal{(n)}})^{-1}}(1-I)}\\
\sqrt{m_1(m)}\brac{\wh{F}_{{\tiny \mbox{$m$,$1$,test}}}\circ{({\Gncal{(n)}})^{-1}}(1-I)-\Ftestcal\circ{({\Gncal{(n)}})^{-1}}(1-I)}\\
\sqrt{n}\brac{{({\Gncal{(n)}})^{-1}}(1-I)-(1-I)}\\
\sqrt{n}\brac{\Ftestcal\circ{({\Gncal{(n)}})^{-1}}(1-I)-\Ftestcal(1-I)}
\end{pmatrix}
\overset{\Loi}{\rightarrow}
\begin{pmatrix}
\mathbb{U}\\
\mathbb{V}\circ G\\
\W\\
G'\mathbb{W}
\end{pmatrix}
\text{ on $\brac{D(0,1)}^4$}.
\end{equation*}

Now, $\F0$~\eqref{eq:ecdfNovzero} and $\F1$~\eqref{eq:ecdfNovtest}  satisfy the following decomposition (obtained similarly to~\eqref{eq:ProcessDecomposition}):
\begin{align*}
\sqrt{\tau_{n,m}}\brac{\F{0}-I}=&-{\sqrt{\frac{\tau_{n,m}}{m}\times\frac{m}{m_0}}}\sqrt{m_0}\brac{\wh{F}_{{\tiny \mbox{$m$,$0$,test}}}\circ{({\Gncal{(n)}})^{-1}}(1-I)-{({\Gncal{(n)}})^{-1}}(1-I)}\nonumber\\ 
&-\sqrt{\frac{\tau_{n,m}}{n}}\sqrt{n}\brac{{({\Gncal{(n)}})^{-1}}(1-I)-(1-I))},\\
\sqrt{\tau_{n,m}}\brac{\F{1}-G}=&-{\sqrt{\frac{\tau_{n,m}}{m}\times\frac{m}{m_1}}}\sqrt{m_1}\brac{\wh{F}_{{\tiny \mbox{$m$,$1$,test}}}\circ{({\Gncal{(n)}})^{-1}}(1-I)-\Ftestcal\circ{({\Gncal{(n)}})^{-1}}(1-I)}\nonumber\\ 
&-\sqrt{\frac{\tau_{n,m}}{n}}\sqrt{n}\brac{\Ftestcal\circ{({\Gncal{(n)}})^{-1}}(1-I)-\Ftestcal(1-I)}.
\end{align*}
We conclude by using $m_0/m\rightarrow \pi_0\in(0,1)$, $m_1/m\rightarrow 1-\pi_0\in(0,1)$, $n/(n+m)\rightarrow\sigma^2\in[0,1]$ and Slutsky's lemma.

\subsection{Proof of Proposition~\ref{thr:BH95Process}}\label{proof:BH95Process}

By Assumption~\ref{as:concavity}, and since $\alpha>[\Flim'(0^+)]^{-1}$, the asymptotic threshold $\T_\alpha=\T^{\BH_\alpha}(\Flim)$ is well defined and belongs to $(0,1)$. 
Now observe that there exists a compact interval $[a,b]\subset(0,1)$ such that
\begin{equation}\label{eq:focalInterval}
\T_\alpha=\sup\set{t\in[a,b],\ \Flim(t)\geq \frac{t}{\alpha}},
\end{equation}
To see this, note that since $\Flim\leq 1$ we have that $\T_\alpha \leq \alpha$, hence $\T_\alpha$ is smaller than any $b\in(\alpha,1)$. Now, since $\T_\alpha>0$, any $a\in (0,\T_\alpha)$ leads to $\T_\alpha\in (a,b)$, which leads to \eqref{eq:focalInterval}.

Now consider the functional
\begin{equation}\label{eq:functionalBHab}
\T^{\BH_\alpha}_{[a,b]}(F)\coloneqq\sup\set{t\in[a,b], F(t)\geq \frac{t}{\alpha}},
\end{equation}
which is similar to \eqref{eq:functionalBH} but with a restricted range on $t$. Equation \eqref{eq:focalInterval} hence reads as $\T_\alpha= \T^{\BH_\alpha}_{[a,b]}(\Flim)$.
In addition, we have 
\begin{equation}\label{eq:focalIntervalProba}
\set{\abs{\Fm{m}{n}(a)-\Flim(a)}\leq\eta}\subset\set{\T^{\BH_\alpha}(\Fm{m}{n})=\T^{\BH_\alpha}_{[a,b]}(\Fm{m}{n})},
\end{equation}
where $\eta=(\Flim(a)-a\alpha^{-1})/2$, which is positive by the choice of $a$. 
Indeed, if  $|{\Fm{m}{n}(a)-\Flim(a)}|\leq\eta$, we have 
$
\Fm{m}{n}(a)-a\alpha^{-1}=\Fm{m}{n}(a)-\Flim(a)+\Flim(a)-a\alpha^{-1}\geq \eta$, 
which implies that $\{|{\Fm{m}{n}(a)-\Flim(a)}|\leq\eta\}\subset\{\T^{\BH_\alpha}(\Fm{m}{n})> a\}$. This proves \eqref{eq:focalIntervalProba}.

Now, $\T^{\BH_\alpha}_{[a,b]}$ is Hadamard differentiable at $\Flim$, tangentially to $\mathcal{C}[a,b]$ with a derivative coinciding with the one of $\T^{\BH_\alpha}$, that is, 
$${(\dot \T^{\BH_\alpha}_{[a,b]})}_{\Flim} ={(\dot \T^{\BH_\alpha})}_{\Flim}.$$
By the convergence of Proposition~\ref{thr:conformalBridge}, which holds on $D[a,b]$, we can apply the functional delta method (see Lemma~\ref{lemma:hadamardNeuvial} for the exact expression of derivatives) to  obtain 
\begin{align*}
\sqrt{\tau_{n,m}}\paren{\T^{\BH_\alpha}_{[a,b]}(\Fm{m}{n})-\T_\alpha}&\cvloi
\frac{1}{\frac{1}{\alpha}-\Flim'(\T_\alpha)}\Z({\T_\alpha});\\
\sqrt{\tau_{n,m}}\paren{\FDPm\paren{\T^{\BH_\alpha}_{[a,b]}(\Fm{m}{n})}-\pi_0\alpha}&\cvloi \frac{\pi_0}{\Flim(\T_\alpha)}\Z_0({\T_\alpha});\\
\sqrt{\tau_{n,m}}\paren{\TDPm\paren{\T^{\BH_\alpha}_{[a,b]}(\Fm{m}{n})}-G(\T_\alpha)}&\cvloi \frac{G'(\T_\alpha)}{\frac{1}{\alpha}-\Flim'(\T_\alpha)}\Z({\T_\alpha})+\Z_1(\T_\alpha),
\end{align*}
where 
$\Z_0$, $\Z_1$ and $\Z$ are the three processes defined in Proposition~\ref{thr:conformalBridge}.
Now, the same convergences hold for $\T^{\BH_\alpha}(\Fm{m}{n})$ by using \eqref{eq:focalIntervalProba} because $\Fm{m}{n}(a)$ converges in probability to $\Flim(a)$.  This concludes the proof.

\section{Proofs for the weighted case}\label{sec:proofsweight}

In this section, we state and prove Lemma~\ref{lemma:DonskerWeight}, Propositions~\ref{prop:BBCoupleweight},~\ref{thr:conformalBridgeWeight}~and~\ref{thr:BH95ProcessWeight}.

\subsection{Weighted versions of the Glivenko-Cantelli and Donsker theorem.}
Recall that 
$\Gncal{w,(n)}(t)$ is given by \eqref{eq:ecdfcalWeight}. The next result applies both in the prediction and novelty detection settings and will be used to prove Propositions~\ref{prop:BBCoupleweight},~\ref{thr:conformalBridgeWeight}~and~\ref{thr:BH95ProcessWeight}..

\begin{lemma}\label{lemma:DonskerWeight}
Assume $\Pcal=\mathcal{U}(0,1)$.
 Let $w$ a weight function satisfying Assumption~\ref{as:weight}. Then it holds
 $$
 {({\Gncal{w,(n)}})^{-1}}\overset{\P}{\rightarrow} ({\Wcal})^{-1} \text{ on $\ell^{\infty}(0,1)$,}
 $$
and
\begin{align*}
\sqrt{n}\paren{{({\Gncal{w,(n)}})^{-1}}-({\Wcal})^{-1}}\overset{\Loi}{\rightarrow}(({\Wcal})^{-1})' \rw \paren{\V \circ{{{V}^w_{{\tiny\mbox{cal}}}}\circ ({\Wcal})^{-1}}+\brac{I-{{V}^w_{{\tiny\mbox{cal}}}}\circ ({\Wcal})^{-1}}N}\\\text{ on $D(0,1)$},
\end{align*}
where $\V$ is a standard Brownian bridge and $N$ an independent standard Gaussian random variable. 
\end{lemma}

Note that in the case $w\equiv 1$ (unweighted case), we recover the convergence presented in Corollary~\ref{lemma:quantilebridge}, because ${\Wcal}={{V}^w_{{\tiny\mbox{cal}}}}=I$.

\begin{proof}[Proof of Lemma~\ref{lemma:DonskerWeight}]
Denote for $n\geq 1$ and $t\in[0,1]$, $$K^{(n)}(t)\coloneqq\frac{1}{n+1}\sum_{k=1}^n w(S_k)\1{S_k\leq t}.$$
Since $w$ is uniformly bounded, the family $\mathcal{F}=\set{w_t:x\in\R\mapsto w(x)\1{x\leq t}; t\in\R}$  is $\mathcal{U}(0,1)$-Glivenko-Cantelli, and $(K^{(n)})_n$ converges uniformly on $[0,1]$ in probability to the function $K$ given by $K(t)\coloneqq \int_{0}^{t}w(x)\dd x$ \citep{SW1986}. Then, by continuity at $K$ of the map of Lemma~\ref{lemma:HadamardWeight}, we get that $(K^{(n)}/K^{(n)}(1))_n$ converges uniformly (in probability) to $\Wcal$. 
Since $\|\Gncal{w,(n)}-K^{(n)}/K^{(n)}(1)\|_\infty\leq w(+\infty)/((n+1) K^{(n)}(1)+w(+\infty))$ tends uniformly to $0$ a.s., $(\Gncal{w,(n)})_n$ converges uniformly (in probability) to $\Wcal$. By continuity of the inverse map at $\Wcal$ (see Lemma~\ref{lemma:hadamardInverse}), we obtain that $({({\Gncal{w,(n)}})^{-1}})_n$ converges in probability to  $({\Wcal})^{-1}$  on $\ell^{\infty}(0,1)$. This proves the first statement.

Next, we turn to prove the second statement. 
Since $w$ is uniformly bounded, the family $\mathcal{F}=\set{w_t:x\in\R\mapsto w(x)\1{x\leq t}; t\in\R}$ is $\mathcal{U}(0,1)$-Donsker \citep{SW1986}.
hence there exist $\K=(\K(t))_{t\in[0,1]}$ a Gaussian process such that
$$\sqrt{n}\paren{K^{(n)}-K}\overset{\Loi}{\rightarrow} \K\ \text{in $\ell^{\infty}[0,1]$},$$
 where the distribution of $\K$ is given by $\E(\K)=0$ and for $(s,t)\in[0,1]^2$,
\begin{align*}
\Cov\paren{\K(s),\K(t)}&=\int_{0}^{s\wedge t}w^2(x)\dd x-\int_{0}^{s}w(x)\dd x\int_{0}^{t}w(x)\dd x;\\
&=K(1)^2\paren{\rw^2{{{V}^w_{{\tiny\mbox{cal}}}}}(s\wedge t)-{\Wcal}(t){\Wcal}(s)},
\end{align*}
where we used $K(1)^2\rw^2{V}^w_{{\tiny\mbox{cal}}}(t)=\int_0^tw^2(x)\dd x$ and ${\Wcal}=K/K(1)$ by the definition of ${V}^w_{{\tiny\mbox{cal}}}$~\eqref{eq:variancecdf}, $\rw$~\eqref{eq:varianceratio} and $\Wcal$~\eqref{eq:cdfcalWeight}. We easily check that 
the condition of the Kolmogorov-\v Centsov theorem is satisfied, so that $\K$ has a continuous version. 
Now, by applying the functional delta method with the map of Lemma~\ref{lemma:HadamardWeight}, we obtain
\begin{equation*}
\sqrt{n}\paren{\paren{\frac{K^{(n)}}{K^{(n)}(1)}}-{\Wcal}}\overset{\Loi}{\rightarrow} \frac{\K}{K(1)}-\frac{\K(1)}{K(1)}{\Wcal},\ \text{in $\mathcal{C}[0,1]$}.
\end{equation*}
 Since $\sqrt{n}\|\Gncal{w,(n)}-K^{(n)}/K^{(n)}(1)\|_\infty\leq\sqrt{n} w(+\infty)/((n+1)K^{(n)}(1)+w(+\infty))\rightarrow 0$ a.s., then by Slustky's lemma, we obtain
\begin{equation*}
\sqrt{n}\paren{\Gncal{w,(n)}-{\Wcal}}\overset{\Loi}{\rightarrow} \frac{\K}{K(1)}-\frac{\K(1)}{K(1)}{\Wcal},\ \text{in $\mathcal{C}[0,1]$}.
\end{equation*}
And by applying the functional delta method with the inverse map at ${\Wcal}$ (see Lemma~\ref{lemma:hadamardInverse}) we obtain, 

\begin{equation*}
\sqrt{n}\paren{\Gncal{w,(n)-1}-({\Wcal})^{-1}}\overset{\Loi}{\rightarrow} (({\Wcal})^{-1})'\paren{-\frac{\K\circ ({\Wcal})^{-1}}{K(1)}+\frac{\K(1)}{K(1)}I},\ \text{in $D(0,1)$}.
\end{equation*}

To conclude, we identify the covariance of the Gaussian limiting process. For,  $1>t\geq s>0$, 
\begin{align*}
&\Cov\paren{-\K(({\Wcal})^{-1}(t))/K(1)+\K(1)/K(1)t,-\K(({\Wcal})^{-1}(s))/K(1)+\K(1)/K(1)s}\\
&=\rw^2{{{V}^w_{{\tiny\mbox{cal}}}}}(({\Wcal})^{-1}(s))-{\Wcal}(({\Wcal})^{-1}(t)){\Wcal}(({\Wcal})^{-1}(s))\\
&\ -s\rw^2{{{V}^w_{{\tiny\mbox{cal}}}}}(({\Wcal})^{-1}(t))+s{\Wcal}(({\Wcal})^{-1}(t))\\
&\ -t\rw^2{{{V}^w_{{\tiny\mbox{cal}}}}}(({\Wcal})^{-1}(s))+t{\Wcal}(({\Wcal})^{-1}(s))\\
&\ +ts\rw^2-ts\\
&=\rw^2{{{V}^w_{{\tiny\mbox{cal}}}}}(({\Wcal})^{-1}(s))-ts-s\rw^2{{{V}^w_{{\tiny\mbox{cal}}}}}(({\Wcal})^{-1}(t))+st \\
&\ -t\rw^2{{{V}^w_{{\tiny\mbox{cal}}}}}(({\Wcal})^{-1}(s))+st+ts\rw^2-ts\\
&=\rw^2\paren{{{V}^w_{{\tiny\mbox{cal}}}}(({\Wcal})^{-1}(s))-s{{V}^w_{{\tiny\mbox{cal}}}}(({\Wcal})^{-1}(t))-t{{V}^w_{{\tiny\mbox{cal}}}}(({\Wcal})^{-1}(s))+ts}.
\end{align*}
Therefore, the last display is equal to
\begin{align*}
&\rw^2\Big({{V}^w_{{\tiny\mbox{cal}}}}(({\Wcal})^{-1}(s))-{{V}^w_{{\tiny\mbox{cal}}}}(({\Wcal})^{-1}(s)){{V}^w_{{\tiny\mbox{cal}}}}(({\Wcal})^{-1}(t))\\
&+\brac{t-{{V}^w_{{\tiny\mbox{cal}}}}(({\Wcal})^{-1}(t))}\brac{s-{{V}^w_{{\tiny\mbox{cal}}}}(({\Wcal})^{-1}(t))}\Big)\\
=&\rw^2\paren{{{V}^w_{{\tiny\mbox{cal}}}}(({\Wcal})^{-1}(s))\brac{1-{{V}^w_{{\tiny\mbox{cal}}}}(({\Wcal})^{-1}(t))}+\brac{t-{{V}^w_{{\tiny\mbox{cal}}}}(({\Wcal})^{-1}(t))}\brac{s-{{V}^w_{{\tiny\mbox{cal}}}}(({\Wcal})^{-1}(t))}}.
\end{align*}
The first term coincides with the covariance term of $\V \circ{{V}^w_{{\tiny\mbox{cal}}}}\circ  ({\Wcal})^{-1}$ for a standard Brownian bridge $\V$, 
while the second term is the covariance of the process $\paren{\brac{t-{{V}^w_{{\tiny\mbox{cal}}}}(({\Wcal})^{-1}(t))}N}_{t\in(0,1)}$ with $N\sim\Nor(0,1)$. This concludes the proof for the second statement.

\end{proof}

\subsection{Prediction setting}

The following result is the weighted version of Proposition~\ref{prop:BBCouple}.

\begin{proposition}\label{prop:BBCoupleweight}
In the prediction setting with Assumptions~\ref{as:CP}, \ref{as:strictCroissance} and \ref{as:weight} and assuming that $n/(n+m)$ tends to $\sigma^2\in[0,1]$, we have, 
\begin{align*}
\begin{pmatrix}
\sqrt{m}\brac{\Fmtest\circ{({\Gncal{w,(n)}})^{-1}}(1-I)-\Ftest\circ{({\Gncal{w,(n)}})^{-1}}(1-I)}\\
\sqrt{n}\brac{\Ftest\circ{({\Gncal{w,(n)}})^{-1}}(1-I)-\Ftest\circ {\paren{\Wcal}^{-1}}(1-I)}
\end{pmatrix}
\overset{\Loi}{\rightarrow}
\begin{pmatrix}
\mathbb{U}\paren{G^w}\\
\rw(G^w)'\brac{\mathbb{V}\circ{I^w}+\brac{I-I^w}N}
\end{pmatrix}\\ \text{ on $\brac{D(0,1)}^2$},
\end{align*}
where $\U$ and $\V$ are two independent standard  Brownian bridges and $N$ is an independent standard Gaussian random variable.
\end{proposition}

\begin{proof}
By the Donsker theorem (Theorem~\ref{lemma:Donsker}), we have 
\begin{equation*}
\sqrt{m}\paren{\Fmtest-\Ftestcal}\overset{\Loi}{\rightarrow}\U(1-\Ftestcal) \text{ on $D(0,1)$},
\end{equation*}
where $\U$ is a standard Brownian bridge.
By Lemma~\ref{lemma:DonskerWeight} (which applies because we can standardize the calibration set, see Lemma~\ref{lemma:UniformCal}), we have 
\begin{align*}
\sqrt{n}\paren{{({\Gncal{w,(n)}})^{-1}}-({\Wcal})^{-1}}\overset{\Loi}{\rightarrow}(({\Wcal})^{-1})' \rw \paren{\U \paren{{{V}^w_{{\tiny\mbox{cal}}}}\circ ({\Wcal})^{-1}}+\brac{I-{{V}^w_{{\tiny\mbox{cal}}}}\circ ({\Wcal})^{-1}}N}\\\text{ on $D(0,1)$},
\end{align*}
where $\V$ is a standard Brownian bridge and $N$ is an independent standard Gaussian random variable. Thus, by independence between $\dcal$ and $\dtest$,
\begin{align*}
\begin{pmatrix}
\sqrt{m}\brac{\Fmtest-\Ftestcal}\\
\sqrt{n}\brac{\Ftest\circ{({\Gncal{w,(n)}})^{-1}}(1-I)-\Ftest\circ {\paren{\Wcal}^{-1}}(1-I)}
\end{pmatrix}
\overset{\Loi}{\rightarrow}
\begin{pmatrix}
\mathbb{U}\paren{1-\Ftestcal}\\
\rw(G^w)'\brac{\mathbb{V}\circ{I^w}+\brac{I-I^w}N}
\end{pmatrix}\\ \text{ on $\brac{D(0,1)}^2$},
\end{align*}
where $\U$ and $\V$ two independent Brownian bridges and $N$ an independent Gaussian r.v.. By Lemma~\ref{lemma:DonskerWeight}, we have that ${({\Gncal{w,(n)})}^{-1}}\overset{\P}{\rightarrow} ({\Wcal})^{-1}$ on $\ell^{\infty}(0,1)$ hence by using Lemma~\ref{cor:JointComposition} we get 
\begin{align*}
\begin{pmatrix}
\sqrt{m}\brac{\Fmtest\circ{({\Gncal{w,(n)}})^{-1}}(1-I)-\Ftest\circ{({\Gncal{w,(n)}})^{-1}}(1-I)}\\
\sqrt{n}\brac{\Ftest\circ{({\Gncal{w,(n)}})^{-1}}(1-I)-\Ftest\circ {\paren{\Wcal}^{-1}}(1-I)}
\end{pmatrix}
\overset{\Loi}{\rightarrow}
\begin{pmatrix}
\mathbb{U}\circ\paren{1-\Ftestcal\circ ({\Wcal})^{-1}(1-I)}\\
\rw(G^w)'\brac{\mathbb{V}\circ{I^w}+\brac{I-I^w}N}
\end{pmatrix}\\ \text{ on $\brac{D(0,1)}^2$},
\end{align*}
which concludes the proof.

\end{proof}

\subsection{Novelty detection setting}\label{sec:proof:ndweighted}

 We define
\begin{align*}
\Fw0(t)&=\frac{1}{m_0(m)}\sum_{i\in\range{m}\cap\cH_{0}} \1{\pw{i}{n}\leq t},\  t\in(0,1);\\
\Fw1(t)&=\frac{1}{m-m_0(m)}\sum_{i\in\range{m}\cap\cH_{0}^{c}} \1{\pw{i}{n}\leq t},\  t\in(0,1),
\end{align*} 
the counterparts of $\F0$ \eqref{eq:ecdfNovzero}  and  $\F1$ \eqref{eq:ecdfNovtest}, respectively.
Hence, the mixture e.c.d.f.
\begin{align*}
\Fmw{m}{n}(t)&=\frac{1}{m}\sum_{i\in\range{m}} \1{\pw{i}{n}\leq t}=\pi_0(m)\Fw0(t)+(1-\pi_0(m))\Fw1(t),\  t\in(0,1),
\end{align*} 
is the counterpart of $\Fm{m}{n}$~\eqref{eq:ecdfmixture}. 

The following result is the weighted version of Proposition~\ref{thr:conformalBridge} (with an oracle weight function).

\begin{proposition}\label{thr:conformalBridgeWeight}
In the novelty detection setting with Assumption~\ref{as:ND}, assume that $P_0$ is absolutely continuous with respect to $\Pcal$ and that the oracle weight function $w^*$  \eqref{eq:oracleWeightNovelty} satisfies Assumption~\ref{as:weight}. Under Assumption~\ref{as:strictCroissance}, assuming that $n/(n+m)\rightarrow \sigma^2\in[0,1]$ and $\pi_0(m)\rightarrow\pi_0\in(0,1)$, we have
\begin{align*}
\sqrt{\tau_{n,m}}
\begin{pmatrix}
\Fwo0-I\\
\Fwo1-G^{w^*}
\end{pmatrix}
\cvloi&
\begin{pmatrix}
\frac{\sigma}{\sqrt{\pi_0}}\mathbb{U}+\rwo\sqrt{1-\sigma^2}\paren{\mathbb{W}_{I^{w^*}}+\brac{I-I^{w^*}}N}\\
\frac{\sigma}{\sqrt{1-\pi_0}}\mathbb{V}_{G^{w^*}}+{(G^{w^*})}'\rw\sqrt{1-\sigma^2}\paren{\mathbb{W}_{I^{w^*}}+\brac{I-I^{w^*}}N}
\end{pmatrix}\\
&\eqqcolon
\begin{pmatrix}
\Zwo_0\\
\Zwo_1
\end{pmatrix}
\text{ on $\brac{D(0,1)}^2$}.
\end{align*}
with $\mathbb{U},\ \mathbb{V}$ and $\mathbb{W}$ three independent Brownian bridges and an independent standard Gaussian random variable $N$. Furthermore,
\begin{align*}
\sqrt{\tau_{n,m}}\paren{\Fmwo{m}{n}-\Flimw}\cvloi &\sqrt{\pi_0 \sigma^2}\mathbb{U}+\sqrt{\paren{1-\pi_0}\sigma^2}\mathbb{V}_{G^{w^*}}\\&+({\Flimw})'\rwo\sqrt{1-\sigma^2}\paren{\mathbb{W}_{I^{w^*}}+\brac{I-I^{w^*}}N}\\
&=\pi_0\Zwo_0+\paren{1-\pi_0}\Zwo_1\eqqcolon\Zwo \text{ on $D(0,1)$}.
\end{align*}
\end{proposition}
\begin{proof}
We apply twice the argument of the proof of Proposition~\ref{prop:BBCoupleweight} to the null and the alternative processes to obtain
 $\U$, $\V$ and $\W$ three independent standard Brownian bridge and an independent standard Gaussian r.v. $N$ such that,
\begin{align*}
\begin{pmatrix}
\sqrt{m_0(m)}\paren{\wh{F}_{{\tiny \mbox{$m$,$0$,test}}}\circ{\paren{\Gncal{{w^*},(n)}}^{-1}}(1-I)-{\paren{\Gncal{{w^*},(n)}}^{-1}}(1-I)}\\
\sqrt{m_1(m)}\paren{\wh{F}_{{\tiny \mbox{$m$,$1$,test}}}\circ{\paren{\Gncal{{w^*},(n)}}^{-1}}(1-I)-\Ftest\circ{\paren{\Gncal{{w^*},(n)}}^{-1}}(1-I)}\\
\sqrt{n}\paren{\Wcalo\circ{\paren{\Gncal{{w^*},(n)}}^{-1}}(1-I)-(1-I)}\\
\sqrt{n}\paren{\Ftest\circ{\paren{\Gncal{{w^*},(n)}}^{-1}}(1-I)-\Ftest\circ{\paren{\Wcalo}^{-1}}(1-I)}
\end{pmatrix}
\end{align*}
converges when $\tau_{n,m}\rightarrow+\infty$ in distribution to
\begin{align*}
\begin{pmatrix}
\mathbb{U}\\
\mathbb{V}\circ G^{w^*}\\
\rwo\paren{\mathbb{W}\paren{I^{w^*}}+\brac{I-I^{w^*}}N}\\
\rwo(G^{w^*})'\paren{\mathbb{W}\paren{I^{w^*}}+\brac{I-I^{w^*}}N}
\end{pmatrix}
\text{ on $\brac{D(0,1)}^4$}.
\end{align*}
We conclude by using the decomposition~\eqref{eq:ProcessDecompositionWeight} on $\Fwo0$ and $\Fwo1$, the Slutsky lemma and the continuous mapping theorem.
\end{proof}

The following result is the  weighted version of Proposition~\ref{thr:BH95Process}. 

\begin{proposition}\label{thr:BH95ProcessWeight}Under Assumption~\ref{as:ND}, assume that $P_0$ is absolutely continuous with respect to $\Pcal$ and that the oracle weight function $w^*$  \eqref{eq:oracleWeightNovelty} satisfies Assumption~\ref{as:weight}. Under Assumptions~\ref{as:strictCroissance},  and \ref{as:concavityWeight}, assuming that the targeted level $\alpha>[(\Flimw)'(0^+)]^{-1}$, assuming that $n/(n+m)\rightarrow \sigma^2\in[0,1]$ and $\pi_0(m)\rightarrow\pi_0\in(0,1)$, we have,

 \begin{align*}
\sqrt{\tau_{n,m}}\paren{\T^{\BH_\alpha}(\Fmwo{m}{n})-\T^{{w^*}}_\alpha}&\cvloi
\frac{1}{\frac{1}{\alpha}-(\Flimw)'(\T^{{w^*}}_\alpha)}\Zwo({\T^{{w^*}}_\alpha});\\
\sqrt{\tau_{n,m}}\paren{\FDPmw\paren{\T^{\BH_\alpha}(\Fmwo{m}{n})}-\pi_0\alpha}&\cvloi \frac{\pi_0}{\Flimw(\T^{{w^*}}_\alpha)}\Zwo_0({\T^{{w^*}}_\alpha});\\
\sqrt{\tau_{n,m}}\paren{\TDPmw\paren{\T^{\BH_\alpha}(\Fmwo{m}{n})}-G^w(\T^{{w^*}}_\alpha)}&\cvloi \frac{{(G^{w^*})}'(\T^{{w^*}}_\alpha)}{\frac{1}{\alpha}-(\Flimw)'(\T^{{w^*}}_\alpha)}\Zwo({\T^{w}_\alpha})+\Zwo_1(\T^{w}_\alpha),
\end{align*}
where $\Zwo_0$, $\Zwo_1$ and $\Zwo{}$are the three processes defined in Proposition~\ref{thr:conformalBridgeWeight}.
\end{proposition}

\begin{proof}[Proof of Proposition~\ref{thr:BH95ProcessWeight}]
It is mutatis mutandis the same proof as in Section~\ref{proof:BH95Process} by using the weighted processes and the weighted convergence in distribution delineated in Proposition~\ref{thr:conformalBridgeWeight}.
\end{proof}

\section{Weighted novelty detection for a general weight function}\label{sec:BHnonUniform}

In this section, we extend Propositions~\ref{thr:conformalBridgeWeight}~and~\ref{thr:BH95ProcessWeight} to the case of a general weight function.
We should add the following technical assumption (which was implicitly satisfied in the oracle case):

\begin{assumption}\label{as:ContinuousdifferentiableNull}
$F_0$ is continuously differentiable.
\end{assumption}
Let us introduce 
\begin{align*}
\bp(t)&=\frac{\pi_0G_0^w(t)}{\Flimwnull(t)}\\
\zeta_\alpha(t)&=1-(1-\pi_0)\frac{({G_0^w})'G^w-{G_0^w}({G^w})'}{\paren{\alpha^{-1}-(\Flimwnull)'(t)}G_0^w}
\end{align*}
where $\bp(t)$ corresponds to the positive false discovery rate at  $t$ \citep{Storey2002}. Note that in the oracle case $\zeta_\alpha(\T^{w^*}_\alpha)=0$, but $\zeta_\alpha(\T^{w}_\alpha)$ is not necessarily zero for a general $w$.  

\begin{proposition}\label{thr:conformalBridgenonNull}
Under Assumptions~\ref{as:ND}, \ref{as:strictCroissance},  \ref{as:ContinuousdifferentiableNull}, assuming that the weight function $w$ satisfies Assumption~\ref{as:weight}, and assuming that $n/(n+m)\rightarrow \sigma^2\in[0,1]$ and $\pi_0(m)\rightarrow\pi_0\in(0,1)$, we have
\begin{align*}
\sqrt{\tau_{n,m}}
\begin{pmatrix}
\Fw0-G_0^w\\
\Fw1-G^w
\end{pmatrix}
\cvloi
\begin{pmatrix}
\frac{\sigma}{\sqrt{\pi_0}}\mathbb{U}_{G_0^w}+({G_0^w})'\rw\sqrt{1-\sigma^2}\paren{\mathbb{W}_{I^w}+\brac{I-I^w}N}\\
\frac{\sigma}{\sqrt{1-\pi_0}}\mathbb{V}_{G^w}+({G^{w}})'\rw\sqrt{1-\sigma^2}\paren{\mathbb{W}_{I^w}+\brac{I-I^w}N}
\end{pmatrix}
\eqqcolon
\begin{pmatrix}
\Zw_0\\
\Zw_1
\end{pmatrix}\\
\text{ on $\brac{D(0,1)}^2$}.
\end{align*}
with $\mathbb{U},\ \mathbb{V}$ and $\mathbb{W}$ three independent Brownian bridges and an independent standard Gaussian random variable $N$. Furthermore,
\begin{align*}
\sqrt{\tau_{n,m}}\paren{\Fmw{m}{n}-\Flimwnull}\cvloi &\sqrt{\pi_0 \sigma^2}\mathbb{U}_{G_0^w}+\sqrt{\paren{1-\pi_0}\sigma^2}\mathbb{V}_{G^w}\\
&+({\Flimwnull})'\rw\sqrt{1-\sigma^2}\paren{\mathbb{W}_{I^w}+\brac{I-I^w}N}\\
&=\pi_0\Zw_0+\paren{1-\pi_0}\Zw_1\eqqcolon\Zw \text{ on $D(0,1)$}.
\end{align*}
\end{proposition}
The proof of Proposition~\ref{thr:conformalBridgenonNull} is omitted because it is completely analogue to the one  of Proposition~\ref{thr:conformalBridgeWeight}.

\begin{theorem}\label{thr:BH95ProcessnonNull}
Under Assumptions~\ref{as:ND}, \ref{as:strictCroissance},  \ref{as:concavityWeight}, \ref{as:ContinuousdifferentiableNull},  assuming that $\alpha>[(\Flimwnull)'(0^+)]^{-1}$ and that the weight function $w$ satisfies Assumption~\ref{as:weight}. Then if $n/(n+m)\rightarrow \sigma^2\in[0,1]$ and $\pi_0(m)\rightarrow\pi_0\in(0,1)$, we have
\begin{align*}
\sqrt{\tau_{n,m}}\paren{\T^{\BH_\alpha}(\Fmw{m}{n})-\T^{w}_\alpha}&\cvloi
\frac{1}{\frac{1}{\alpha}-(\Flimw)'(\T^{w}_\alpha)}\Zw({\T^{w}_\alpha});\\
\sqrt{\tau_{n,m}}\paren{\FDPmw\paren{\cR^w_\alpha}-\bp\paren{\T^{w}_\alpha}}&\cvloi \frac{\Zw_0({\T^{w}_\alpha})}{G_0^w\paren{\T^{w}_\alpha}}\bp\paren{\T^{w}_\alpha}\brac{1-\bp\paren{\T^{w}_\alpha}\zeta_\alpha\paren{\T^{w}_\alpha}}\\
&\phantom{\cvloi}\ -\frac{\Zw_1({\T^{w}_\alpha})}{G^w\paren{\T^{w}_\alpha}}\bp\paren{\T^{w}_\alpha}\brac{1-\bp\paren{\T^{w}_\alpha}}\zeta_\alpha\paren{\T^{w}_\alpha};\\
\sqrt{\tau_{n,m}}\paren{\TDPm\paren{\cR^w_\alpha}-G^w(\T^{w}_\alpha)}&\cvloi \frac{({G^w})'(\T^{w}_\alpha)}{\frac{1}{\alpha}-(\Flimwnull)'(\T^{w}_\alpha)}\Zw({\T^{w}_\alpha})+\Zw_1(\T^{w}_\alpha),
\end{align*}
  where the three processes $(\Zw_0,\Zw_1,\Zw)$ are defined in Proposition~\ref{thr:conformalBridgenonNull}.
\end{theorem}

The proof of 
Theorem~\ref{thr:BH95ProcessnonNull} is omitted because completely analogue to the one of Proposition~\ref{thr:BH95ProcessWeight}.

\section{Additional tools for asymptotics}

\subsection{Donsker and Glivenko-Cantelli theorems}

The results below can be found in \cite{Vaart1998} and \cite{SW1986}.

\begin{theorem}[Glivenko-Cantelli]\label{lemma:glivenkoCantelli}
Let $\Ftest$ be a c.d.f., and $\Fmtest$ be an empirical version of $\Ftest$ with $m$ i.i.d. points. Then,
\begin{equation*}
\norm{\Fmtest-\Ftest}_\infty\rightarrow 0 \text{ a.s.}
\end{equation*}
\end{theorem}

\begin{proposition}\label{lemma:glivenkoCantelliQuantile}
Let $(U_k)_{k\geq 1}$ be i.i.d. random variables uniformly distributed over $(0,1)$. Denote $\mathbb{F}^{(n)}$ the e.c.d.f. the family $\paren{U_k}_{k\in\range{n}}$. Then,  
\begin{equation*}
\norm{\paren{\mathbb{F}^{(n)}}^{-1}-I}_\infty\rightarrow 0 \text{ a.s.}
\end{equation*}
\end{proposition}

\begin{theorem}[\cite{Don1952}]\label{lemma:Donsker}
Let $(X_i)_{i\geq1}$ i.i.d. real random variables. Let $F$ the cumulative distribution function of $X_1$ and $\mathbb{F}_{n}$ the empirical cdf of $\{X_1,\cdots,X_n\}$.  Then,
\begin{equation*}
\sqrt{n}\paren{\mathbb{F}_n-F}\overset{\Loi}{\rightarrow}{\mathbb{U}_F},
\end{equation*}
with $\mathbb{U}$ a standard Brownian bridge. 
\end{theorem} 

\begin{proposition}
Let $w$ a weight function satisfying Assumption~\ref{as:weight}.
Then the family of function $\mathcal{F}=\set{w_t:x\in\R\mapsto w(x)\1{x\leq t}; t\in\R}$ is $\Pcal$-Donsker and $\Pcal$-Glivenko-Cantelli.
\end{proposition}

\subsection{Hadamard differentiability}

The three first results below can be found in \cite{Vaart1998}. 

\begin{lemma}\label{lemma:hadamardComposition}
Let $F:[a,b]\rightarrow \R$ be a continuously differentiable function. The map $T:\phi\mapsto F\circ\phi$ with entries being functions $\phi:E\mapsto[a,b]$ contained in $\ell^{\infty}(E)$ is Hadamard differentiable tangentially to $\ell^\infty(E)$ with derivative:
\begin{equation*}
\dot{J}_\phi(H)=F'(\phi)H.
\end{equation*}
\end{lemma}

\begin{assumption}\label{as:DerivativePositive}
The cumulative distribution function $F$ have a compact support $[a,b]$ in the sense that that for all $x\leq a$, $F(x)=0$, for all $x\geq b$, $F(x)=1$ and for all $x\in(a,b)$, $0<F(x)<1$.
Furthermore, $F$ is continuously differentiable on $(a,b)$ with stricly positive derivative.
\end{assumption}

\begin{lemma}\label{lemma:hadamardInverse}
Let $F$ satisfying Assumption~\ref{as:DerivativePositive}. Then the inverse map $G\mapsto G^{-1}$ with domain the set of cumulative distribution function of probability measure on $(a,b]$ with value in $\ell^{\infty}(0,1)$ is Hadamard differentiable at $F$ tangentially to $\mathcal{C}[a,b]$ with derivative being the map\begin{equation*}
H\mapsto -H\circ F^{-1} {F^{-1}}'=-\frac{H\circ F^{-1}}{F'\circ F^{-1}}.
\end{equation*}
\end{lemma}

\begin{corollary}\label{lemma:quantilebridge}
Let $(X_i)_{i\geq1}$ be i.i.d. real random variables. Let $F$ be the cumulative distribution function of $X_1$ and $\mathbb{F}_{n}$ be the empirical cdf of $\{X_1,\cdots,X_n\}$. Denote $F^{-1}$ and $\mathbb{F}^{-1}_n$ the corresponding quantile functions. Assume that $F$ satisfies Assumption~\ref{as:DerivativePositive} with derivative $f$. Then,
\begin{equation*}
\sqrt{n}\paren{\mathbb{F}^{-1}_n-F^{-1}}\overset{\Loi}{\rightarrow}\frac{\mathbb{U}}{f\circ F^{-1}}=({F^{-1}})' \U,
\end{equation*}
with $\mathbb{U}$ being a standard Brownian bridge. 
\end{corollary}

\begin{lemma}\label{lemma:HadamardWeight}
 Let $w$ a weight function satisfying Assumption~\ref{as:weight}. Denote $K:t\in[0,1]\rightarrow \int_{0}^t w(x)\dd x$. The map $T:F\in D[0,1]\rightarrow F/F(1)\in D[0,1]$ is Hadamard differentiable at $K$ tangentially to $\mathcal{C}[0,1]$ with the following formula:

\begin{equation*}
\dot{T}_K(H)(u)=\frac{H}{K(1)}-\frac{H(1)}{K(1)}W,
\end{equation*}
with $W=T(K)=K/K(1)$.

\end{lemma}

\begin{proof}[Proof of Lemma \ref{lemma:HadamardWeight}]
Let $K\in D[0,1]$ defined as in the statement. Let $(H_t)_{t\in)0,1]}$ be a family of function in $\mathcal{C}[0,1]$ with $H_t\rightarrow H\in\mathcal{C}[0,1]$ uniformly on $[0,1]$.
We have, for $t$ small  enough, 
\begin{align*}
T(K+tH_t)-T(K)&=\frac{K(1)\paren{K+tH_t}-\paren{K(1)+tH_t(1)}K}{\paren{K(1)+tH_t(1)}K(1)}\\
&=\frac{K(1)tH_t-tH_t(1)K}{\paren{K(1)+tH_t(1)}K(1)}\\
&=t\times\paren{\frac{H_t}{K(1)+tH_t(1)}-\frac{H_t(1)}{K(1)+tH_t(1)}\frac{K}{K(1)}}.
\end{align*}
Hence, we obtain
\begin{align*}
\frac{T(K+tH_t)-T(K)}{t}&\underset{t\rightarrow 0}{\longrightarrow}{\frac{H}{K(1)}-\frac{H(1)}{K(1)}\frac{K}{K(1)}},
\end{align*}
which prove the Hadamard differentiability.
\end{proof}

We gather below the formulas of the Hadamard derivatives of the functionals of interest when studying the asymptotic of the $\FDP$. They are obtained from \cite{Neu2008}. 

\begin{lemma}\label{lemma:hadamardNeuvial}
Let $G:[0,1]\rightarrow [0,1]$ be a continuously differentiable increasing strictly concave function. If $\alpha>[G'(0)]^{-1}$, then $\T^{\BH_\alpha}$ is Hadamard differentiable at $G$ tangentially to $\cC[0,1]$ with the following expression for all $H\in\cC[0,1]$
\begin{equation*}
(\dot \T^{\BH_\alpha})_G(H)=\frac{H\paren{\T^{\BH_\alpha}(G)}}{\alpha^{-1}-G'\paren{\T^{\BH_\alpha}(G)}}.
\end{equation*}

Let $(G_0,G_1)$ be two continuously differentiable c.d.f. functions from $[0,1]$ to $[0,1]$ and let $\pi\in(0,1)$ such that $G=\pi G_0+(1-\pi)G_1$ is stricly concave. Let $\alpha>[G'(0)]^{-1}$. In addition, let us define 
\begin{align}
\nu&:(F_0,F_1)\rightarrow F_0(\T^{\BH_\alpha}(\pi F_0+(1-\pi)F_1))\nonumber;\\ 
\Psi&:(F_0,F_1)\rightarrow \frac{\pi F_0(\T^{\BH_\alpha}(F))}{F(\T^{\BH_\alpha}(F))};\label{eq:functionalFDP}\\
\Phi&:(F_0,F_1)\rightarrow F_1(\T^{\BH_\alpha}(F)),\label{eq:functionalTDP}
\end{align}
where $F=\pi F_0+(1-\pi)F_1$.
Then, $\nu$, $\Psi$ and $\Phi$ are  Hadamard differentiable at $(G_0,G_1)$ tangentially to $(\cC[0,1])^2$ with the following derivative expressions: for all $(H_0,H_1)\in(\cC[0,1])^2$,
\begin{align*}
\dot \nu_{G_0,G_1}(H_0,H_1)&=G_0'\paren{\T^{\BH_\alpha}(G)}\frac{H\paren{\T^{\BH_\alpha}(G)}}{\alpha^{-1}-G'\paren{\T^{\BH_\alpha}(G)}}+H_0\paren{\T^{\BH_\alpha}(G)};\\
\dot \Psi_{G_0,G_1}(H_0,H_1)&=\frac{\pi}{G\paren{\T^{\BH_\alpha}(G)}}\paren{\dot \nu_{G_0,G_1}(H_0,H_1)-\frac{G_0\paren{\T^{\BH_\alpha}(G)}}{G\paren{\T^{\BH_\alpha}(G)}}\brac{H\paren{\T^{\BH_\alpha}(G)}+G'\paren{\T^{\BH_\alpha}(G)}(\dot \T^{\BH_\alpha})_G(H) }};\\
\dot \Phi_{G_0,G_1}(H_0,H_1)&=G_1'\paren{\T^{\BH_\alpha}(G)}\frac{H\paren{\T^{\BH_\alpha}(G)}}{\alpha^{-1}-G'\paren{\T^{\BH_\alpha}(G)}}+H_1\paren{\T^{\BH_\alpha}(G)},
\end{align*}
where $H=\pi H_0+(1-\pi)H_1$.
\end{lemma}

\subsection{Random change of time}

We present here a version of the random change of time lemma of \cite{Bill1999} (page 151), which is adapted to the topological spaces $D(\R)$, $\ell^{\infty}(0,1)$ and $D(0,1)$.
\begin{lemma}\label{cor:JointComposition}
Let $(U_n,V_n)_n$ be a sequence of random processes in $D(\R)\times D(0,1)$, $\U$ and $\V$ two random processes in $D(\R)$ and $D(0,1)$, respectively, which are both a.s. continuous and such that $(U_n,V_n)\overset{\Loi}{\rightarrow}(\U,\V)$ in $D(\R)\times D(0,1)$. Let $(F_n)_n$ be a sequence of random processes in $D(0,1)$ and $F\in D(0,1)$ such that $(F_n)_n$ converges in probability to $F$ on $\ell^{\infty}(0,1)$. Assume that for all $n\in\N$, $U_n\circ F_n\in D(0,1)$. Then,
\begin{equation*}
(U_n\circ F_n,V_n)\overset{\Loi}{\rightarrow}(\U\circ F,\V)\text{ on $(D(0,1))^2$}.
\end{equation*} 
\end{lemma}
\begin{proof}[Proof of Lemma \ref{cor:JointComposition}]
Since $\U$ and $\V$ are continuous, $(U_n,V_n)_n$ converges in distribution to $(\U,\V)$ on $\ell^{\infty}(\R)\times\ell^{\infty}(0,1)$ by Lemma~\ref{lemma:equivalenceCv}.
Hence, by Slutsky's lemma, the sequence $(U_n,V_n,F_n)_n$ converges in distribution to $(\U,\V,F)$ on $\ell^{\infty}(\R)\times(\ell^{\infty}(0,1))^2$. Hence by Lemma~\ref{lemma:ContinuityCompos} and the continuous mapping theorem we obtain,
\begin{equation*}
(U_n\circ F_n,\V)\overset{\Loi}{\rightarrow}\U\circ F\text{ on $\ell^{\infty}(0,1)$}.
\end{equation*}
Since $\U$ is continuous, then $\U\circ F$ is in $D(0,1)$ and the previous convergence implies the convergence in $(D(0,1))^2$.
\end{proof}

Note that when using Lemma~\ref{cor:JointComposition} in our work, the convergence of the sequence $(U_n)_n$ is typically given by the Donsker Theorem, while the convergence of $(F_n)_n$ is given by the Glivenko-Cantelli theorem.
\begin{lemma}\label{lemma:ContinuityCompos}
Define the following map $\Psi:(F_1,F_2,F_3)\in \ell^{\infty}(\R)\times (\ell^{\infty}(0,1))^2\mapsto (F_1\circ F_3,F_2)\in (\ell^{\infty}(0,1))^2$. For all $U\in\mathcal{C}(\R)$ and $(F,V)\in (\ell^{\infty}(0,1))^2$, $\Psi$ is continuous at $(U,V,F)$.
\end{lemma}

\begin{proof}[Proof of Lemma \ref{lemma:ContinuityCompos}]
Let $U\in\mathcal{C}(\R)$ and $(V,F)\in (\ell^{\infty}(0,1))^2$. Let $(U_n,V_n,F_n)_n$ a sequence in $\ell^{\infty}(\R)\times(\ell^{\infty}(0,1))^2$ which converges to $(U,V,F)$.
Let $K\subset(0,1)$ be a compact set. Let $\varepsilon>0$. The set $\set{F_n(x),x\in K,n\in\N}$ is  bounded since $\norm{F_n-F}_{\infty,K}\rightarrow 0$ and $(F_n)_n$ and $F$ belongs in $\ell^{\infty}(0,1)$, hence is included in a compact set $K'$ of $\R$. Since $U$ is continuous, we have $\eta>0$ such that for all $(x,y)\in {K'}^{2}$, $\abs{x-y}\leq\eta$ implies that $\abs{U(x)-U(y)}\leq\varepsilon$. Let $N>0$ such that for all $n\geq N$, $\norm{U_n-U}_{\infty,K'}\leq\varepsilon$ and $\norm{F_n-F}_{\infty,K}\leq\eta$. Then for all $n\geq N$,
\begin{align*}
\norm{U_n\circ F_n-U\circ F}_{\infty,K}&\leq\norm{U_n\circ F_n-U\circ F_n}_{\infty,K}+\norm{U\circ F_n-U\circ F}_{\infty,K}\\
&\leq\norm{U_n-U}_{\infty,K'}+\varepsilon\\
&\leq 2\varepsilon,
\end{align*}
where the first $\varepsilon$ appears since for all $x\in K$, $F_n(x)\in K'$, $F(x)\in K'$ and $\abs{F_n(x)-F(x)}\leq \eta$. Furthermore for all $\tilde{K}$ a compact subset of $(0,1)$, $||V_n-V||_{\infty,\tilde{K}}\rightarrow 0$. Hence, $\Psi$ is continuous at $(U,V,F)$.
\end{proof}

Finally, the following result is classical and shows how a convergence in distribution on $D[0,1]$ can imply a uniform convergence \citep{Bill1999}, with the local topology defined in Section~\ref{sec:Topology}.
\begin{lemma}[\cite{Bill1999}]\label{lemma:equivalenceCv}
Let $(X_n)_n$ be  a sequence of random processes on $D(0,1)$. Let $X\in D(0,1)$ be a continuous function a.s.. Then, $(X_n)_n$ converges in distribution to $X$ on $D(0,1)$ if and only if $(X_n)_n$ converges in distribution to $X$ on $\ell^{\infty}(0,1)$
\end{lemma}

\section{Useful Lemmas}

\begin{lemma}\label{lemma:QuantileWeight}
Let $\paren{s_1,\cdots,s_{n+1}}\in{\R\cup\set{+\infty}}$, $(w_k)_{k\in\range{n+1}}$ be a family of positive weight summing to $1$ and $\mu=\sum_{k\in\range{n+1}}w_k\delta_{s_k}$ be a probability measure on $\R\cup\set{+\infty}$. Define for all $\alpha\in(0,1)$, $Q_\alpha(\mu)=\inf\set{x\in\R:\mu\paren{[-\infty,x]}\geq \alpha}$ the $\alpha$-quantile of the probability measure $\mu$.
Then, for all $t\in\R$ and $\alpha\in(0,1)$, the two following assertions are equivalent:
\begin{itemize}
\item[(i)] $\sum_{k\in\range{n+1}} w_k\1{s_k\geq t}\leq\alpha$
\item[(ii)] $ t>Q_{1-\alpha}(\mu)$
\end{itemize}
\end{lemma}

\begin{proof}[Proof of Lemma \ref{lemma:QuantileWeight}]
First note that, since the weights are summing to $1$,  (i) is equivalent to $ \sum_{k\in\range{n+1}} w_k\1{s_k< t}\geq 1- \alpha$.
Now prove that (ii) implies (i). If $t>Q_{1-\alpha}(\mu)$, then $\sum_{k\in\range{n+1}} w_k\1{s_k< t}\geq \sum_{k\in\range{n+1}} w_k\1{s_k\leq Q_{1-\alpha}(\mu) }\geq 1- \alpha$ and thus (i) holds.

Conversely, if (i) holds, we have $\sum_{k\in\range{n+1}} w_k\1{s_k< t}\geq 1-\alpha>0$ and thus $$S\coloneqq\max_{k\in\range{n+1}}\set{s_k: s_k<t}\in \R$$ exists. 
With such a $S<t$ we have $\sum_{k\in\range{n+1}} w_k\1{s_k\leq S}=\sum_{k\in\range{n+1}} w_k\1{s_k< t}\geq 1-\alpha$. 
This proves $S\geq Q_{1-\alpha}(\mu)$ by definition of the quantile function. Hence $t> Q_{1-\alpha}(\mu)$ and we have proved (ii).

\end{proof}

\begin{lemma}\label{lemma:taunmAsymptotic}
$\frac{\tau_{n,m}}{m}\rightarrow\sigma^2\in[0,1]$ if and only if $\frac{m}{n}\rightarrow \sigma^{-2}-1\in[0,+\infty]$.
\end{lemma}

\section{{Relaxing assumptions}}\label{sec:WeakCPlimit}

{In this section, we show that some of our results can be stated under weaker distributional assumptions. }{More precisely, we aim to transfer the assumptions on the marginal distributions of the scores (e.g. Assumption~\ref{as:strictCroissance}) to assumptions on the marginal distribution of the theoretical $p$-value $(\p{i}{+\infty})_{i\in\range{m}}$ defined in \eqref{eq:theopvalue}.}

{For this, let $S$ be a random variable with c.d.f. $\Fcal$. We denote $\Fcalcag$ its left continuous version, i.e. $\Fcalcag(t)=\P(S<t)$ for all $t\in\R$. We denote its general inverse $\Fcalcag^{-1}$ as follows:
\begin{align*}
\Fcalcag^{-1}(u)&=\sup\set{t\in\R\telque \Fcalcag(t)\leq u},\ u\in(0,1),
\end{align*}
with the convention that $\sup{\emptyset}=-\infty$ (see Lemma 13 in the supplementary material of \cite{fromont2016family} for an extensive study of $\Fcalcag^{-1}$). Let us introduce
\begin{align}
{\Ftestcalcag}(t)&=\Ftest\circ \Fcalcag^{-1}(t) , \:t \in [0,1]\label{eq:Compositioncdfcag};
\end{align}    
with the convention that $\Ftestcalcag(1)=1$ and that $\Ftest(-\infty)=0$. }

Assuming that the calibration scores $(S_k)_{k\geq 1}$ and test scores $(T_i)_{i\in\range{n}}$ are independent and that the calibration scores are i.i.d. with c.d.f. $\Fcal$ and the test scores are i.i.d. with c.d.f. $\Ftest$, then for all $i\in\N$:
\begin{equation}\label{eq:theopvalueEquality}
\p{i}{+\infty}=1-\Fcalcag(T_i).
\end{equation}
Since $\Fcalcag(T_i)$ has for c.d.f. $\Ftestcalcag$ on $[0,1]$, using the latter gives us that:
\begin{equation}\label{eq:Weakaltcdf}
\P(\p{i}{+\infty}<t)=1-\Ftestcalcag(1-t)\eqqcolon\widetilde{G}(t),\ \forall t\in(0,1)
\end{equation}  
hence making assumptions on $\widetilde{G}$~\eqref{eq:Weakaltcdf} or equivalently on $\Ftestcalcag$~\eqref{eq:Compositioncdfcag} consists on making assumptions on the marginal distribution of the theoretical $p$-value.

\subsection{{Results without weight}}

Assumptions~\ref{as:CP} and \ref{as:strictCroissance} can be relaxed as follows.
\begin{assumption}\label{as:WeakcvAlter}
The set of calibration scores $\paren{S_k}_{k\geq 1}$ and the set of test scores  $\paren{T_i}_{i\geq 1}$ are two independent families of real random variables. The $S_k, k\geq 1,$ (resp. $T_i, i\geq 1$) are i.i.d. with distribution $\Pcal$ (resp. $\Ptest$) and cumulative distribution function  $\Fcal$ (resp. $\Ftest$). The function $\widetilde{G}$ \eqref{eq:Weakaltcdf} is continuously differentiable on  $(0,1)$.
\end{assumption}

Assumption~\ref{as:WeakcvAlter} is an assumption on the data structure (as in the  beginning of Assumption~\ref{as:CP} and most importantly on the distribution of the theoretical $p$-value and not directly on the distribution of the scores as in the last part of Assumption~\ref{as:CP} and Assumption~\ref{as:strictCroissance}.  
Example~\ref{ex:Weakinterest} shows that Assumption~\ref{as:WeakcvAlter} is strictly weaker than Assumption~\ref{as:strictCroissance}.
 \begin{example}\label{ex:Weakinterest}
Assume that $\Pcal$ is the uniform distribution on $(0,1)\cup(2,3)$ and $\Ptest$ has for density $f_1(x)\1{x\in(0,1)}+f_2(x-2)\1{x\in(2,3)}$ with $f_1$, $f_2$ two continuous function from $[0,1]$ to $\R^+$ with $f_1(1)=f_2(0)$. We get for all $t\in[0,3]$ that $\Fcal(t)=1/2\paren{t\1{t\in(0,1)}+\1{t\in[1,3]}+ (t-2)\1{t\in(2,3]}}$ which is continuous nondecreasing and constant on $(1,2)$, while for all $u\in(0,1)$, $\Fcal^{-1}(u)=2u\1{u\in(0,1/2]}+(1+2u)\1{u\in(1/2,1)}$ which has a discontinuity point at $1/2$. Denoting $F_1(t)=\int_0^t f_1\dd\lambda$ and $F_2(t)=\int_0^t f_2\dd\lambda$ for all $t\in[0,1]$ we get that
\begin{equation*}
\Ftestcalcag(u)=F_1(2u)\1{u\in(0,1/2)}+\paren{F_1(1)+F_2(2u-1)}\1{u\in[1/2,1)},\ \forall u\in(0,1),
\end{equation*}
which is continuously differentiable on $(0,1)$ with derivative equals to $2f_1(2u)$ for all $u\in(0,1/2]$ and to $2f_2(2u-1)$ for all $u\in[1/2,1)$ since $f_1(1)=f_2(0)$.
\end{example}

\begin{theorem}\label{thr:WeakcvAlter}
Under Assumptions~\ref{as:WeakcvAlter} 
 and assuming that $n/(n+m)$ tends to $\sigma^2\in[0,1]$, we have
\begin{align*}
\sqrt{\tau_{n,m}}\paren{\FCPm-\wt{G}}\cvloi\sigma\mathbb{U}\circ \widetilde{G}+\sqrt{1-\sigma^2}\widetilde{G}'\mathbb{V}
\text{ on $D(0,1)$},
\end{align*}
where $\U$ and $\V$ are two independent Brownian bridges and $\widetilde{G}$ is defined in~\eqref{eq:Weakaltcdf}.
\end{theorem}
Theorem~\ref{thr:WeakcvAlter} is proved in Section~\ref{proof:WeakcvAlter} and extends Theorem~\ref{thr:cvAlter}. 
 The function $\widetilde{G}$ \eqref{eq:Weakaltcdf} is the generalisation of $G$ \eqref{eq:altcdf} when $\Fcal$ is not continuous and is equal to $G$ when $\Fcal$ is continuous.

The same kind of weakened assumption holds in the novelty detection setting. In this setting there are two different distribution for the theoretical $p$-values: the one under the null hypothesis (the test score has for c.d.f. $F_0$ under the null) and the one under the alternative (the test score has for c.d.f. $\Ftest$ under the alternative). 
 $\wt{G}$ given by \eqref{eq:Weakaltcdf} is the c.d.f. of the theoretical $p$-values under the alternative, and the c.d.f. of the theoretical $p$-value under the null is given by:
\begin{equation}\label{eq:WeakNullCDF}
\wt{G}_0(t)=1-F_0\circ\Fcalcag^{-1}(1-t),\ \forall t\in(0,1).
\end{equation} 
Since we want uniform theoretical $p$-values under the null, we will assume that $\wt{G}_0=I$. This is for example true when $\Pcal=P_0$  and $\Pcal$ is the distribution in Example~\ref{ex:Weakinterest}. We define the mixture c.d.f. ${\Flimcag}$:
\begin{equation}\label{eq:WeakmixtureLimit}
\Flimcag=\pi_0 \wt{G}_0+(1-\pi_0)\wt{G},
\end{equation}
and we state our new weakened assumption in the novelty detection setting.
\begin{assumption}
\label{as:WeakND}
The set of calibration scores $\paren{S_k}_{k\geq 1}$ and the set of test scores $\paren{T_i}_{i\geq 1}$ 
are two independent families of real random variables. The variables $S_k$, $k\geq 1$,  are i.i.d. with distribution $\Pcal$ and c.d.f. $\Fcal$.  The variables $T_i$, $i\geq 1$, are independent,  the variables $T_i$, $i\in\cH_0$, are identically distributed as a null score distribution $\Pcal$ and c.d.f. $\Fcal$, and the variables $T_i$, $i\in\cH_1$, are identically distributed as an alternative score distribution $\Ptest$ (potentially different from $P_0$) with c.d.f. $\Ftest$. Moreover, the function $\wt{G}$ \eqref{eq:Compositioncdfcag} is continuously differentiable on $(0,1)$, the function $\wt{G}_0$~\eqref{eq:WeakNullCDF} is equal to $I$ on $(0,1)$ and the function $\Flimcag$~\eqref{eq:WeakmixtureLimit} is concave.
\end{assumption}
 
Assumption~\ref{as:WeakND} is a generalisation of Assumptions~\ref{as:ND}, \ref{as:strictCroissance} and \ref{as:concavity}. More precisely, while the data share the same structure (independence and same distributions between groups), the assumptions are now on the marginal distribution of the theoretical $p$-values under the null and the alternative, and on the mixture distribution of the theoretical $p$-value in Efron's two groups model.  Example~\ref{ex:Weaktest} shows the generalisation given by Assumption~\ref{as:WeakND}. 

\begin{example}\label{ex:Weaktest}
Assume that $P_0=\Pcal$ is the uniform distribution on $(0,1)\cup(2,3)$. Let $\gamma>0$, we take $\Ptest$ a distribution with density $Z_\gamma\gamma\paren{ e^{\gamma x}\1{x\in(0,1)}+e^{\gamma (x-1)}\1{x\in(2,3)}}$ where $Z_\gamma=(e^{2\gamma}-1)^{-1}$ is a normalising constant. $\Pcal$ does not verify Assumption~\ref{as:strictCroissance}, but as shown in Example~\ref{ex:Weakinterest} $\Ftestcalcag$ is continuously differentiable. More precisely we get for all $u\in(0,1)$,
\begin{align*}
\Ftestcalcag(u)=Z_\gamma\paren{e^{2\gamma u}-1},
\end{align*}
hence for any $\pi_0\in(0,1)$, we get that $\Flimcag$ is concave. 
Furthermore, even if $\Fcal$ is constant on $(1,2)$, we still get that $\Fcal\circ\Fcalcag^{-1}(u)=u$ for all $u\in(0,1)$. So even if Assumption~\ref{as:strictCroissance} does not hold, Assumption~\ref{as:WeakND} holds hence Theorem~\ref{thr:WeakBH95} can be applied in this case.
\end{example}

\begin{theorem}\label{thr:WeakBH95}
Under Assumption~\ref{as:WeakND}, let us consider $\mathcal{R}_\alpha$ the BH procedure at level $\alpha$ applied to the conformal $p$-values \eqref{eq:IntroConfpvalues}, with a level $\alpha>[\Flimcag'(0^+)]^{-1}$. If $n/(n+m)\rightarrow \sigma^2\in[0,1]$ and $\pi_0(m)\rightarrow\pi_0\in(0,1)$, we have
\begin{align*}
\sqrt{\tau_{n,m}}\paren{\FDPm(\mathcal{R}_\alpha)-\pi_0\alpha}&\cvloi \mathcal{N}\paren{0,\alpha^2\pi_0\brac{\sigma^2+(1-\sigma^2)\pi_0}\frac{1-\T_\alpha}{\T_\alpha}};
\\
\sqrt{\tau_{n,m}}\paren{\TDPm\paren{\mathcal{R}_\alpha}-\widetilde{G}(\T^{*})}&\cvloi\mathcal{N}\paren{0,\frac{\Sigma_\alpha}{(\alpha^{-1}-\Flimcag'(\T_\alpha))^2}},
\end{align*}
with
$
\Sigma_\alpha={{\widetilde{G}'(\T_\alpha)^2\T_\alpha(1-\T_\alpha)\brac{\pi_0\sigma^2+(1-\sigma^2){\alpha^{-2}}}}}{}
+{\brac{\alpha^{-1}-\pi_0}^2\paren{1-\pi_0}^{-1}\widetilde{G}(\T_\alpha)(1-\widetilde{G}(\T_\alpha))\sigma^2}{},
$
and $\T_\alpha=\sup\set[0]{t\in(0,1),\ \Flimcag(t)\geq {t}{\alpha}^{-1}}.$
\end{theorem}

Theorem~\ref{thr:WeakBH95} is proved in Section~\ref{proof:WeakBH95} and is a generalisation of Theorem~\ref{thr:BH95}.

\subsection{Proof of Theorems~\ref{thr:WeakcvAlter} and \ref{thr:WeakcvAlter}}\label{sec:proof:weakerass}
The proofs of Theorems~\ref{thr:WeakcvAlter} and \ref{thr:WeakBH95} are consequences of the following lemma. 
\begin{lemma}\label{lemma:WeakUniformCal}
Let $S$ and $T$ be any independent real random variables. Let $w:\R\rightarrow\R^+$ be any weight function. 
Let $U\sim\mathcal{U}(0,1)$ a random variable indepent of $T$. We have the following equality in distribution:
\begin{equation*}
w(S)\1{S\geq T}\overset{\Loi}{=}\paren{w\circ\Fcalcag^{-1}}(U)\1{U\geq \Fcalcag(T)},
\end{equation*} 
and $\Fcalcag(T)$ has for c.d.f. $\Ftestcalcag$ \eqref{eq:Compositioncdfcag}.
\end{lemma}
\begin{proof}
First, we use that $S\overset{\Loi}{=}\Fcalcag^{-1}(U)$. Let $t\in\R$. We have that $\P(\Fcalcag^{-1}(U)<t)=\P(U<\Fcalcag(t))$ by Lemma 13 (3) from \cite{fromont2016family}. This implies that $\P(\Fcalcag^{-1}(U)<t)=\Fcalcag(t)$ since $\Fcalcag(t)\in[0,1]$ hence the c.d.f. of $\Fcalcag^{-1}(U)$ is the right continuous version of $\Fcalcag$: $\Fcal$. Since $S$ and $T$ are independent, and that $U$ and $T$ are also independent we have that $(S,T)\overset{\Loi}{=}(\Fcalcag^{-1}(U),T)$. Using the latter we obtain
\begin{align*}
w(S)\1{S\geq T}&\overset{\Loi}{=}w\paren{\Fcalcag^{-1}(U)}\1{\Fcalcag^{-1}(U)\geq T}\\
&=\paren{w\circ\Fcalcag^{-1}}(U)\1{U\geq \Fcalcag(T)},
\end{align*}
by using again Lemma 13 (3) from \cite{fromont2016family}.
\end{proof}
Thanks to this general lemma,  one can assume that all the calibration point are i.i.d. uniform on $(0,1)$, and to change $\Ftest$ by $\Ftestcalcag$ and $w$ by $w\circ\Fcalcag^{-1}$ (which is equal to $1$ in the unweighted case).

\subsubsection{Proof of Theorem~\ref{thr:WeakcvAlter}}\label{proof:WeakcvAlter}

Using Lemma~\ref{lemma:WeakUniformCal}, one can assume without any loss of generality that all the calibration point are i.i.d. uniform on $(0,1)$ and that all the test points are i.i.d. with c.d.f. $\Ftestcalcag$.
Since for all $t\in(0,1)$, $\Ftestcalcag(t)=1-\wt{G}(1-t)$ we have by Assumption~\ref{as:WeakcvAlter} that the data, $I$ (the new c.d.f. of the calibration points) and $\Ftestcalcag$ (the new c.d.f. of the test points) satisfy Assumption~\ref{as:CP} (since the scores are still i.i.d. within each groups and independent by Assumption~\ref{as:WeakcvAlter})  and Assumption~\ref{as:strictCroissance}. The result is then implied by Theorem~\ref{thr:cvAlter} since for all $t\in(0,1)$, $\wt{G}(t)=1-\Ftestcalcag(I^{-1}(1-t))$.

\subsubsection{Proof of Theorem~\ref{thr:WeakBH95}}\label{proof:WeakBH95}

Using Lemma~\ref{lemma:WeakUniformCal}, one can assume without any loss of generality that all the calibration point are i.i.d. uniform on $(0,1)$, that all the test points under the alternative are i.i.d. with c.d.f. $\Ftestcalcag$ and that all the test points under the null are i.i.d. with c.d.f. $\wt{G}_0$ which is equal to $I$ on $(0,1)$ by Assumption~\ref{as:WeakND}.
In these setting, we get by Assumption~\ref{as:WeakND} that Assumptions~\ref{as:ND}, \ref{as:strictCroissance} and \ref{as:concavity} are true with $\wt{G}_0=I$, i.e. the test points under the null have the same distribution that the calibration points.
The result is then implied by Theorem~\ref{thr:BH95} since for all $t\in(0,1)$, $\wt{G}(t)=1-\Ftestcalcag(I^{-1}(1-t))$ and ${\Flimcag}(t)=\pi_0 t+(1-\pi_0)\wt{G}(t)$.

\subsection{Result with weight}

When using a weight function, after applying Lemma~\ref{lemma:WeakUniformCal} one still needs to put assumptions on $w$ and $\Wcal$ (at least on a quantity wich will play a similar role) in order to use the asymptotics results necessary to the convergence of the $\FCP$, namely Lemma~\ref{lemma:DonskerWeight} and therein obtaining that the family $\wt{\mathcal{F}}=\set[0]{\wt{w}_t:x\in(0,1)\mapsto w\circ{\Fcalcag^{-1}(x)}\1{x\leq t}}$ is $\mathcal{U}(0,1)$-Glivenko-Cantelli and Donsker.
We introduce the new weighted calibration function:
\begin{equation*}\label{eq:WeakcadfcalWeight}
\Wcalcag:t\in[0,1]\mapsto \frac{\int_0^t w\circ\Fcalcag^{-1}(u)\dd u}{\int_0^1 w\circ\Fcalcag^{-1}(u)\dd u},
\end{equation*}
which corresponds to $\Wcal$~\eqref{eq:cdfcalWeight} with $w$ replaced by $w\circ\Fcalcag^{-1}$ and $\Pcal$ replaced by the uniform distribution over $(0,1)$. 
By definition, $\Wcalcag$ is differentiable on $(0,1)$ with derivative $w\circ\Fcalcag^{-1}$ which leads us to introduce the following assumption wich will replace Assumption~\ref{as:weight} in this relaxed setting. 

\begin{assumption}\label{as:Weakweight}
The weight function $w:\R\cup\set{+\infty}\rightarrow \R^+$  is measurable function.
Furthermore, the function $w\circ\Fcalcag$ is bounded and continuous on $(0,1)$ and positive on its support in the sense that there exist $0\leq a<b\leq 1$ such that for all $x\in(a,b)$, $w\circ\Fcalcag(x)>0$ and for all $x\in(0,1)\backslash{(a,b)}$, $w\circ\Fcalcag^{-1}(x)=0$.
\end{assumption}

Assumption~\ref{as:Weakweight} is still an assumption on the distribution of an asymptotic $p$-values, but it is less straightforward than Assumptions~\ref{as:WeakcvAlter} and \ref{as:WeakND}. To understand this, assume that the weight function $w$ is non-negative and bounded, and let $(S_k)_{k\geq1}$ be i.i.d. real random variables with uniform distribution on $(0,1)$ and let $S$ be an independent real random variable with uniform distribution $(0,1)$. Let $\Fcalcag$ be any left continuous with right limits c.d.f. and assume that $\E(w\circ\Fcalcag^{-1}(S_1))>0$. Then, by applying the law of large numbers we get, for any $w(+\infty)>0$,
\begin{equation*}
\frac{w(+\infty)+\sum_{k\in\range{n}}w\circ\Fcalcag^{-1}(S_k)\1{S_k\geq S}}{w(+\infty)+\sum_{k\in\range{n}}w\circ\Fcalcag^{-1}(S_k)}\rightarrow \frac{\Ec{w\circ\Fcalcag^{-1}(S_1)\1{S_1\geq S}}{S}}{\E\brac{w\circ\Fcalcag^{-1}(S_1)}}.
\end{equation*}
Since $S$ as an uniform distribution on $(0,1)$, Assumption~\ref{as:Weakweight} is an assumption on the c.d.f. of the limiting $p$-values when you test if a uniform random variable has the distibution induced by $w\circ\Fcalcag^{-1}$ on the uniform distribution on $(0,1)$. Assumption~\ref{as:Weakweight} states that the $p$-value ${\Ec{w\circ\Fcalcag^{-1}(S_1)\1{S_1\geq S}}{S}}/{\E\brac{w\circ\Fcalcag^{-1}(S_1)}}$ has a density which is positive on its support. Example~\ref{ex:WeakinterestWeight} shows the generalisation given by Assumption~\ref{as:Weakweight}. 

\begin{example}\label{ex:WeakinterestWeight}
Assume the setting of Example~\ref{ex:Weakinterest}: $\Pcal$ is the uniform distribution on $(0,1)\cup(2,3)$ and $\Ptest$ has for density $f_1(x)\1{x\in(0,1)}+f_2(x-2)\1{x\in(2,3)}$ with $f_1$, $f_2$ two continuous function from $[0,1]$ to $\R^+$ with $f_1(1)=f_2(0)$. Those two distributions satisfy Assumption~\ref{as:WeakcvAlter}.
We define the weight function $w:x\in(0,3)\mapsto w_1(x)\1{x\in(0,1)}+w_2(x-2)\1{x\in(2,3)}$ with $w_1$, $w_2$ two bounded continuous functions from $[0,1]$ to $\R^+$ with $w_1(1)=w_2(0)$ and with a couple $(a,b)\in(0,1)^2$ such that if $u\in(0,1)$ $w_1(u)=0$ if and only if $u\geq a$, and $w_2(u)=0$ if and only $u\geq b$. Denoting for all $t\in[0,1]$ $W_1(t)=\int_0^t w_1(u)\dd u$, $W_2(t)=\int_0^t w_2(u)\dd u$ and the normalising constant $Z_w=W_1(1)+W_2(1)$ we get for all $u\in(0,1)$,
\begin{equation}
\Wcalcag(u)=Z_w^{-1}\paren{W_1(2u)\1{u\in[0,1/2)}+\brac{W_1(1)+W_2(2u-1)}\1{u\in[1/2,1]}},
\end{equation}
which is differentiable with derivative $w\circ\Fcalcag^{-1}$ equals to $2Z_w^{-1}w_1(2u)$ for all $u\in(0,1/2)$ and to $2Z_w^{-1}w_2(2u-1)$ for all $u\in[1/2,1)$ which is continuous on $(0,1)$, bounded, and is positive on its support $(a/2,(b+1)/2)$ hence Assumption~\ref{as:Weakweight} is satified.
\end{example} 

Now we define the new asymptotic mean,
\begin{equation}\label{eq:WeakaltcdfWeight}
\widetilde{G^w}(t)\coloneqq 1-\Ftestcalcag\circ(\Wcalcag)^{-1}(1-t),\  t\in(0,1),
\end{equation}
the c.d.f. of the theoretical $p$-values $p^{w,(+\infty)}=1-\Wcalcag(T)$.
We introduce the new quantities involved in the variance:
\begin{align}
\Vwcag(t)&\coloneqq \frac{\int_{0}^t w(\Fcalcag^{-1}(u))^2\dd u}{\int_{0}^1 w(\Fcalcag^{-1}(u))^2\dd u},\  t\in[0,1];\label{eq:Weakvariancecdf}\\
\widetilde{I^w}(t)&\coloneqq 1-\Vwcag\circ{(\Wcalcag)^{-1}}(1-t),\  t\in(0,1)\label{eq:WeakvarianceIdentity},
\end{align}
which are the quantities replacing $V^w_{\tiny\mbox{cal}}$~\eqref{eq:variancecdf} and $I^w$~\eqref{eq:varianceIdentity} in Lemma~\ref{lemma:DonskerWeight} and Proposition~\ref{prop:BBCoupleweight} with the parameters $(\Fcal, \Ftest, w)$ equal to $(I, \Ftestcalcag, w\circ\Fcalcag^{-1})$.

\begin{theorem}\label{thr:WeakcvAlterWeight}
Under Assumptions~\ref{as:WeakcvAlter} and \ref{as:Weakweight} 
 and assuming that $n/(n+m)$ tends to $\sigma^2\in[0,1]$, we have
\begin{equation*}
\sqrt{\tau_{n,m}}\paren{\FCPm-\wt{G^w}}\cvloi\sigma\U\circ \wt{G^w}+\sqrt{1-\sigma^2}\rw \paren{\wt{G^w}}'\paren{\mathbb{V}\paren{\wt{I^w}}+\brac{I-\wt{I^w}}N} \text{ on $D(0,1)$},
\end{equation*}
where $\U$, $\V$ are two independent standard Brownian bridges and $N$ is an independent standard Gaussian random variable.
\end{theorem}
Theorem~\ref{thr:WeakcvAlterWeight} is a corrolary of  Theorem~\ref{thr:cvAlter} after using the general standardisation Lemma~\ref{lemma:WeakUniformCal}.

\end{appendix}

\end{document}